\documentclass[11pt,a4paper]{article}
\usepackage[utf8]{inputenc}

\usepackage{hyperref}

\usepackage{epsfig}
\usepackage{amsmath}
\usepackage{amsthm}
\usepackage{amssymb}
\usepackage{enumerate}
\usepackage{wrapfig}
\usepackage{subcaption}
\usepackage{multicol}
\usepackage{float}
\usepackage{color,soul}
\usepackage{listings}
\usepackage{caption}
\usepackage{stmaryrd}
\usepackage{dsfont}
\usepackage{algorithm}
\usepackage{algorithmic}
\usepackage{tikz}
\usepackage{afterpage}

\usepackage{ifthen}

\theoremstyle{plain}
\newtheorem{theorem}{Theorem}
\newtheorem{lemma}[theorem]{Lemma}
\newtheorem{corollary}[theorem]{Corollary}

\newtheorem{remark}{Remark}

\newcommand{\Reel}{\mathbb{R}}
\newcommand{\R}{\mathbb{R}}
\newcommand{\N}{\mathbb{N}}

\newcommand{\dd}{{\mathrm{d}}}

\DeclareMathOperator*{\argmin}{arg\,min}

\definecolor{applegreen}{rgb}{0.55, 0.71, 0.0}
\definecolor{green-yellow}{rgb}{0.68, 1.0, 0.18}
\definecolor{darkpastelgreen}{rgb}{0.01, 0.75, 0.24}
\definecolor{chartreuse(web)}{rgb}{0.5, 1.0, 0.0}

\newcommand{\Prob}[2][]{ \ifthenelse{\equal{#1}{}}
				    {\mathbb{P} \left( #2 \right)}
				    {\mathbb{P} \left( #2 \mid #1 \right)}
			  }

\newcommand{\Ind}[1]{\mathbf{1}_{#1}}

\newcommand{\Bigo}[1]{\mathcal{O}\left(#1\right)}

\newcommand{\cP}{\mathcal{P}}
\newcommand{\cF}{\mathcal{F}}

\newcommand{\bN}{\mathbb{N}}

\newcommand{\un}{\mathds{1}}

\makeatletter
\newcommand\xrightharpoonup[2][]{%
  \ext@arrow 9999{\hpnfill@}{#1}{#2}}
\newcommand\hpnfill@{%
  \arrowfill@\relbar\relbar\rightharpoonup}
\makeatother

\definecolor{verde}{rgb}{0.2, 0.5, 0.3}

\DeclareCaptionLabelFormat{anumber}{(a#2)}
\DeclareCaptionLabelFormat{bnumber}{(b#2)}
\DeclareCaptionLabelFormat{cnumber}{(c#2)}

\makeatletter
\renewcommand\p@subfigure{\thefigure.}
\makeatother

\title{
	Constrained overdamped Langevin dynamics for symmetric multimarginal optimal transportation
}

\date{\today}

\author{
	Aur\'elien Alfonsi \and Rafa\"el Coyaud \and Virginie Ehrlacher
}

\begin{document}

\maketitle

\begin{abstract}
  The Strictly Correlated Electrons (SCE) limit of the Levy-Lieb functional in Density Functional Theory (DFT) gives rise to a symmetric multi-marginal optimal transport problem with Coulomb cost, where the number of marginal laws is equal to the number of electrons in the system, which can be very large in relevant applications. In this work, we design a numerical method, built upon constrained overdamped Langevin processes to solve Moment Constrained Optimal Transport (MCOT) relaxations (introduced in A. Alfonsi, R. Coyaud, V. Ehrlacher and D. Lombardi, Math. Comp. 90, 2021, 689--737) of symmetric multi-marginal optimal transport problems with Coulomb cost. Some minimizers of such relaxations can be written as discrete measures charging a low number of points belonging to a space whose dimension, in the symmetrical case, scales linearly with the number of marginal laws. We leverage the sparsity of those minimizers in the design of the numerical method and prove that any local minimizer to the resulting problem is actually a \emph{global} one. We illustrate the performance of the proposed method by numerical examples which solves MCOT relaxations of 3D systems with up to 100 electrons.
\end{abstract}

\setcounter{tocdepth}{2}
\tableofcontents

\section{Introduction}

Optimal transport (OT) problems~\cite{santambrogio2015optimal, villani2008optimal} appear in numerous application fields such as data science~\cite{peyre2019computational},
finance~\cite{BeHLPe}, economics~\cite{carlier2012optimal,galichon2017survey,galichon2018optimal} or physics~\cite{villani2003topics}. Hence an increasing interest in developing efficient numerical methods for this types of problems among the applied mathematics community.

\medskip

In this article, we specifically focus on multi-marginal symmetric optimal transportation problems arising from quantum chemistry. Density Functional Theory (DFT)~\cite{parr1980density} is one of the most popular theories in quantum chemistry in order to compute the ground state of electrons within a molecule.
It is exact in principle, due to the Hohenberg-Kohn theorem, up to the knowledge of the Levy-Lieb functional, which is unfortunately not computable in practice. Hence, a wide zoology of electronic structure models have been developped in
the chemistry community where approximations of this Levy-Lieb functional are computed \cite{lewin2019universal}. Actually, it has been recently proved~\cite{MR3714366, buttazzo2018continuity, buttazzo2012optimal, cotar2013density,cotar2018smoothing,friesecke2013n, lewin2018semi}
that the semi-classical limit of this Levy-Lieb functional is the solution of a symmetric multi-marginal optimal transport problem which we state now.

\medskip

For all $p\in \mathbb{N}^*$ (where $\N^*$ denotes the set of positive integers $\{1, 2, 3, \dots \}$), we denote by $\mathcal{P}\left( \mathbb{R}^p\right)$ the set of probability measures on $\mathbb{R}^p$.
For $d\in \mathbb{N}^*$, for all $\mu\in \mathcal{P}(\mathbb{R}^d)$ and $M\in \mathbb{N}^*$ a fixed number of marginal laws (the number of electrons in DFT), we will denote the set of $M$-couplings for $\mu$ by
\begin{equation}
  \Pi(\mu; M) := \Bigg\{ \pi \in \cP\left((\R^d)^M\right) :
  \forall 1 \le m \le M, \int_{(\R^d)^{M-1}} d \pi(x_1, \dots, x_M) = d \mu(x_m) \Bigg\}.
\end{equation}
Let $c: (\R^d)^M \to \R_+ \cup \{ + \infty\}$ be a $M$-symmetric (i.e.\ such that for all $(x_1,\cdots, x_M)\in \left(\mathbb{R}^d\right)^M$, $c(x_1, \dots, x_M) = c(x_{\sigma(1)}, \dots, x_{\sigma(M)})$ for $\sigma \in \mathcal{S}_M$ a $M$-permutation) non-negative lower semi-continuous (l.s.c.)
function. The function $c$ is called hereafter the \emph{cost function}. Then, the multimarginal symmetric optimal transport problem associated to $\mu$, $M$ and $c$ is defined as
\begin{equation}\label{eqn:OT}
  I(\mu) = \inf_{\pi \in \Pi(\mu; M)} \int_{(\R^d)^M} c(x_1, \dots, x_M) d \pi(x_1, \dots, x_M).
\end{equation}

In DFT applications, the cost $c$ is defined as the Coulomb cost $c(x_1 , \dots, x_M) = \sum_{m_1 < m_2} \frac{1}{|x_{m_1} - x_{m_2}|}$. Then,
this multimarginal symmetric optimal transport problem allows to compute the interaction energy between electrons,
given an electronic density (equal to $M\mu$), in the Strictly Correlated Electrons (SCE) limit -- bringing interest in numerical methods for large multimarginal systems.

\medskip

A straightforward discretization of problem (\ref{eqn:OT}) (using a discretization of the state space $\mathbb{R}^d$ with a discrete $d$ dimensional grid for instance) leads to a linear programming problem, whose size scales exponentially with $M$. Hence,
for large values of $M$, specific numerical methods have to be used in order to circumvent the curse of dimensionality. Hence, new application or efficiency oriented approaches have been developed for such problems,
using entropic relaxation and the Sinkhorn algorithm~\cite{benamou2015iterative,benamou2016numerical}, dual formulations of the problem~\cite{mendl2013kantorovich} or sparsity structure of the minimizers of the
discrete problems~\cite{friesecke2018breaking,vogler2019kantorovich}, which can be combined with a semidefinite relaxation \cite{khoo2019semidefinite, khoo2019convex}.

\medskip

In a recent paper~\cite{alfonsi2019approximation}, the authors considered a relaxation of the optimal transport problem  (Moment Constrained Optimal Transport -- MCOT)
which boils down to considering a particular instance of Generalized Moment Problem~\cite{henrion2009gloptipoly, lasserre2008semidefinite, lasserre2010moments}. The idea of the proposed approach is to change the discretization approach in the sense the state space
$\mathbb{R}^d$ is not discretized anymore, but the marginal constraints in (\ref{eqn:OT}) are relaxed into a finite number of moment constraints.
Taking advantage of the $M$-symmetry of the problem, it was proved in~\cite[Proposition 3.3]{alfonsi2019approximation} that some minimizers of the obtained relaxed problems could be written as discrete measures charging a low number of points which scales independently of $M$.

Thus, a natural idea inspired from this result is to restrict the minimization set considered in the MCOT problem to the set of probability measures of $(\mathbb{R}^d)^M$
which can be written as discrete measures charging a low number of points and satisfying the associated moment constraints.
The resulting problem, called hereafter the \itshape particle problem\normalfont, amounts to optimize the positions of the points and the weights charging the associated Dirac measures\footnotemark. In principle, the low number of points needed to obtain a representation of a minimizer to the MCOT problem
should help in tackling the curse of dimensionality. However, the non-convexity of the particle problem remains a numerical challenge.

\footnotetext{Note that we use, in this article, the term \emph{particle} to designate a \emph{Dirac measure} (seen in the minimization problem as a vector in $\R_+ \times (\R^d)^M$ accounting for a nonnegative weight and the coordinates of a point in $(\R^d)^M$), and not with the physics meaning that encompasses electrons -- the electronic density of which, in the DFT application, would correspond in this article to $M$ times the marginal law $\mu$.}

\medskip

One of the first contribution of this paper is to prove that, despite the non-convexity of the obtained particle problem, any of its local minimizers are actually \bfseries global minimizers\normalfont. Besides, we prove that the set of local minimizers, which is hence
identical to the set of global minimizers, is polygonally connected. This first result is stated in Section~\ref{sect:MCOT} of the article.

The second contribution of the paper is to propose a numerical scheme in order to find an optimum solution to the particle problem. The numerical method builds on the use
of a constrained overdamped Langevin process projected on a submanifold defined by the constraints of the problem, in the spirit
of~\cite{ciccotti2008projection, leimkuhler2004simulating, stoltz2010free, lelievre2012langevin, lelievre2019hybrid, zhang2020ergodic}.
Such processes are actually already used in the context of molecular dynamics (for which the constraint is defined through the use of a so-called \itshape reaction coordinate \normalfont function). We give in this paper
some elements of theoretical analysis justifying the interest of such processes for the resolution of multi-marginal optimal transportation problems and outline the link between such constrained overdamped Langevin processes and entropic regularization
of optimal transport problems. This is the object of Section~\ref{sec:Langevin}. Finally, we present the numerical scheme we consider in this article in Section~\ref{sect:numericaldiscretization} and the numerical results
obtained with this approach in Section~\ref{sect:numericaltests}. Proofs of our main theoretical results are postponed until Section~\ref{sect:proofs}.

We want here to stress on the fact that this numerical scheme enabled us to obtain
approximations of solutions to (\ref{eqn:OT}) for very high-dimensional problems, for instance in cases where $d=3$ and $M = 100$. Such a method thus appears to be a very promising approach in order to solve large-scale problems in DFT
for systems involving a large number of electrons.

\medskip
Let us point out here that algorithms based on constrained overdamped Langevin dynamics can also be used in principle for the resolution of general multimarginal optimal transport problems and multimarginal martingale optimal transport problems,
as there exist an MCOT approximation for both types of problems \cite{alfonsi2019approximation}. In these cases, the number of marginal constraints to be imposed scales linearly in $M$\footnotemark, hence
the practical implementation of the numerical method proposed in this paper is more intricate than in the symmetric case studied here, where the number of constraints is independant of $M$.

\footnotetext{In the case of multimarginal martingale optimal transport, if there is no assumption of Markovian relationship between the marginal laws, the scaling in the number of constraints for the approximation of the martingale constraints may be exponential in $M$.}

%

\section{Mathematical properties of MCOT particle problems}\label{sect:MCOT}

We recall in this section the MCOT problem which was introduced in \cite{alfonsi2019approximation}, together with
the associated particle problem. We also state here our first theoretical results which describe the set of minimizers associated to the particle problem.

\subsection{MCOT and particle problems}

As introduced in~\cite{alfonsi2019approximation}, the Moment Constrained Optimal Transport (MCOT) problem is a particular case of Generalized Moment Problem~\cite{lasserre2010moments} which may be seen as a relaxation of optimal transport
where the marginal constraints are alleviated and replaced by a finite number of moment constraints. In the following, we restrain our analysis to symmetrical multimarginal optimal
transport for the sake of clarity but let us mention here that the results presented here can be extended to general multimarginal optimal transport, as well as martingale optimal transport.

\medskip

Let $d \in \N^*$, $\mu \in \cP(\R^d)$, $M \in \N^*$ and $c: (\R^d)^M \to \R_+ \cup \{ +\infty\}$ be a lower semi-continuous symmetric function. The MCOT problem is  a relaxation of the optimal transport problem \eqref{eqn:OT} which we present now.
Let $N\in \mathbb{N}^*$ and let us consider a set $(\phi_n)_{1\leq n \leq N} \subset L^1(\R^d,\mu;\R)$ of $N$ continuous real-valued functions, integrable with respect to $\mu$ and called hereafter \emph{test functions}.
For all $1\leq n \leq N$, let us denote by
\begin{equation}
  \mu_n = \int_{\R^d} \phi_n(x) \dd \mu(x),
\end{equation}
the moments of $\mu$, by
\begin{align}
  &\Pi(\mu;(\phi_n)_{1\leq n \leq N}; M) := \Bigg\{ \pi \in \cP\left((\R^d)^M\right) : \\
  & \forall 1\le n \le N,\notag \int_{(\R^d)^{M}} \sum_{m=1}^M |\phi_n(x_m)| \dd \pi(x_1, \dots, x_M) <\infty, \\
  & \int_{(\R^d)^M} \left( \frac{1}{M} \sum_{m=1}^M \phi_n(x_m)\right) \dd \pi(x_1, \dots, x_M) = \mu_n \Bigg\}, \notag
\end{align}
the set of probability measures on $(\R^d)^M$ for which the mean of the moments against the test functions of the marginal laws are equal to the one of $\mu$, and by
\begin{align}
  &\Pi^S(\mu;(\phi_n)_{1\leq n \leq N}; M) := \Bigg\{ \pi \in \cP\left((\R^d)^M\right) : \\
  & \forall 1\le n \le N,\notag \int_{(\R^d)^{M}} \sum_{m=1}^M |\phi_n(x_m)| \dd \pi(x_1, \dots, x_M) <\infty, \\
  & \forall 1 \le m \le M, \int_{(\R^d)^M} \phi_n(x_m) \dd \pi(x_1, \dots, x_M) = \mu_n \Bigg\} \notag
\end{align}
the set of probability measures on $(\R^d)^M$ that have, for each marginal law, the same moments as $\mu$ against the test functions.

\medskip

For technical reasons linked to the fact that the optimal problem is defined on the unbounded state space $\mathbb{R}^d$, we assume in addition that there exists a non-decreasing non-negative continuous function
$\theta:\R_+\to \R_+$ satisfying $\theta(r) \xrightarrow[r \to + \infty]{} + \infty$ and for which there exists $C > 0$ and $0 < s < 1$ such that for all $1 \le n \le N$ and all $x \in \R^d$,
\begin{equation}
  | \phi_n(x) | \le C (1 + \theta(|x|))^s.
\end{equation}
We finally choose a positive real number $A>0$ satisfying $A \geq A_0 := \int_{\mathbb{R}^d} \theta(|x|) d \mu(x)$.

\vspace{0.5cm}

Then, the \itshape MCOT problem \normalfont is defined by
\begin{equation}\label{eqn:MCOTs}
 I^N:= \inf_{\substack{\pi \in \Pi^S(\mu; (\phi_n)_{1\leq n \leq N}; M) \\  \frac 1 M \int_{(\R^d)^M} \sum_{m=1}^M \theta(|x_m|) d \pi(x_1, \dots, x_M) \le A}} \int_{(\R^d)^M} c(x_1, \dots, x_M) d \pi(x_1, \dots, x_M). \tag{MCOT$^S$}
\end{equation}
Under appropriate assumptions on the family of test functions $(\phi_n)_{1\leq n \leq N}$, it is proved in~\cite{alfonsi2019approximation} that the value of $I^N$ can be made arbitrarily close to $I$ as $N$, the number of
test functions, goes to infinity. Besides, converging subsequences of minimizers to (\ref{eqn:MCOTs}) necessarily converge to some minimizer of (\ref{eqn:OT}).
This is the reason why (\ref{eqn:MCOTs}) can be seen as a particular discretization approach for the numerical approximation of  Problem (\ref{eqn:OT}).

\begin{remark}
It is proved in~\cite{alfonsi2019approximation} that the value of $I^N$ does not depend on the value of $A$ provided that $A$ satisfies $A\geq A_0$.
\end{remark}

\medskip

Using the symmetry of the cost $c$ and the marginal constraints, it can be easily checked that $I^N$ is also equal to
\begin{equation}\label{eqn:MCOT}
  I^N = \inf_{\substack{\pi \in \Pi(\mu; (\phi_n)_{1\leq n \leq N}; M) \\ \frac 1 M \int_{(\R^d)^M} \sum_{m=1}^M \theta(|x_m|) d \pi(x_1, \dots, x_M) \le A}} \int_{(\R^d)^M} c(x_1, \dots, x_M) d \pi(x_1, \dots, x_M). \tag{MCOT}
\end{equation}

Then, from \cite[Proposition 3.3]{alfonsi2019approximation}, there exists at least one minimizer to problem (\ref{eqn:MCOT}), which can be written as
\begin{equation}\label{eq:pidiscrete}
\pi^N = \sum_{k = 1}^{K} w_k \delta_{\left(x^k_1,\dots,x^k_M\right)},
\end{equation}
for some $0 < K \leq N + 2$, with $w_k\geq 0$ and  $x^k_m \in \mathbb{R}^d$ for all $1\le m\le M$ and $1\leq k \leq K$. Besides, the symmetrized measure associated to $\pi^N$, which is defined by
\begin{equation}\label{eq:defpisym}
\pi_{S}^{N} := \frac{1}{M!}\sum_{\sigma \in \mathcal{S}_M}\sum_{k = 1}^{K} w_k \delta_{\left(x^k_{\sigma(1)},\dots,x^k_{\sigma(M)}\right)}
\end{equation}
where $\mathcal S_M$ is the set of permutations of $\{1, \cdots, M\}$, is a minimizer to (\ref{eqn:MCOTs}).

\medskip

The proof of this result makes use of Tchakaloff's theorem~\cite[Corollary 2]{bayer2006proof}, which is recalled in Theorem~\ref{cor:Tchakaloff} in Section~\ref{sec:Tchakaloff}.
Note that since $\Pi(\mu;(\phi_n)_{1\leq n \leq N}; M) \subset \Pi(\mu; M)$, when $I$ is finite, it naturally holds that $I^N \le I < \infty$.

\vspace{0.5cm}

These theoretical results naturally lead us to consider an optimization problem similar to (\ref{eqn:MCOT}) but where the optimization set is reduced to the set of measures of $\Pi(\mu; (\phi_n)_{1\leq n \leq N}; M)$ which can be written
as discrete measures under the form (\ref{eq:pidiscrete}) for some $K\in \mathbb{N}^*$. This naturally leads to the following optimization problem,
which we call hereafter the \itshape MCOT particle problem \normalfont with $K$ particles:
\begin{equation}\label{eqn:MCOTK}
  I^N_K := \inf_{  (W,Y) \in \mathcal U_K^N} \quad \sum_{k = 1}^K w_k c\left(X^k\right), \tag{MCOT$^K$}
\end{equation}
where
\begin{align}
  &\mathcal U^N_K := \Bigg\{ (W,Y) \in \mathbb{R}_+^K \times \left( (\mathbb{R}^d)^M \right)^K, \quad W = (w_k)_{1\leq k \leq K}, \; Y=(X^k)_{1\leq k \leq K},  \\
  & \sum_{k=1}^K w_k = 1, \quad \sum_{k=1}^K  w_k \vartheta(X^k) \le A, \quad \forall 1\le n \le N,\; \sum_{k=1}^K w_k \varphi_n(X^k) = \mu_n\Bigg\}, \notag
\end{align}
with, for all $X = (x_1, \cdots, x_M) \in (\mathbb{R}^d)^M$ and all $1\leq n \leq N$,
\begin{equation}\label{eq:Phin}
\vartheta(X) := \frac 1 M\sum_{m=1}^M\theta\left(|x_{m}|)\right) \quad \mbox{ and } \quad \varphi_n(X) := \frac{1}{M}\sum_{m=1}^M  \phi_n(x_{m}).
\end{equation}

In view of \cite[Proposition 3.3]{alfonsi2019approximation}, we have $I^N_K = I^N$ as soon as $K\geq N+2$.

\medskip

A few remarks are in order at this point.
\begin{remark}
  \begin{enumerate}[(i)]
    \item Considering problem~\ref{eqn:MCOTK} as a starting point for a numerical scheme seems very appealing, especially in contexts when $M$ is large. Indeed, in principle, the resolution of (\ref{eqn:MCOTK}) only requires the optimization of at most
    $K (M+1)$ scalars, thus would require the resolution of an optimization problem defined on a continuous optimization set involving a number of parameters which only scales \bfseries linearly \normalfont\itshape
    with respect to the number of marginal laws. Thus, gradient-based algorithms are natural to consider for the numerical resolution of (\ref{eqn:MCOTK}), at least for differentiable test functions.
    \item Problem~\ref{eqn:MCOTK} is highly non-convex, whereas the original MCOT problem~(\ref{eqn:MCOT}) reads as a (high-dimensional) linear problem\footnotemark. This definitely makes the numerical resolution of (\ref{eqn:MCOTK}) a challenging task. This is the reason why we consider
    in this article \bfseries randomized versions of gradient-based \normalfont \itshape algorithms for the resolution of (\ref{eqn:MCOTK}).
    Nevertheless, strikingly, we prove in this article that, despite the lack of convexity, any local minimizers to the MCOT particle problem~(\ref{eqn:MCOT}) are actually \bfseries global \normalfont\itshape minimizers, provided that $K\geq 2N+6$.
    This is the object of Section~\ref{sec:minimizers} to state this result and further mathematical properties of the set of minimizers to (\ref{eqn:MCOTK}).
 \end{enumerate}
\end{remark}

\footnotetext{More generally, any non-linear minimization problem can be reframed as a linear minimization problem in a much larger space (the measure space), as $\min_{x \in \R^d} c(x) = \min_{\mathbb{P}} \int_{\R^d} c(y) \dd \mathbb{P}(y)$.}

\medskip

The main focus of this article is to propose numerical schemes relying on stochastic versions of gradient-based algorithms in order to find minimizers to the MCOT particle problem. Such numerical schemes actually make use of \itshape constrained
overdamped Langevin processes\normalfont, which are usually encountered in the context of molecular dynamics simulations~\cite{stoltz2010free, lelievre2012langevin}. In Section~\ref{sec:Langevin}, we relate such stochastic processes
with MCOT problems and entropic regularizations of the latter.

\medskip


In numerical tests, and especially in the 3D case, the schemes proposed in this article perform better when using a large number of particles $K$, with weights $w_k$ assumed to be fixed and equal to $\frac{1}{K}$ which are not optimized upon. That is why we introduce here the resulting optimization, called the \itshape MCOT fixed-weight particle problem \normalfont with $K$ particles, which reads as follows:
\begin{equation}\label{eqn:MCOTKfiwedweight}
  J^N_K := \inf_{\substack{ Y:=\left( X^k\right)_{1\leq  k \leq K} \in ((\R^d)^M)^K,\\
    \forall 1 \le n \le N,\, \frac{1}{K}\sum_{k=1}^K  \varphi_n(X^k) = \mu_n, \\
     \frac{1}{K}\sum_{k=1}^K\vartheta(X^k) \le A}} \quad \sum_{k = 1}^K \frac{1}{K} c\left(X^k\right). \tag{MCOT$^K$ -fixed weight}
\end{equation}

\begin{remark}
\begin{enumerate}[(i)]
\item Let us stress on the fact that the existence of a solution to (\ref{eqn:MCOTKfiwedweight}) is not guaranteed in general. This stems from the fact that there may not exist a set of points
$Y= \left( X^k\right)_{1\leq  k \leq K}$ satisfying the constraints of problem (\ref{eqn:MCOTKfiwedweight}). However, for all $N,K\in \mathbb{N}^*$, it always holds that $J^N_K \geq I^N_K$.

Let however consider $(W,Y)\in \mathcal{U}^N_{N+2}$ a minimizer of~\eqref{eqn:MCOTK} and assume that the cost $c$ and the test functions $\phi_n$ are bounded. Then, by rounding the weights $w_k$ to a multiple of $1/K$, and then by using $\ell$ copies of particles with weight $\ell/K$, we can construct $\tilde{Y}=\left( \tilde{X}^k\right)_{1\leq  k \leq K}$ such that
$$
\frac{1}{K}\sum_{k=1}^K \varphi_n(\tilde{X}^k)\approx \mu_n + \mathcal{O}\left( \frac{1}{K}\right).
$$
Thus, $\tilde{Y}$ satisfies the moment constraints of problem
 (\ref{eqn:MCOTKfiwedweight}) up to an error of order $\mathcal O\left( \frac{1}{K}\right)$ and achieves a cost that is also $\mathcal O\left( \frac{1}{K}\right)$ away from the optimal cost achieved by $(W,Y)$.

Furthermore, in the limit $K \to \infty$ optima of problems \eqref{eqn:MCOTKfiwedweight} (with an accepted error $\Bigo{\frac{1}{K}}$ on the constraints) converge to the optimum of the problem \eqref{eqn:MCOT}.

\item Yet, in the numerical experiments in the fixed weight case in 3D, the convergence in $K$ appears to be faster than $\Bigo{\frac{1}{K}}$ and even low values of $K$ can give sharp approximations of the optimum of \eqref{eqn:MCOT}.
\end{enumerate}
\end{remark}

\subsection{Properties of the set of minimizers of the particle problem}\label{sec:minimizers}

The aim of this section is to present the first main theoretical result of this paper, which states some mathematical properties on the set of minimizers of the particle problem~\ref{eqn:MCOTK}.

For any $(W,Y)\in \mathbb{R}_+^K \times ((\mathbb{R}^d)^M)^K$, we define by
$$
\mathcal I(W,Y):= \sum_{k=1}^K w_k c(X^k),
$$
where $W := (w_k)_{1\leq k \leq K}$ and $Y := (X^k)_{1\leq k \leq K}$. Problem~(\ref{eqn:MCOTK}) can then be equivalently rewritten as
\begin{equation}\label{eq:reformulation}
I^N_K = \mathop{\inf}_{(W,Y)\in \mathcal U_K^N} \mathcal I(W,Y).
\end{equation}

We begin this section by Theorem~\ref{lem:monotonecost}, which states that for any two elements of $\mathcal U_K^N$,
there exists a continuous path with values in $\mathcal U_K^N$ which connects these two elements, and such that $\mathcal I$ monotonically varies along this path.


\begin{theorem}\label{lem:monotonecost}
  Let us assume that $K \ge 2N + 6$. Let $(W_0, Y_0), (W_1, Y_1) \in \mathcal U_K^N$. Then, there exists a continuous application $\psi: [0,1] \to \mathcal U_K^N$ made of a polygonal chain such that
  $\psi(0) = (W_0, Y_0)$, $\psi(1) = (W_1, Y_1)$ and such that the application $[0,1]\ni t \mapsto \mathcal I(\psi(t))$ is monotone.
\end{theorem}


In order to explain the main ideas of the proof of Theorem~\ref{lem:monotonecost}, let us remark that,
using Tchakaloff's theorem (recalled in Section~\ref{sec:Tchakaloff}), for any measure $\pi \in \Pi(\mu;(\phi_n)_{1 \le n \le N}; M)$, satisfying,
\begin{equation}\label{eqn:ineqconstmonotonecost}
  \int_{(\R^d)^M} \vartheta d \pi \le A,
\end{equation}
and charging $K \ge 2N + 6$ points, one can find a measure $\tilde{\pi} \in \Pi(\mu;(\phi_n)_{1 \le n \le N}; M)$ charging $N+3$ points, whose support is included in the one of $\pi$, and having the same cost and the same moment against $\vartheta$.
Then, the segment $((1-t)\pi + t \tilde{\pi})_{t \in [0, 1]}$ is included in $\Pi(\mu;(\phi_n)_{1 \le n \le N}; M)$, charges at most $2N +6$ points and keeps the cost and the moment against $\vartheta$ constant.
Besides, let $\tilde{\pi}_0, \tilde{\pi}_1 \in \Pi(\mu;(\phi_n)_{1 \le n \le N}; M)$ be two measures with support on at most $N + 3$ points, and such that for $i=0,1$, $\tilde{\pi}_i$ satisfies \eqref{eqn:ineqconstmonotonecost}.
Then, the segment $((1-t)\tilde{\pi}_0 + t \tilde{\pi}_1)_{t \in [0, 1]}$ is included in $\Pi(\mu;(\phi_n)_{1 \le n \le N}; M)$, satisfies the inequality constraint \eqref{eqn:ineqconstmonotonecost} for all $t \in [0, 1]$, charges at most $2N +6$ points, and the cost varies linearly along it.
By identifying $(W_0, Y_0)$ with $\pi_0$ (resp. $(W_1, Y_1)$ with $\pi_1$), one can join $\pi_0$ to $\pi_1$ by segments (with appropriately defined intermediate measures $\tilde{\pi}_0$ and $\tilde{\pi}_1$) satisfying the constraints, and along which the cost varies linearly. The adaptation of these ideas to vectors $(W_0, Y_0), (W_1, Y_1) \in \mathcal U_K^N$, which requires to take into account the displacement of the positions between $Y_0$ and $Y_1$ as well as the ordering of the coordinates, is the object of Section \ref{sec:prooftheorem}.

A direct consequence of Theorem~\ref{lem:monotonecost} is then Corollary~\ref{lem:globalminimas} which states that any local minimizer to problem~\eqref{eq:reformulation} (or equivalently problem~\ref{eqn:MCOTK})
is actually a \itshape global minimizer \normalfont as soon as $K\geq 2N+6$. In addition, the set of minimizers forms an polygonally connected (and thus arc-connected) set.

\begin{corollary}\label{lem:globalminimas}
   Let us assume that $K \ge 2N + 6$. Then, any local minimizer of the MCOT particle problem~\eqref{eqn:MCOTK} is actually a global minimizer. Besides, the set of (local or global) minimizers of the MCOT particle problem~\eqref{eqn:MCOTK}
   is an polygonally connected subset of $\mathbb{R}_+^K \times ((\mathbb{R}^d)^M)^K$.
\end{corollary}


\section{Overdamped Langevin processes for MCOT particle problems}\label{sec:Langevin}

The motivation of this section is twofold: first, the numerical method used in this article for the resolution of the particle problems~\eqref{eqn:MCOTK} and~\eqref{eqn:MCOTKfiwedweight} can be seen as a time discretization of
constrained overdamped Langevin dynamics, which are usually encountered in molecular dynamics simulation; second, we draw here a link, on the formal level, between the long-time and large number of particles limit of these processes and a regularized version
of the MCOT problem~\eqref{eqn:MCOT} using the so-called Kullback-Leibler entropy regularization, very similar to the regularization which is at the core of
the Sinkhorn algorithm for the resolution of optimal transportation problem~\cite{peyre2019computational}.

The objective of Section~\ref{sec:general} is to recall some fundamental properties
of general constrained overdamped Langevin processes. Then, in Section~\ref{sec:appliMCOT}, we consider specific processes which are related to
the MCOT problem presented in Section~\ref{sect:MCOT}.

\subsection{Properties of general constrained overdamped Langevin processes}\label{sec:general}

\subsubsection{Definition}\label{sect:gencase}

Let $p\in \mathbb{N}^*$. Let us first introduce \itshape unconstrained \normalfont overdamped Langevin processes in the state space $\mathbb{R}^p$.
Let $(\Omega, \mathcal F, \mathbb{P})$ be a probability space. An overdamped Langevin stochastic process is
a stochastic process $(Y_t)_{t\geq 0}$ solution to the following stochastic differential equation
$$
dY_t = -\nabla V(Y_t)\,dt + \beta \,dW_t,
$$
where $V: \mathbb{R}^p \to \mathbb{R}$ is a smooth function, called hereafter the \itshape potential function \normalfont of the overdamped Langevin process,
$\beta >0$ is a positive coefficient which is proportional to the square root of the temperature of the system in
molecular dynamics ($\beta=\sqrt{2 \bar{T}}$ with $\bar{T}$ the temperature), and $(W_t)_{t\geq 0}$ is a $p$-dimensional Brownian motion.

\medskip

\itshape Constrained \normalfont overdamped Langevin processes are overdamped Langevin processes whose trajectory is enforced to be included into a given submanifold. In the sequel,
we assume that the submanifold is characterized as the zero isovalued set of a given smooth function
$\Gamma: \R^p \to \R^q$ for some $q\in \mathbb{N}^*$, so that the corresponding submanifold is defined by
$$
{\cal M} = \{ Y \in \R^p, \Gamma(Y) = 0 \}.
$$
We assume in the sequel that the submanifold $\mathcal M$ is arc connected. In addition, let us assume that there exists a neighborhood $\mathcal W$ of $\cal M$ such that, for all $Y\in \mathcal W$,
\begin{equation}\label{eq:xifullrank}
G(Y):=\nabla \Gamma(Y)^T \nabla \Gamma(Y) \in \R^{q \times q}
\end{equation}
is an invertible matrix, where $\nabla \Gamma(Y)_{i,j}=\partial_i \Gamma_j$ for $1\le i\le p$ and $1\le j \le q$. These two assumptions on the function $\Gamma$, together with the implicit function theorem, imply that $\cal M$ is a regular $(p-q)$-dimensional submanifold.

\medskip

A constrained overdamped Langevin process~\cite[Section 3.2.3]{stoltz2010free} is a $\mathbb{R}^p$-valued stochastic process $(Y_t)_{t\geq 0}$ that solves the stochastic differential equation
\begin{equation}\label{eqn:colp}
  \left\{
  \begin{aligned}
    d Y_t &= -\nabla V(Y_t) \,dt + \beta d W_t + \nabla \Gamma(Y_t) d \Lambda_t, \\
    \Gamma(Y_t) &= 0,
  \end{aligned}
  \right.
\end{equation}
where~$\beta > 0$, $(W_t)_{t\geq 0}$ is a $p$-dimensional Brownian process and $(\Lambda_t)_{t\geq 0}$ is a $q$-dimensional stochastic adapted stochastic process,
which ensures that $Y_t$ belongs to the submanifold~$\cal M$ almost surely for all $t \in \R^+$. More precisely, $\Lambda_t$ is the Lagrange multiplier associated to the constraint $\Gamma(Y_t) = 0$ and is defined by
\begin{equation}\label{def_Lambda}
  d\Lambda_t=G^{-1}(Y_t) \left[\left(\nabla \Gamma(Y_t)^T \nabla V(Y_t)-\frac {\beta^2}2 \left(\begin{array}{c}
      \sum_{i=1}^p \partial_i^2 \Gamma_1(Y_t) \\ \vdots \\ \sum_{i=1}^p \partial_i^2 \Gamma_q(Y_t)
   \end{array}\right) \right)dt -\beta \nabla \Gamma(Y_t)^TdW_t\right].
\end{equation}
Thus, if we define $P(y)=\mathrm{Id}-\nabla \Gamma(y)^TG^{-1}(u) \nabla \Gamma(y)$ the projection operator, we get $$dY_t=P(Y_t)[-\nabla V(Y_t)+ \beta dW_t ]-\frac{\beta^2}{2} \nabla \Gamma(Y_t)^TG^{-1}(Y_t)\left(\begin{array}{c}
      \sum_{i=1}^p \partial_i^2 \Gamma_1(Y_t) \\ \vdots \\ \sum_{i=1}^p \partial_i^2 \Gamma_q(Y_t)
   \end{array}\right)dt.$$

Let us assume in addition that
\begin{equation}\label{eqn:Zdef}
  Z := \int_{\R^p} e^{-\frac{2V(Y)}{\beta^2}}  \,d \sigma_{{\cal M}} (Y) < + \infty,
\end{equation}
where $\,d\sigma_{{\cal M}}$ is the surface measure (induced by the Lebesgue measure in $\R^p$, see \cite[Remark 3.4]{stoltz2010free} for a precise definition) on the submanifold $\cal M$.
Let us introduce the probability measure $\eta \in \mathcal{P}(\mathbb{R}^p)$ defined by
\begin{equation}\label{eqn:etadef}
  d \eta(Y) := \frac{1}{Z} e^{-\frac {2V(Y)}{\beta^2}} |\det G(Y) |^{-1/2} d \sigma_{{\cal M}} (Y).
\end{equation}

Under suitable assumptions, \cite[Proposition 3.20]{stoltz2010free} states that $\eta$ is the unique equilibrium distribution of the stochastic process $Y_t$ solution to the constrained overdamped Langevin dynamics~\eqref{eqn:colp} and that
\begin{equation}\label{ergo_Y}
  Y_t \text{ weakly converges to } \eta \text{ as } t\to+\infty.
\end{equation}

\subsubsection{Long-time and large number of particles limit}\label{sect:largeparticleslimit}

We recall here some results proved in~\cite[Section 2.3 and Proposition 5.1]{samaey2011numerical}, where the authors consider the so-called large-particle limit of constrained overdamped Langevin dynamics subject to average moment constraints.
The objective of the work~\cite{samaey2011numerical} was to study the properties of the constrained overdamped Langevin process in a large number of particles limit and to show the convergence towards $\eta$ of the invariant distribution of the approximating particle system when the number of particles $K\to \infty$.\footnotemark More precisely, from now on,
let us consider $p' = K p$ for some $K\in \mathbb{N}^*$. We define for any $K\in \mathbb{N}^*$ the potential function $V^K$ and the constraint function $\Gamma^K$ by:
$$
\forall Y  = (X^k)_{1\leq k \leq K} \in (\mathbb{R}^p)^K, \quad V^K(Y):= \frac{1}{K}\sum_{k=1}^K V\left(X^k\right) \quad \mbox{ and } \Gamma^K(Y):= \frac{1}{K}\sum_{k=1}^K \Gamma\left(X^k\right).
$$
We then consider the following constrained overdamped Langevin process $(Y^K_t)_{t\geq 0}$ that is assumed to be solution to the stochastic differential equation
\begin{equation}\label{eqn:colpK}
  \left\{
  \begin{aligned}
    d Y^K_t &= -\nabla V^K(Y^K_t) \,dt + \beta d W^K_t + \nabla \Gamma^K(Y^K_t) d \Lambda^K_t, \\
    \Gamma^K(Y^K_t) &= 0,
  \end{aligned}
  \right.
\end{equation}
where~$(W^K_t)_{t\geq 0}$ is a $Kp$-dimensional Brownian process and $(\Lambda^K_t)_{t\geq 0}$ is a $q$-dimensional stochastic adapted stochastic process,
which ensures that $Y^K_t$ satisfies the constraint $\Gamma^K(Y^K_t) = 0$ almost surely. The process $Y^K$ is usually called a particle system: each
coordinate $X^k$ for $1\le k\le K$ is seen as a particle.  The large number of particles limit consists in considering the limit as $K$ goes
to infinity of the stochastic process $(Y^K_t)_{t\geq 0}$.

\footnotetext{We use here the notation $K$ for the number of particles in view of the use of the Langevin dynamics to solve \eqref{eqn:MCOTK} problems, from Section \ref{sec:appliMCOT}, for which $K \ge 2N + 6$. Yet, the results recalled in Section \ref{sect:largeparticleslimit} are general and unrelated to MCOT applications.}

\medskip

It follows from~\eqref{ergo_Y} that, under suitable assumptions, as $t$ goes to $\infty$, the law of the process $Y^K_t$ converges to the probability
measure $\eta^K  \in \mathcal P\left( (\mathbb{R}^p)^K\right)$ defined for all $ Y^K = (X^1, \cdots, X^K) \in (\mathbb{R}^p)^K$ by
\begin{equation}
 \,d\eta^K(Y^K) = \frac{1}{Z^K} \left( \Pi_{k=1}^K e^{-\frac{2V(X^k)}{\beta^2}}\right)  d \sigma_{{\cal M}^K}(Y^K),
\end{equation}
where
$$
\mathcal M^K := \left\{Y^K \in (\mathbb{R}^p)^K, \; \Gamma^K(Y^K) = 0  \right\},
$$
$$
Z^K:= \int_{(\R^p)^K} e^{-\frac{2V^K(Y^K)}{\beta^2}} \,d \sigma_{{\cal M}^K} (Y^K),
$$
and
$$
G^K (Y^K) := \nabla \Gamma^K(Y^K)^T \nabla \Gamma^K(Y^K) \in \R^{q \times q}.
$$
For $1\leq k \leq K$, $\left(X_t^{k}\right)_{t\geq 0}$ is a $p$-dimensional stochastic process.
Let us denote by $\zeta^K_t\in \mathcal{P}(\mathbb{R}^p)$ the law of the first particle $X^{1}_{t}$. Then, the symmetry of the functions $V^K$ and $\Gamma^K$
implies that $\zeta^K_t$ weakly converges in law when $t \to \infty$ to the probability measure $\zeta^K_\infty$ defined for all $X\in \mathbb{R}^p$ by
\begin{equation}
  \,d \zeta^K_\infty(X) = \int_{(\R^p)^{K-1}} \,d \eta^K (X,X^2, \cdots,X^K).
\end{equation}

\medskip

Under appropriate assumptions on $V$ and $\Gamma$ which we do not detail here \cite[Proposition 5.1]{samaey2011numerical}, the sequence $\left(\zeta^K_\infty\right)_{K\in \mathbb{N}^*}$ weakly converges in $\mathcal{P}(\mathbb{R}^p)$ as $K$ goes to infinity to a
probability measure $\pi_\beta^* \in \mathcal{P}(\mathbb{R}^p)$ which is the unique solution to
\begin{equation} \label{eq:entropicrelaxlimotpb}
  \pi_\beta^*  := \argmin_{\substack{\pi \in {\cal P}\left(\R^p\right) \\ \int_{\R^p} \Gamma \,d \pi = 0}}
  \int_{\R^p} \ln\left(\frac{\,d \pi(X)}{(Z^\infty)^{-1} e^{-\frac{2V(X)}{\beta^2}}\,dX }\right) d \pi(X),
\end{equation}
where $Z^\infty := \int_{\R^p} e^{-\frac{2v(X)}{\beta^2}} \,dX$. In other words, $\pi_\beta^*$ is thus a probability measure on $\mathbb{R}^p$, which is absolutely continuous
with respect to the Lebesgue measure and which is solution to
\begin{equation} \label{eq:entropicrelaxlimotpb2}
  \pi_\beta^*  := \argmin_{\substack{\pi \in {\cal P}\left(\R^p\right) \\ \int_{\R^p} \Gamma \,d \pi = 0}}
 \int_{\mathbb{R}^p} V\left(X\right)\,d\pi(X) + \frac{\beta^2}{2}\int_{\R^p} \ln\left(\frac{\,d \pi(X)}{dX }\right) d \pi(X).
\end{equation}

\subsection{Application to MCOT problems}\label{sec:appliMCOT}

The aim of this section is to illustrate the link between the MCOT problems presented in Section~\ref{sect:MCOT} and the constrained overdamped Langevin processes introduced in Section~\ref{sec:general}.
We start by considering the \itshape fixed weight MCOT particle problem\normalfont~\eqref{eqn:MCOTKfiwedweight}, before considering the MCOT particle problem with adaptive weights~\eqref{eqn:MCOT}.

\subsubsection{Fixed-weight MCOT particle problem}\label{sec:fixedweight}

We first draw the link between constrained Langevin overdamped dynamics and the fixed weight MCOT particle problem\normalfont~\eqref{eqn:MCOTKfiwedweight}. Then, for all $K\in \mathbb{N}^*$, let us consider $(Y^K_t)_{t\geq 0}$
a constrained overdamped Langevin process solution to the stochastic differential equation (\ref{eqn:colpK}) with $p = dM$, $q=N$, $V = c$ and $\Gamma = \left(\varphi_1-\mu_1, \cdots, \varphi_N-\mu_N\right)$
where for all $1\leq n \leq N$, $\varphi_n$ is defined by~(\ref{eq:Phin}).

\medskip

Then, the stochastic dynamics (\ref{eqn:colpK}) can be viewed as a randomized version of a constrained gradient numerical method for the resolution of problem (\ref{eqn:MCOTKfiwedweight}), where for all $t\geq 0$,
$Y_t^K = (X_t^{1}, \cdots, X_t^{K})\in ((\mathbb{R}^d)^M)^K$ and where for all $1\leq k \leq K$,
$$
X_t^{k} = \left( x_{1,t}^{k}, \cdots, x_{M,t}^{k}\right) \in \left(\mathbb{R}^d\right)^M.
$$
Note that it is not clear in general that $V$ and $\Gamma$ satisfy the regularity assumptions which ensure the convergence results stated in Section~\ref{sect:largeparticleslimit} to hold true. But,
formally, assuming that the long-time limit and large number of particles convergence holds nevertheless, the associated measure $\pi_\beta^*$ solution~(\ref{eq:entropicrelaxlimotpb2}) can be equivalently rewritten as
\begin{equation} \label{eq:entropicrelaxlimotpb3}
  \pi_\beta^*  := \argmin_{\substack{\pi \in {\cal P}\left( (\R^d)^M\right) \\ \forall 1\leq n \leq N, \; \int_{(\R^d)^M} \varphi_n \,d \pi = \mu_n}}
  \mathcal{J}(\pi),
\end{equation}
where
$$
\mathcal J(\pi):= \int_{(\mathbb{R}^d)^M} c\left(X\right)\,d\pi(X) + \frac{\beta^2}{2}\int_{(\R^d)^M} \ln\left(\frac{\,d \pi(X)}{dX}\right) d \pi(X).
$$

Recall that $\pi_\beta^*$ is the large number of particles limit of the long-time limit of the law of one particle associated to the constrained overdamped Langevin process. Notice that $\pi_\beta^*$
can be equivalently seen as the solution of an entropic regularization of the
MCOT problem (\ref{eqn:MCOT}), where the term $ \int_{(\R^d)^M} \ln\left(\frac{\,d \pi(X)}{dX}\right) d \pi(X)$ can be identified as the Kullback-Leibler entropy of the measure $\pi$ with respect to the Lebesgue measure.
Thus, Problem~\ref{eq:entropicrelaxlimotpb3} is close to the entropic regularization of optimal transport problems used in several works~\cite{benamou2015iterative, de2018entropic,nenna2016numerical, peyre2019computational},
in particular for the so-called Sinkhorn algorithm~\cite{benamou2015iterative}.

\medskip

Let us point out here that, at least on the formal level, we expect the family $(\pi^*_\beta)_{\beta >0}$ to weakly converge to a minimizer of (\ref{eqn:MCOT}) as $\beta$ goes to $0$ (a similar result is proven in \cite[Theorem 2.7]{MR3635459}).

\subsubsection{Adaptive-weight MCOT particle problem}\label{sec:adaptiveweight}

A similar link can be drawn between constrained Langevin overdamped dynamics and the MCOT particle problem\normalfont~\eqref{eqn:MCOTK} with adaptive weights.

In order to fit in the framework of the constrained Langevin overdamped dynamics, without any positivity constraint, let us introduce a continuous surjective function $f: \mathbb{R} \to \mathbb{R}_+$, which we call hereafter a \itshape weight function\normalfont. We assume that $f$ satisfies the following assumption: there exists an interval
$I\subset \mathbb{R}$ such that the Lebesgue measure of $I$ is equal to $1$ and such that $\int_I f = 1$. A simple choice of admissible weight function can be given by $f(a) = a^2$ for all $a\in \mathbb{R}$ with $I = (\frac 1 2, \frac {25^{1/3}} 2)$.

\medskip

Then, for all $K\in \mathbb{N}^*$, let us consider $\left(\overline{Y}^K_t\right)_{t\geq 0}$
a constrained overdamped Langevin process solution to the stochastic differential equation (\ref{eqn:colpK}) with $p = dM +1$, $q = N+1$, and where for all $\overline{X} = (a,X) \in \mathbb{R} \times (\mathbb{R}^d)^M$,
$\overline{V}(\overline{X}) = f(a)c(X)$ and $\overline{\Gamma}(\overline{X}) = \left(f(a) -1, f(a)\varphi_1(X)-\mu_1, \cdots, f(a)\varphi_N(X)-\mu_N\right)$.
Then, the stochastic dynamics (\ref{eqn:colpK}) can be viewed as a randomized version of a constrained gradient numerical method for the resolution of the optimization problem
\begin{equation}
\inf_{  (A,Y) \in \mathcal V_K^N} \quad \sum_{k = 1}^K f(a_k) c\left(X^k\right),
\end{equation}
where
\begin{align}
  &\mathcal V^N_K := \Bigg\{ (A,Y) \in \mathbb{R}^K \times \left( (\mathbb{R}^d)^M \right)^K, \quad A = (a_k)_{1\leq k \leq K}, \; Y=(X^k)_{1\leq k \leq K},  \\
  & \sum_{k=1}^K f(a_k) = 1, \quad \sum_{k=1}^K  f(a_k) \vartheta(X^k) \le A, \quad \forall 1\le n \le N,\; \sum_{k=1}^K f(a_k) \varphi_n(X^k) = \mu_n\Bigg\}, \notag
\end{align}
which is equivalent to problem (\ref{eqn:MCOTK}) using the surjectivity of $f$.

Note that the choice of the function $f$ can influence the dynamics as it regulates both the way the brownian motion $W$ affects the weights, and the balance, in the minimization of $\overline{V}$ and in the enforcement of the constraint $\overline{\Gamma}(\overline{X}) = 0_N$, between a displacement of particles and a change in weights.

\medskip

Here again, it is not clear in general that $\overline{V}$ and $\overline{\Gamma}$ satisfies the regularity assumptions which ensures the convergence results stated in Section~\ref{sect:largeparticleslimit} to hold true. But,
using formal computations, we can consider the associated measure $\overline{\pi}_\beta^{\rm a} \in \mathcal{P}\left( \mathbb{R} \times (\mathbb{R}^d)^M \right)$ solution to
\begin{equation} \label{eq:entropicrelaxlimotpb4}
  \overline{\pi}_\beta^{\rm a}  := \argmin_{\substack{\overline{\pi} \in {\cal P}\left( \mathbb {R} \times (\R^d)^M\right) \\
  \int_{a\in \mathbb{R}}\int_{X\in(\R^d)^M} f(a) \,d \overline{\pi}(a,X) = 1\\
  \forall 1\leq n \leq N, \; \int_{a\in \mathbb{R}}\int_{X\in(\R^d)^M} f(a)\varphi_n(X) \,d \overline{\pi}(a,X) = \mu_n}} \overline{\mathcal J}(\overline{\pi}),
\end{equation}
where
$$
\overline{\mathcal J}(\overline{\pi}) := \int_{a\in \mathbb{R}}\int_{X\in(\mathbb{R}^d)^M} f(a) c\left(X\right)\,d\overline{\pi}(a,X)
 + \frac{\beta^2}{2}\int_{a\in \mathbb{R}}\int_{X\in(\R^d)^M} \ln\left(\frac{\,d \overline{\pi}(a,X)}{dadX}\right) d \overline{\pi}(a,X).
$$
Let us introduce now $\pi_\beta^{\rm a}\in \mathcal{P}\left((\mathbb{R}^d)^M\right)$ defined by
$$
\,d\pi_\beta^{\rm a}(X) = \int_{a\in \mathbb{R}} f(a)\,d\overline{\pi}_\beta^{\rm a}(a,X).
$$
Then $\pi_\beta^{\rm a}$ satisfies the constraints of problem (\ref{eq:entropicrelaxlimotpb3}) and
$$
\overline{\mathcal J}(\overline{\pi}_\beta^{\rm a}) = \int_{X\in (\mathbb{R}^d)^M} c(X)\,d\pi_\beta^{\rm a}(X) + \frac{\beta^2}{2}\int_{a\in \mathbb{R}}\int_{X\in(\R^d)^M} \ln\left(\frac{\,d \overline{\pi}_\beta^{\rm a}(a,X)}{dadX}\right)d \overline{\pi}_\beta^{\rm a}(a,X).
$$
%

\noindent Notice that, as a consequence, problem (\ref{eq:entropicrelaxlimotpb4}) may be seen as a second kind of entropic regularization of (\ref{eqn:MCOT}) and that $\pi_\beta^{\rm a}$ is expected to be an approximation of some minimizer to (\ref{eqn:MCOT}) as $\beta$ goes to $0$.

\medskip

Let us notice here that the assumption made on $f$ ensures that, for all $\pi \in \mathcal{P}\left( (\mathbb{R}^d)^M\right)$, there exists a probability measure $\overline{\pi} \in \mathcal{P}\left(\mathbb{R} \times (\mathbb{R}^d)^M\right)$ such that
$$
\,d\pi(X) = \int_{a\in \mathbb{R}} f(a)\,d\overline{\pi}(a,X).
$$
Indeed, defining $\,d\overline{\pi}(a,X):= \un_I(a)\,da \otimes \,d\pi(X)$ yields the desired result. Besides, we easily check that $\overline{\mathcal{J}}(\overline{\pi})=\mathcal{J}(\pi)$, which leads immediately to $\overline{\mathcal J}(\overline{\pi}_\beta^{\rm a}) \le \mathcal J(\pi_\beta^*)$ from the optimality of $\overline{\pi}_\beta^{\rm a}$.

\section{Numerical optimization method}\label{sect:numericaldiscretization}

We present in this section the numerical procedure we use in our numerical tests to compute approximate solutions to the particle problems with fixed weights (\ref{eqn:MCOTKfiwedweight}) or adaptive weights~(\ref{eqn:MCOTK}) for a fixed
given $K\in \mathbb{N}^*$.
Note that (\ref{eqn:MCOTKfiwedweight}) can be equivalently rewritten as
\begin{equation}\label{eqn:MCOTKfiwedweight2}
  J^N_K := \inf_{\substack{ Y^K \in ((\R^d)^M)^K,\\
 \Gamma^K(Y^K)= 0, \\
 \Theta^K(Y^K) \le A}} \quad V^K(Y^K),
\end{equation}
where $V^K$ and $\Gamma^K$ are defined in Section~\ref{sec:fixedweight}, and where
$$
\Theta^K:
\left\{
\begin{array}{ccc}
 ((\R^d)^M)^K & \mapsto & \mathbb{R}\\
Y:=(X^1, \cdots, X^K) & \to & \frac{1}{K}\sum_{k=1}^K \vartheta(X^k).\\
\end{array}
\right.
$$
Besides, problem~(\ref{eqn:MCOTK}) can be rewritten equivalently as
\begin{equation}\label{eqn:MCOTK2}
  I^N_K := \inf_{  \substack{  \overline{Y}^K \in \left(\mathbb{R} \times ((\R^d)^M)\right)^K\\
  \overline{\Gamma}^K(\overline{Y}^K)= 0, \\
 \overline{\Theta}^K(\overline{Y}^K) \le A}} \quad \overline{V}^K(\overline{Y}^K),
\end{equation}
where $\overline{V}^K$ and $\overline{\Gamma}^K$ are defined in Section~\ref{sec:adaptiveweight}, and where
$$
\overline{\Theta}^K:
\left\{
\begin{array}{ccc}
 ( \mathbb{R} \times (\R^d)^M)^K & \mapsto & \mathbb{R}\\
\overline{Y}:=((a^1,X^1), \cdots, (a^K,X^K) ) & \to & \frac{1}{K}\sum_{k=1}^K f(a^k)\vartheta(X^k).\\
\end{array}
\right.
$$

For the sake of simplicity, we restrict the presentation here to the method used for the resolution of (\ref{eqn:MCOTKfiwedweight2}), since the method used for the resolution of (\ref{eqn:MCOTK2}) follows exactly the same lines.

\subsection{Time-discretization of constrained overdamped Langevin dynamics}\label{sec:TdLangevin}

The numerical procedure considered in this paper consists in a time discretization of the dynamics~\eqref{eqn:colp} with an adaptive time step and noise level. The main idea of the algorithm is the following:
let $(W_n)_{n\in \mathbb{N}}$ be a sequence of iid normal vectors of dimension $dMK$. At each iteration $n\in \mathbb{N}^*$ of the procedure, starting from an initial guess $Y_0^K \in \mathcal M^K$ for $n=0$,
a new approximation $Y_{n+1}^K \in \mathcal M^K$ is computed as the projection in some sense of $Y_{n+1/2}^K := Y_n^K - \nabla V^K(Y_n^K)\Delta t_n + \beta_n \sqrt{\Delta t_n} W_n$ onto $\mathcal M^K$, where $\Delta t_n>0$ is
the time step and $\beta_n >0$ is the noise level
at iteration $n$. Precisely, the next iterate $Y^K_{n+1}$ is computed as $Y^K_{n+1/2} + \nabla \Gamma^K(Y^K_n) \cdot \Lambda_{n+1}^K$ where $\Lambda_{n+1}^K\in \mathbb{R}^N$ is a Lagrange multiplier which ensures that
the constraint $\Gamma^K(Y_{n+1}^K) = 0$ is satisfied.

The complete resulting procedure is summarized in Algorithm~\ref{algo:ConstOvrdmpdLang}.

\begin{algorithm}[htp]
  \caption{Constrained Overdamped Langevin Algorithm \label{algo:ConstOvrdmpdLang}}
  \begin{algorithmic}
    \STATE Input $Y^K_0 \in {\cal M}^K$,  $\Delta t_0 >0$, $\beta_0 >0$, $\tau_0 >0$, $i_{\mathrm{const}} \in \N^*$,  $i_{\mathrm{max}} \in \N^*$, $\mathrm{NoiseDecrease}: \R^+\times \N \to \R^+ $, $n_\mathrm{max} \in \N^*$
    \STATE Fix $n=0$, $\Lambda^K_0 = 0 $.
    \STATE Define $(W_n)_{n \in \N}$ a sequence of i.i.d. normal vectors of the same dimension as $Y^K_0$.
    \WHILE{ $n \le n_\mathrm{max}$}
      \STATE $\mathrm{AdaptTimeStep}(Y^K_n, \Lambda^K_n, \Delta t_n, \beta_n, \tau_n)$
      \STATE $Y^K_{n + 1/2} := Y^K_n - \nabla V^K(Y^K_n) \Delta t_n + \beta_n \sqrt{\Delta t_n} W_n$
      \IF{$\mathrm{Projection}(Y^K_{n + 1/2}, \nabla \Gamma^K(Y^K_n), \Lambda^K_n, i_{\rm max})$ succeeds}
        \STATE $Y^K_{n + 1}, \Lambda^K_{n + 1}, i_n \leftarrow \mathrm{Projection}(Y^K_{n + 1/2}, \nabla \Gamma^K(Y^K_n), \Lambda^K_n, i_{\rm max})$
        \IF{$i_n \le i_{\mathrm{const}}$}
          \STATE $\tau_{n+1} \leftarrow 2\tau_n$
        \ENDIF
        \STATE $\beta_{n+1} \leftarrow \mathrm{NoiseDecrease}(\beta_n, n)$
        \STATE $\Delta t_{n+1} \leftarrow \Delta t_n$; $\tau_{n+1} \leftarrow \tau_n$
        \STATE $n \leftarrow n+1$
      \ELSE
        \STATE $\tau_n \leftarrow \tau_n /2$
      \ENDIF
    \ENDWHILE
    \RETURN $\min(V^K(Y^K_n), 0\leq n \leq n_{\rm max} )$
  \end{algorithmic}
\end{algorithm}

We discuss here three main difficulties about the algorithm we propose:
\begin{itemize}
 \item the initialization step which consists in finding an element $Y_0^K\in \mathcal M^K$;
 \item the choice of the values of the time step $\Delta t_n$ and noise level $\beta_n$ at each iteration of the algorithm;
 \item the practical method used in order to compute a projection of $Y_{n+1/2}^K$ onto the submanifold $\mathcal M^K$, and in particular the value of the Lagrange multiplier $\Lambda_{n+1}^K$.
\end{itemize}

\medskip

The procedure chosen to adapt the time step and noise level is discussed in Section~\ref{sec:adapttime}. The algorithm used to compute a projection of $Y_{n+1/2}^K$ onto the submanifold $\mathcal M^K$ and the value of the Lagrange multiplier $\Lambda_{n+1}^K$
is detailed in Section~\ref{sec:projection}. Finally, the initialization procedure used to compute a starting guess $Y_0^K\in \mathcal M^K$ is exaplined in Section~\ref{sec:init}.

\subsection{Time step and noise level adaptation procedure}\label{sec:adapttime}

Two remarks are in order to motivate the procedure we propose here:
\begin{enumerate}[(i)]
  \item \label{it:implem:const} the computation of the Lagrange multiplier $\Lambda^K_{n+1}$ at each iteration $n$ of the algorithm and of the resulting value of $Y^K_{n+1}$ must be fast (as it is executed at each step).
  \item \label{it:implem:deltat} the time-step $\Delta t_n$ must be:
  \begin{enumerate}[(a)]
    \item \label{it:implem:deltat:small}small enough for the procedure that computes the Lagrange multiplier to be well-defined,
    \item \label{it:implem:deltat:large}large enough for the total number of iterations needed to observe convergence to be reasonable. In practice, $n_{\rm max}$ was chosen to be of the order of $20000$ in the numerical experiments
    presented in Section~\ref{sect:numericaltests}.
  \end{enumerate}
\end{enumerate}

  To address item (\ref{it:implem:const}), we use a Newton method similar to the one proposed in~\cite{lelievre2012langevin, lelievre2019hybrid} to enforce the constraints and compute the Lagrange multiplier $\Lambda^K_{n+1}$
  which is summarized in Algorithm~\ref{algo:EnforceConstraints} and detailed in Section~\ref{sec:projection}. This method is observed to converge fast if
 the value $Y_{n+1/2}^K$  is close enough to the submanifold $\mathcal M^K$. The tolerance threshold allowed at each step on the satisfiability of the constraints is given by $\tau_n>0$, the value of which is also adapted at each step.
 Its precise value is inferred as follows:
  if the Newton method converges fast enough (i.e.\ if the number of iterations needed to ensure convergence $i_n$ is lower than some fixed value $i_{\mathrm{const}}$),
  then the value of $\tau_n$ is multiplied by $2$. On the other hand, if the Newton method does not converge in a maximum number of iterations (given by $i_\mathrm{max}$),
  then $\tau_n$ is divided by 2. This step may involve a new time-step computation for iteration $n$, which we detail below.

  The time-step $\Delta t_n$ is adapted (in order to answer item (\ref{it:implem:deltat})) through the $\mathrm{AdaptTimeStep}$ subprocedure (Algorithm \ref{algo:AdaptTimeStep}).
  It is increased at each step $n$  if the constraints are satified up to a tolerance threshold \emph{lower} than $\tau_n$ (in order to answer
  (\ref{it:implem:deltat:large})). Otherwise, the time-step
  is divided by $2$ as many times as needed for $Y^K_{n + 1/2}$ to satisfy the constraints defining the submanifold up to a tolerance lower than $\tau_n$ (in order to satisfy item (\ref{it:implem:deltat:small})).

Moreover, the noise-level $\beta_n$ is decreased at each iteration $n$ at a rate inspired from Robbins-Siegmund Lemma \cite[Theorem 6.1]{pages2018numerical} for non-constrained stochastic gradient optimization,
using the $\mathrm{NoiseDecrease}$ function in Algorithm \ref{algo:ConstOvrdmpdLang}. This is managed through the $\mathrm{NoiseDecrease}$ function. In the numerical experiments presented in Section~\ref{sect:numericaltests}, we used two possible choices of $\mathrm{NoiseDecrease}$ function defined respectively by $(\beta, n) \mapsto \beta$ (noise level unchanged)
and $(\beta, n) \mapsto \sqrt{\frac{n}{n+1}}\beta$ (slow decrease of the noise level: note that this is the relative decrease, so that after $n$ steps, the noise is $\beta_0/\sqrt{1+n}$).

\begin{algorithm}[htp]
  \caption{AdaptTimeStep subprocedure \label{algo:AdaptTimeStep}}
  \begin{algorithmic}
    \STATE Input: $Y^K$, $\Lambda$, $\Delta t$, $\beta$, $\tau$, $n$
    \IF{$ \| \Gamma^K (Y^K - \nabla V^K(Y^K) 2 \Delta t + W_n \sqrt{2\Delta t} \beta) \| \le \tau$}
      \STATE $\Delta t \leftarrow 2 \Delta t$;
    \ELSE
      \WHILE{$ \| \Gamma^K ( Y^K - \nabla V^K(Y^K) \Delta t + W_n \sqrt{2\Delta t} \beta) \| \ge \tau$}
        \STATE  $\Delta t \leftarrow \Delta t /2$; $\Lambda \leftarrow \Lambda /2$
      \ENDWHILE
    \ENDIF
  \end{algorithmic}
\end{algorithm}

\subsection{Projection method}\label{sec:projection}

As mentioned earlier, to compute $Y^K_{n+1} \in \mathcal M^K$ and $\Lambda^K_{n+1}$ from $Y^K_{n+1/2}$, we use a Newton method similar to the one proposed
in~\cite{lelievre2012langevin, lelievre2019hybrid}. We refer the reader to~\cite[Section 2.2.2]{lelievre2019hybrid} for theoretical considerations on such projections.

More precisely, the procedure reads as follows: given $Y^K_n, Y^K_{n+1/2}\in ((\mathbb{R}^d)^M)^K$, the aim of the Newton procedure is to find a solution $\Lambda_{n+1}^K\in \mathbb{R}^N$ to the equation
$$
\Gamma^K\left( Y^K_{n+1/2} + \nabla \Gamma^K (Y^K_n) \cdot \Lambda_{n+1}^K\right) = 0.
$$
We numerically observe that this Newton procedure only converges in cases when $Y^K_{n+1/2}$ and $Y^K_n$ are close enough to the manifold $\mathcal M^K$. Provided that $Y_n^K$ belongs to $\mathcal M^K$, $Y^K_{n+1/2}$ can be made arbitrarily
close to the submanifold
provided that the value of the time step $\Delta t_n$ is chosen small enough. We also refer the reader to~\cite[Theorem 1.4.1]{polak1997optimization} for theoretical conditions which guarantee the convergence of this Newton procedure.

This projection procedure, together with the routine for the adaptation of the error tolerance $\tau_n$ on the satisfiability of the constraints, is summarized in Algorithm~\ref{algo:EnforceConstraints}. Note that this Newton algorithm requires
the inversion of matrices of the form
$$
\nabla \Gamma^K(Y^K_{n+1/2} + \nabla \Gamma^K(Y^K_n) \cdot \Lambda)^T \cdot \nabla \Gamma(Y^K_n)
$$
for $\Lambda\in \mathbb{R}^N$ and that we cannot theoretically guarantee the invertibility of this matrix in general. In practice, it naturally depends significantly on the choice of test functions $(\phi_n)_{1\leq n \leq N}$.

\begin{algorithm}[htp]
  \caption{Projection subprocedure (Newton method) \label{algo:EnforceConstraints}}
  \begin{algorithmic}
    \STATE Input: $Y^K_{n+1/2}$, $\nabla \Gamma^K(Y^K_n)$, $\Lambda^K_n$, $i_{\rm max}$
    \STATE $i = 0$, $\Lambda'_0 \leftarrow \Lambda_n$
    \WHILE{$ \| \Gamma^K(Y^K_{n+1/2} + \nabla \Gamma^K(Y^K_n) \cdot \Lambda'_i) \| > 10^{-16}$ and $i \le i_{\max}$}
    \STATE $\Lambda'_{i+1} \leftarrow \Lambda'_i - \left(\nabla \Gamma^K(Y^K_{n+1/2} + \nabla \Gamma^K(Y^K_n) \cdot \Lambda'_i)^T \cdot \nabla \Gamma(Y^K_n)\right)^{-1} \cdot \Gamma^K((Y^K_{n+1/2} + \nabla \Gamma^K(Y^K_n) \cdot \Lambda'_i)$
      \STATE $i \leftarrow i+1$
    \ENDWHILE
    \IF{$ \| \Gamma^K(Y^K_{n+1/2} + \nabla \Gamma^K(Y^K_n) \cdot \Lambda'_i) \| \le 10^{-16}$}
      \RETURN $Y^K_{n+1/2} + \nabla \Gamma^K(Y^K_n) \cdot \Lambda'_i$, $\Lambda'_i$, $i$
    \ELSE
      \RETURN Projection failure.
    \ENDIF
  \end{algorithmic}
\end{algorithm}

\subsection{Initialization procedure}\label{sec:init}

Algorithm~\ref{algo:ConstOvrdmpdLang} is initialized with an initial guess $Y^K_0$ which is assumed to belong to the constraints submanifold $\mathcal M^K$.
In practice, finding an element which belongs to this submanifold is a delicate task, especially when the number of test functions is large. Indeed, as mentioned in the preceding section, the Newton procedure described in
Section~\ref{sec:projection} only converges if the starting point of the algorithm is sufficiently close to the manifold $\mathcal M^K$. This is the reason why this initialization step is rather performed using a method inspired from~\cite[Section 5 example 3]{zhang2020ergodic}.
A Runge-Kutta~3
(Bogacki-Shampine) numerical scheme~\cite[(5.8-42)]{ralston2001first} is used in order to discretize the dynamics
$$
\frac{d}{dt}Y^K(t) =  F(Y^K(t))
$$
starting from a random initial state $Y^K(t=0) = Y^{K,0}\in \left((\mathbb{R}^d)^M\right)^K$, where $F$ is defined as
\begin{equation}
\forall Y^K \in \left((\mathbb{R}^d)^M\right)^K,\quad  F(Y^K) = -\left\| \Gamma^K\left(Y^K\right) \right\|_2^2 \frac{\nabla \Gamma^K(Y^K) \cdot \Gamma^K(Y^K) }{\left\| \nabla \Gamma^K(Y^K) \cdot \Gamma^K(Y^K) \right\|_2^2}.
\end{equation}
We observe that such a numerical procedure is more robust than a Newton algorithm, even if it can converge very slowly.

\medskip

Let us mention here that, in the case of the particle problem~(\ref{eqn:MCOTK2}) with adaptive weights, an additional step may be used prior to such a Runge-Kutta method, which consists in using a Carath\'eodory-Tchakaloff subsampling procedure.
Carath\'eodory-Tchakaloff subsampling~\cite{piazzon2017caratheodory, tchernychova2016caratheodory} has been introduced to compute low nodes cardinality cubatures.

In our context, this method can be adapted to find a low nodes cardinality starting point, as close as possible to the constraints submanifold $\overline{\mathcal M}^K$.
More precisely, the method works as follows: we fix a value $K_\infty \gg K$ and compute $(X^1, \cdots, X^{K_\infty})$ iid samples of random vectors according to the probability law $\mu$. A
Non-Negative Least Squares (NNLS) is then used to find a sparse solution to the optimization problem
\begin{equation} \label{eq:nnlspb}
  w^* \in \argmin_{w \in \R_+^{K_\infty}} \left\|\Phi w - \bar{\mu} \right\|^2,
\end{equation}
where $\Phi:=\left(\Phi_{ n,k}\right)_{1\leq n \leq N+1, 1\leq k \leq K_\infty} \in \mathbb{R}^{N\times K_\infty}$, $\overline{\mu} = (\mu_1, \cdots, \mu_N, 1)\in \mathbb{R}^{N+1}$ and
$$
\forall 1\leq k \leq K_\infty, \quad \forall 1\leq n \leq N, \; \Phi_{n,k} = \varphi_n(X^k) \mbox{ and } \Phi_{N+1,k} = 1.
$$

By Kuhn-Tucker conditions for the NNLS problem \cite[Theorem (23.4)]{lawson1995solving}, there exists a solution $w^*:=(w_k^*)_{1\leq k \leq K_\infty}\in \R_+^{K_\infty}$ to (\ref{eq:nnlspb})
such that $\#J  \le N+1$ with $J:= \left\{ 1\leq k \leq K_\infty, w_k^*>0\right\}$. Common algorithms such as the Lawson-Hanson method~\cite[Theorem (23.10)]{lawson1995solving} enable to compute such a sparse solution.
Let us point out that any solution $w^*$ to (\ref{eq:nnlspb}) then satisfies
\begin{equation}
\sum_{n=1}^N \left| \sum_{k \in J} w^*_k \varphi_n(X^k) - \bar{\mu}_n \right|^2 \le \sum_{n=1}^N \left| \frac 1 {K_\infty} \sum_{k =1}^{K_\infty} \varphi_n(X^k) - \bar{\mu}_n \right|^2.
\end{equation}

In practice, in the case when $\#J \leq K$, the positions and weights returned by the Carath\'eodory-Tchakaloff Subsampling procedure are subdivided and randomly perturbed with a small amount of noise.

\clearpage
\section{Numerical tests}\label{sect:numericaltests}

The aim of this section is to illustrate the results obtained via the numerical procedure described in Section~\ref{sect:numericaldiscretization} for the resolution of the particle problems~(\ref{eqn:MCOTK2}) and~(\ref{eqn:MCOTKfiwedweight2})
in different test cases.

Section~\ref{sect:1D} is devoted to results obtained in cases where $d=1$ and Section~\ref{sect:3D} contains numerical results obtained in examples where $d=3$. The experiments presented in this section have been implemented in python~3 using scipy and numpy modules,
and tested on a server with an Intel Xeon processor with 32 cores (hyperthreaded) and 192 Go RAM.

\subsection{One-dimensional test cases ($d=1$)}\label{sect:1D}

\subsubsection{Theoretical elements}

In the case where $d=1$, the solution to the optimal transport problem~\eqref{eqn:OT} is analytically known in the case when $c$ is a symmetric repulsive cost from~\cite[Theorem 1.1]{colombo2015multimarginal}. For sake of completeness, we recall their result for the cost function that we consider in our numerical experiments.


\begin{theorem}[Colombo, De Pascale, Di Marino, 2015]\label{thm:colombo}
  Let $\epsilon\ge0$ and $c:\R^M \to [0, + \infty]$ be the cost defined by
  \begin{equation}\label{cost_1D}
  \forall x_1,\cdots,x_M\in \mathbb{R},\quad  c(x_1, \dots, x_M) = \sum_{1\le  i, j\le M, i\not=j}   \frac 1 {\epsilon +|x_i-x_j|}.
  \end{equation}
  Let $\mu$ be an non atomic probability measure on $\R$ such that
  \begin{equation}
    \min_{\pi \in \Pi(\mu; M)} \int_{\R^M} c(x_1, \dots, x_M) d \pi (x_1, \dots, x_M) < + \infty.
  \end{equation}
  Let $- \infty = d_0 < d_1 < \dots < d_M = + \infty$ be such that
  \begin{equation}
    \mu([d_i, d_{i+1}]) = \frac 1 M, \quad i = 0, \dots, M-1.
  \end{equation}
  Let $T: \R \to \R$ be the unique (up to $\mu$-null sets) function increasing on each interval $[d_i, d_{i+1}]$, $i=0, \dots, M-1$ and such that
  \begin{equation}
    \begin{split}
      T\#(\Ind{[d_i, d_{i+1}]}\mu) &= \Ind{[d_{i+1}, d_{i+2}]}\mu, \quad i=0,\dots, M-2 \\
      T\#(\Ind{[d_{M-1}, d_{M}]}\mu) &= \Ind{[d_{0}, d_{1}]}\mu.
    \end{split}
  \end{equation}
  Then $T$ is an admissible map for
  \begin{equation}
    \inf_{T: \R \to \R \, \text{Borel}, \, T\# \mu = \mu, \, T^{(M)} = \mathrm{Id}} \int_{\R} c(x, T(x), \dots, T^{(M-1)}(x)) d \mu(x),
  \end{equation}
  where $T^{(i)} = \overbrace{T \circ \dots \circ T}^{\text{$i$ times}}$.

  Moreover, the only symmetric optimal transport plan is the symmetrization of the plan induced by the map $T$.
\end{theorem}

We make use of Theorem \ref{thm:colombo} to compare the exact solution of problem~\eqref{eqn:OT} together with the approximation given by the numerical procedure described in Section~\ref{sect:numericaldiscretization}
to solve the MCOT particle problems with fixed or adaptive weights.

\subsubsection{Marginals, test functions, cost and weight functions}

\paragraph{Marginals.} The numerical experiments in this section were realized with three different marginal laws, which are respectively denoted by $\mu_1$, $\mu_2$ and $\mu_3$ and defined by

\begin{align}
  \,d\mu_1(x) &:= \frac 1 2 \Ind{[-1, 1]}(x)\,dx, \\
  \,d\mu_2(x) &:= \left[\frac{\pi}{10}\cos\left(\frac{5\pi}{2}x\right) + 0.46\right]\Ind{[-1, 1]}(x)\,dx, \\
  \,d\mu_3(x) &:= \left[0.13\pi\cos\left(\frac{13\pi}{2}x\right) + 0.48\right]\Ind{[-1, 1]}(x)\,dx.
\end{align}

The densities of $\mu_1,\mu_2,\mu_3$ are plotted in Figure \ref{fig:1Dlaws}.

\begin{figure}[tp]
  \centering
  \begin{subfigure}{0.32\textwidth}
    \includegraphics[width=\textwidth]{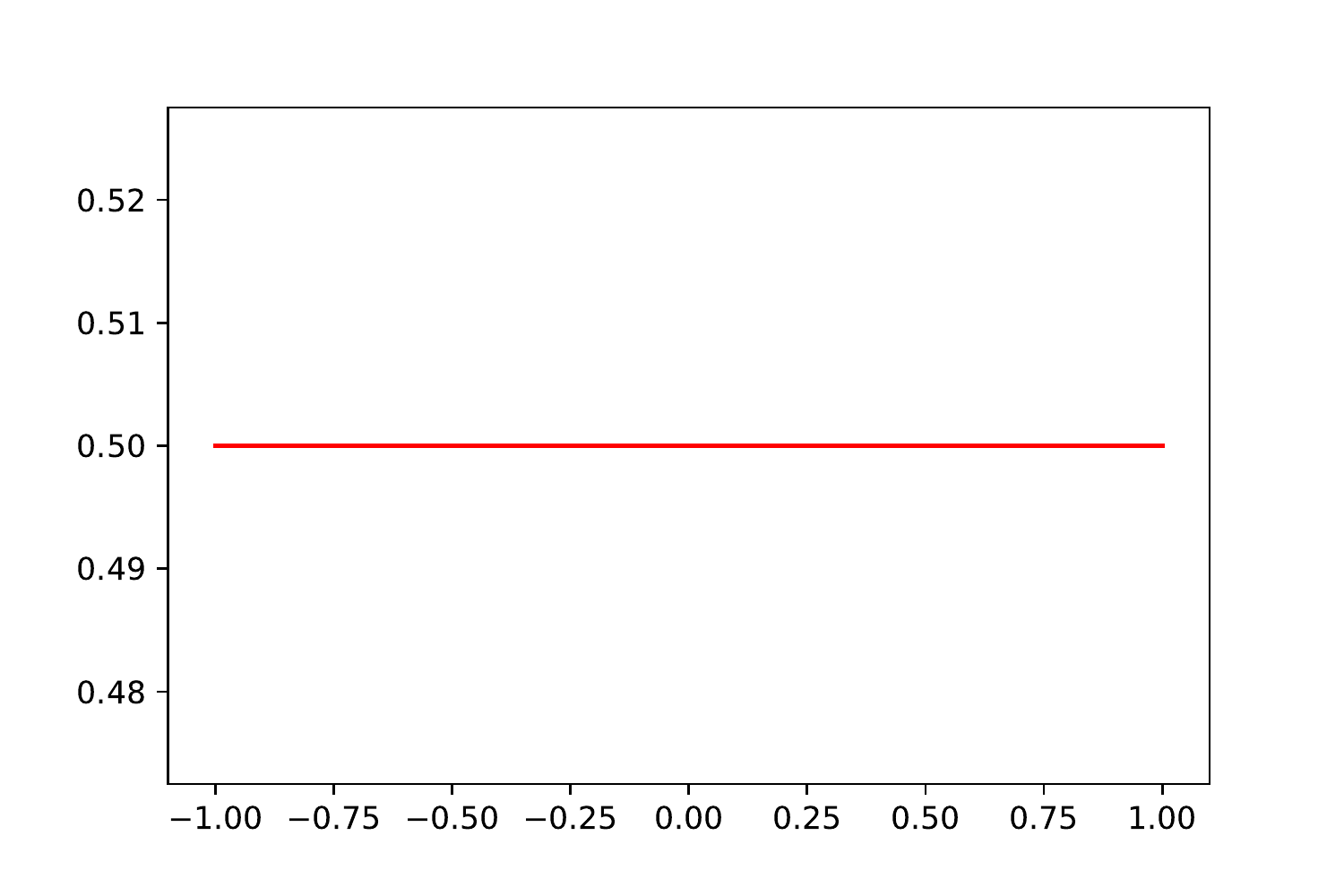}
    \caption{$\mu_1$\label{fig:1Dmu1}}
  \end{subfigure}
  \begin{subfigure}{0.32\textwidth}
    \includegraphics[width=\textwidth]{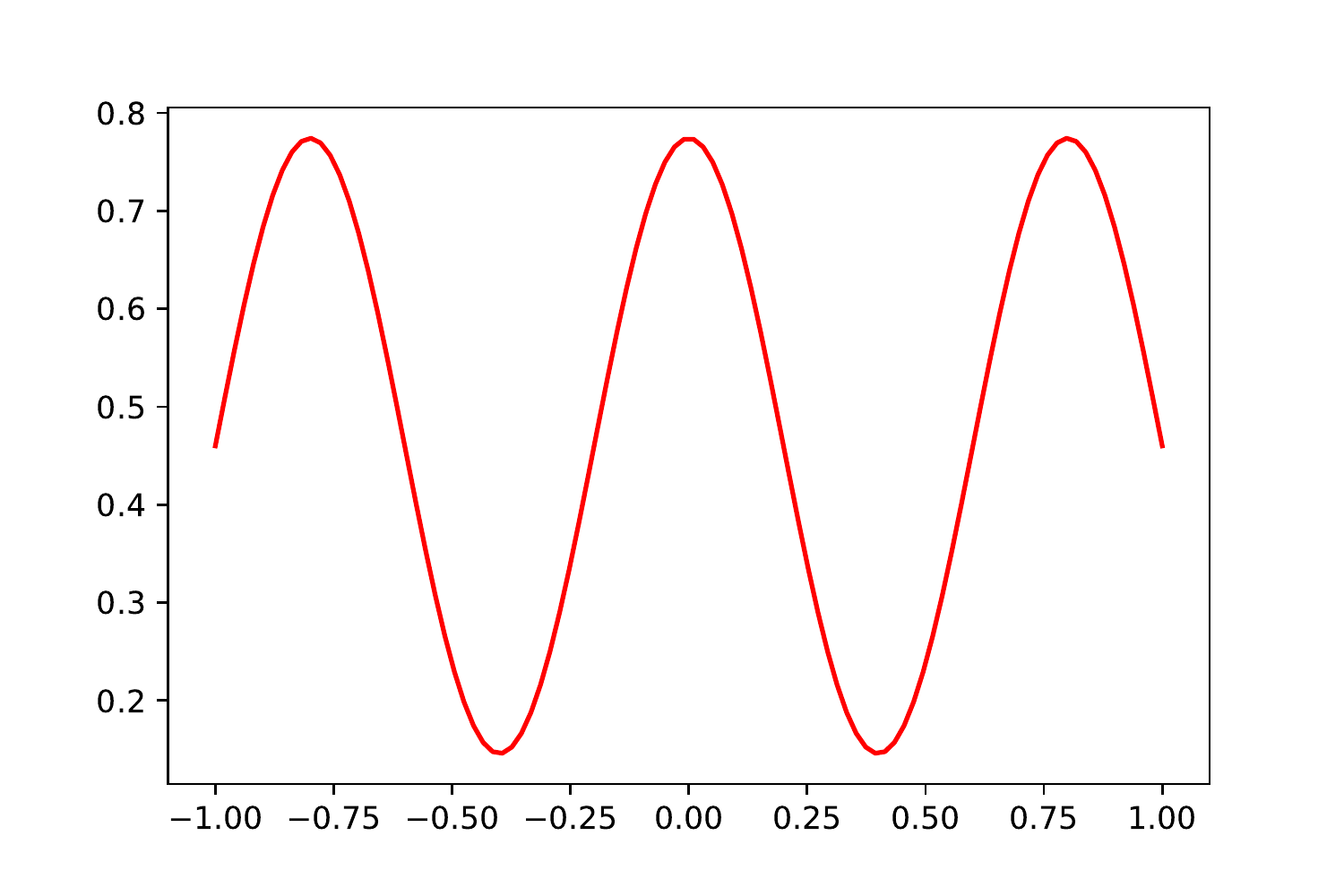}
    \caption{$\mu_2$\label{fig:1Dmu2}}
  \end{subfigure}
  \begin{subfigure}{0.32\textwidth}
    \includegraphics[width=\textwidth]{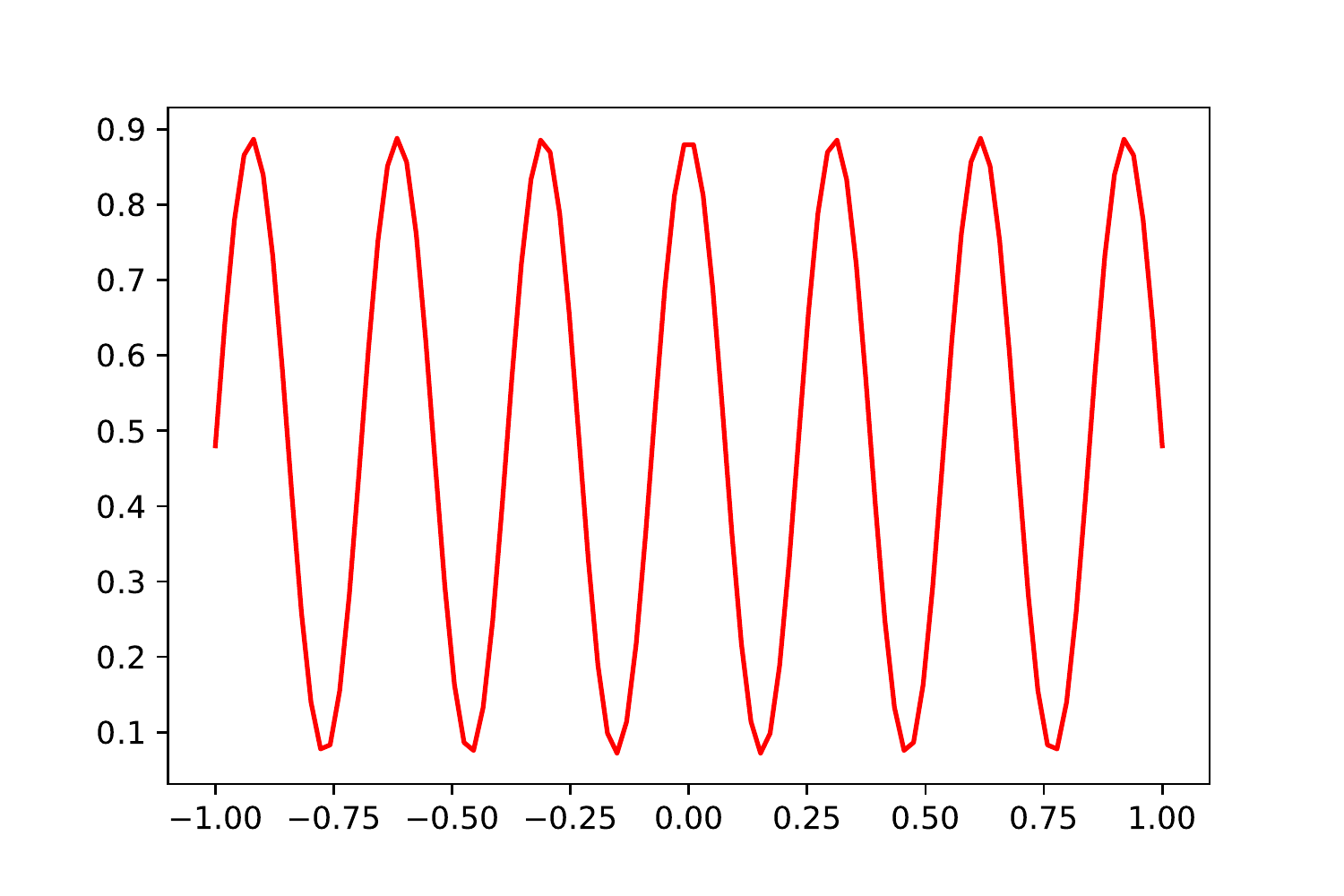}
    \caption{$\mu_3$\label{fig:1Dmu3}}
  \end{subfigure}
  \caption{Densities of the marginal laws tested for 1D numerical tests.
  \label{fig:1Dlaws}}
\end{figure}

\paragraph{Test functions.} The test functions $(\phi_n)_{1\leq n \leq N}$ used are Legendre Polynomials with the following scaling
\begin{equation}
  \phi_n = \frac{\sqrt{2n+\frac 1 2}}{n+1} P_n,
\end{equation}
where $P_n$ is the Legendre Polynomial of degree $n$.
As the marginal laws considered have their support in $[-1,1]$, we chose the Legendre polynomials for their orthogonality property. Besides, by using polynomials, the matrix $\nabla \Gamma(X)$ is related to a Vandermonde matrix, the invertibility of which (crucial to enforce the constraints by Algorithm \ref{algo:EnforceConstraints} or the Runge-Kutta method) is ensured as long as particles are spread on more than $N$ locations.

\paragraph{Cost.} We use in all experiments the regularized  Coulomb  cost function~\eqref{cost_1D} with $\epsilon = 10^{-1}$.

\paragraph{Weight functions.} Two different choices of weight functions $f$ are studied in the numerical experiments presented below: the squared weight function $f: \mathbb{R} \ni a \mapsto a^2$ and the exponential weight function
$f: \mathbb{R} \ni a \mapsto e^{-a}$. Although we do not have strong criteria to chose a weight function, the intuition behind the squared weight function is that it can behave well regarding the enforcement of the constraints by a Newton method, given that $\overline{\Gamma}$ is then a polynomial. The intuition behind the exponential weight function is that it could slow down the cancellation of the weights of the particles, keeping alive more degrees of freedom for the optimization process.

\subsubsection{Initialization step -- Figure~\ref{fig:init1D} }

The aim of Figure~\ref{fig:init1D} is to plot the decrease of $\|\Gamma^K(Y^K_m)\|_\infty$ as a function of the number of iterations of the Runge-Kutta~3 method presented in Section~\ref{sec:init}, in a test case where $M =5$.
We numerically observe here that, as expected, as $N$ increases, the number of iterations needed by the Runge-Kutta procedure to reach convergence increases.
Besides, we observe that the additional degrees of freedom of the cases using weight functions allow a faster initial optimization -- yet not heavily pronounced, as well as an initialization slightly faster for the squared weight function compared to the exponential one.


\begin{figure}[t]
  \centering
  \begin{subfigure}{0.48\textwidth}
    \includegraphics[width=\textwidth]{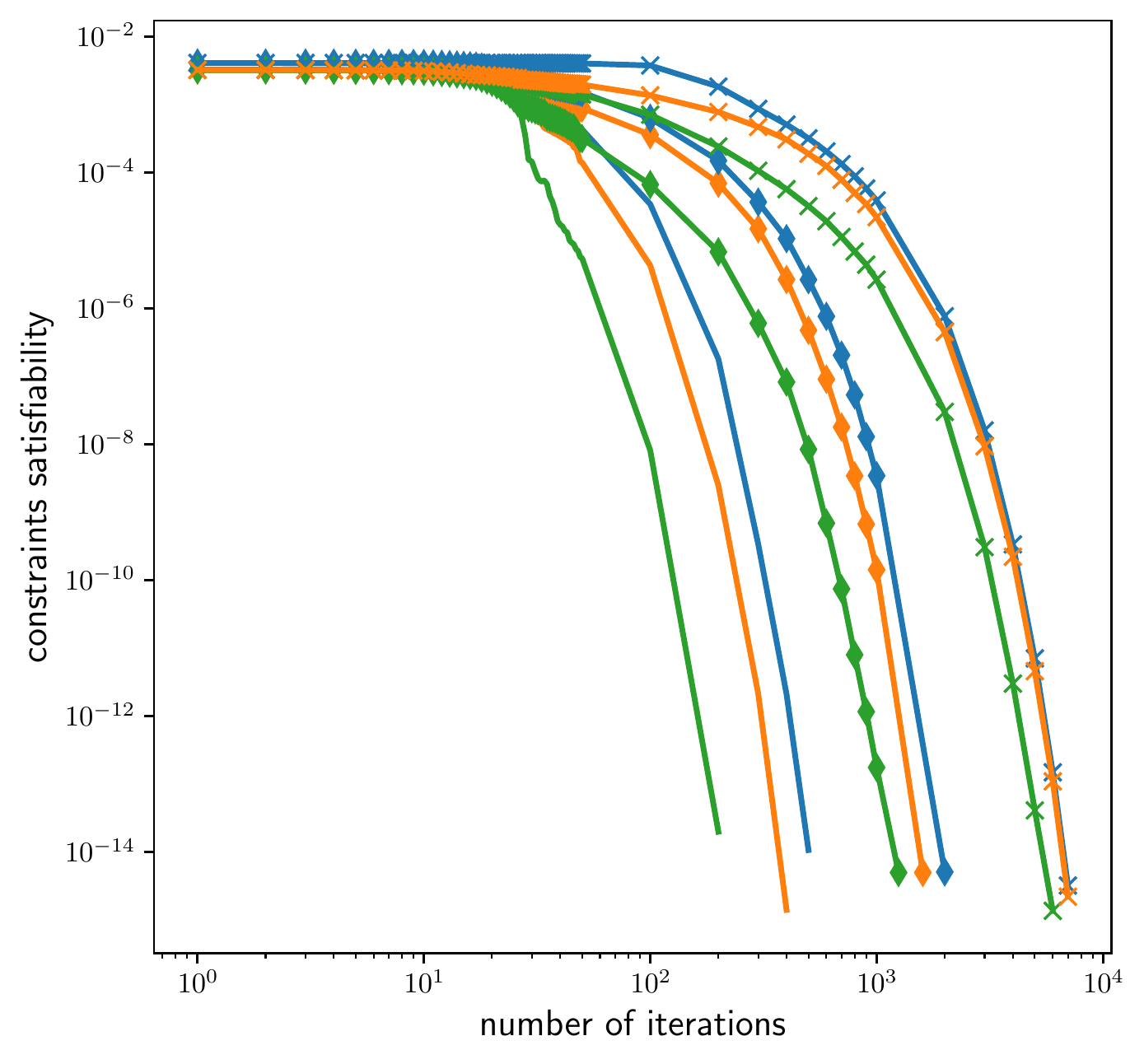}
    \caption{$\mu_2$ \label{fig:init1Dmu2}}
  \end{subfigure}
  \begin{subfigure}{0.48\textwidth}
    \includegraphics[width=\textwidth]{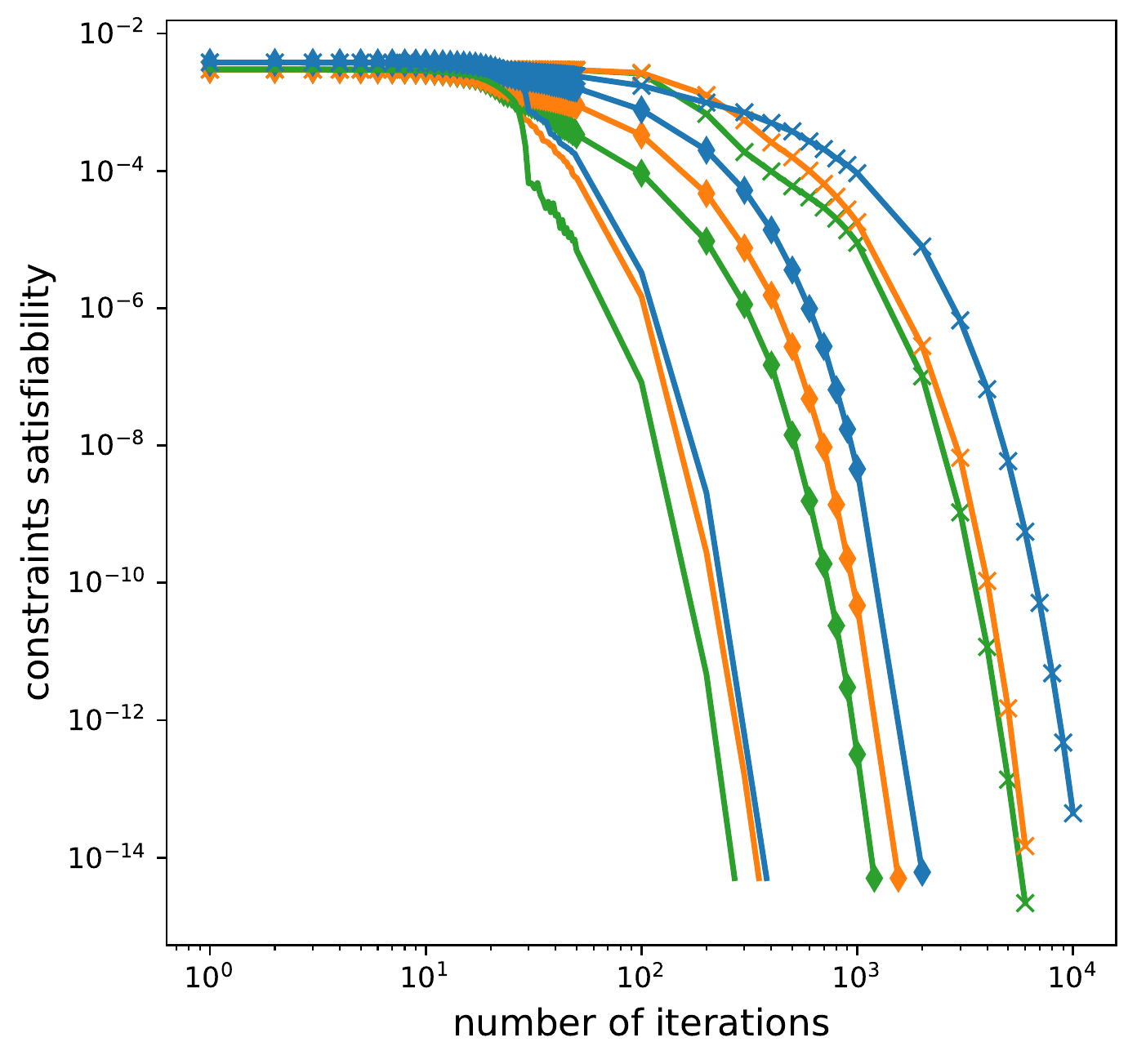}
    \caption{$\mu_3$ \label{fig:init1Dmu3}}
  \end{subfigure}
  \caption{Evolution of $\| \Gamma^K(Y^K_m) \|_\infty$ for different weight functions as a function of the number of iterations $m$ of the Runge-Kutta~3 procedure. Tests were performed with $M=5$, $K=10000$. Blue curves uses fixed weights, orange curves uses an exponential weight function and green curves a squared weight function. No marker is for $N = 10$, a diamond marker for $N=20$ and a ``+" marker for $N= 40$. Caratheodory-Tchakaloff subsampling gave initial values of $1.11 \times 10^{-16}$ ( $3.83 \times 10^{-16}$, $3.02 \times 10^{-16}$) for $\mu_2$, $N = 10$ (resp. $N = 20$, $N = 40$) and $3.33 \times 10^{-16}$ ( $1.28 \times 10^{-16}$, $4.66 \times 10^{-16}$) for $\mu_3$, $N = 10$ (resp. $N = 20$, $N = 40$). \label{fig:init1D}}
\end{figure}

\subsubsection{Decrease of the cost function -- Figures \ref{fig:cv1Dnoise}, \ref{fig:cv1DN} and \ref{fig:cv1Dmetric}}

The aim of Figures \ref{fig:cv1Dnoise}, \ref{fig:cv1DN} and \ref{fig:cv1Dmetric} is to plot the evolution of $V^K(Y^K_n)$ (or $\overline{V}^K(\overline{Y}^K_n)$) as a function of $n$ the number of iterations of the constrained overdamped Langevin algorithm presented in Section~\ref{sect:numericaldiscretization}
for various values of $N$, various weight functions, values of $\beta_0$ and $\mathrm{NoiseDecrease}$ functions, and using or not a subsampling at initialization.
We observe in Figure \ref{fig:cv1Dnoise} that decreasing the noise as the squareroot of the number of iterations $n$ converges faster than keeping it constant, and that keeping $\beta_0 = 0$ is the fastest. In Figure \ref{fig:cv1DN} we remark that the higher $N$ the slower the optimization (with the particular case of $\mu_3$, $N=20$ with the squared weight function which does not converge in 20000 iterations), and that cases initialized by Caratheodory-Tchakaloff subsampling tend to start with a higher cost. In Figure \ref{fig:cv1Dmetric}, we observe that with $K = 10000$ particles, considering fixed or variable weights does not strongly change the speed of convergence (but for the case $\mu_3$, $N=20$ with the squared weight function mentioned above). However, using variable weights with $K = 100$ particles seems to be the fastest set of parameters.

\begin{figure}[htp]
  \centering
  \begin{subfigure}{0.45\textwidth}
    \includegraphics[width=\textwidth]{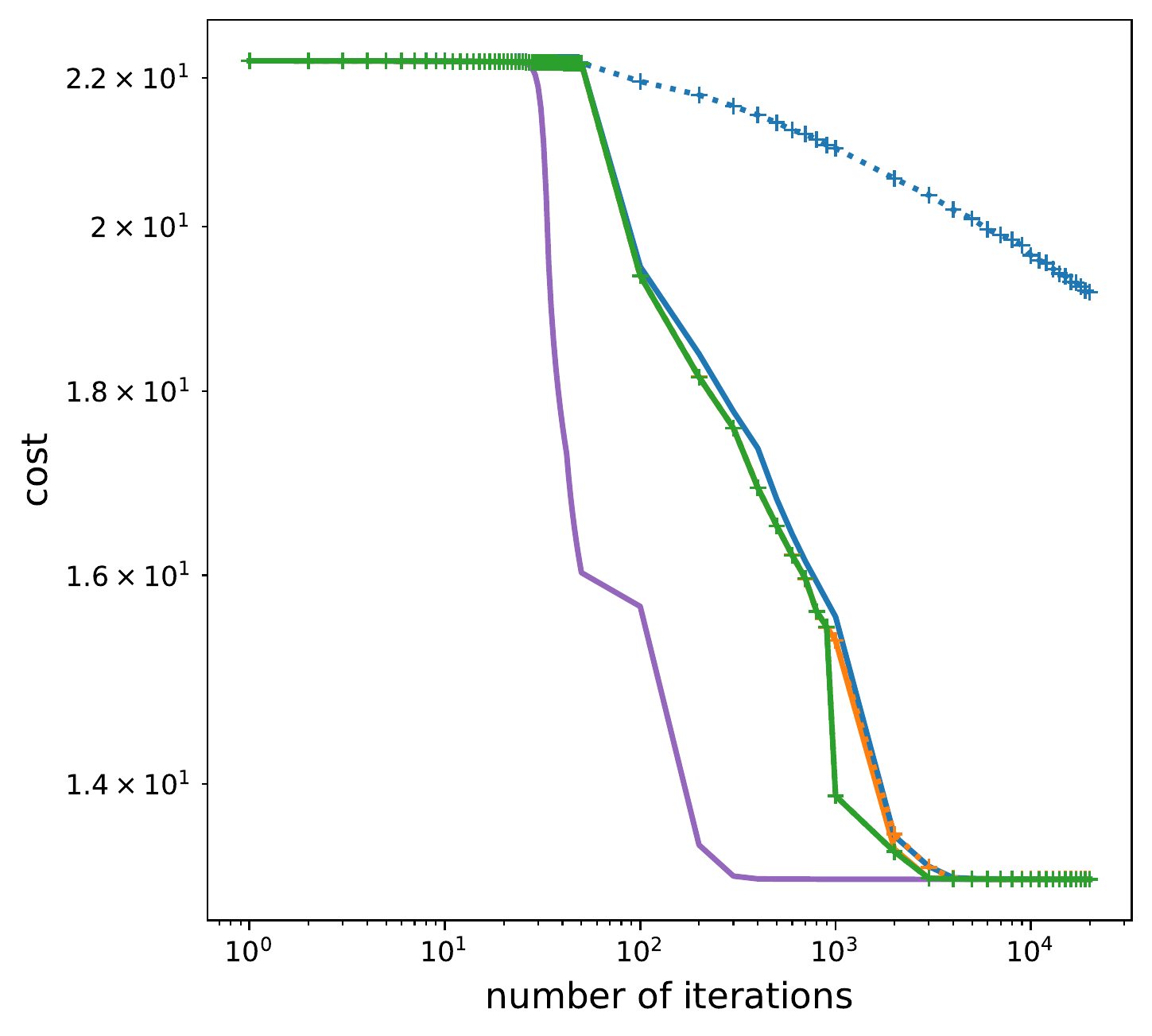}
    \caption{$\mu_2$, fixed weights \label{fig:cvnoisecst1Dmu2}}
  \end{subfigure}
  \begin{subfigure}{0.45\textwidth}
    \includegraphics[width=\textwidth]{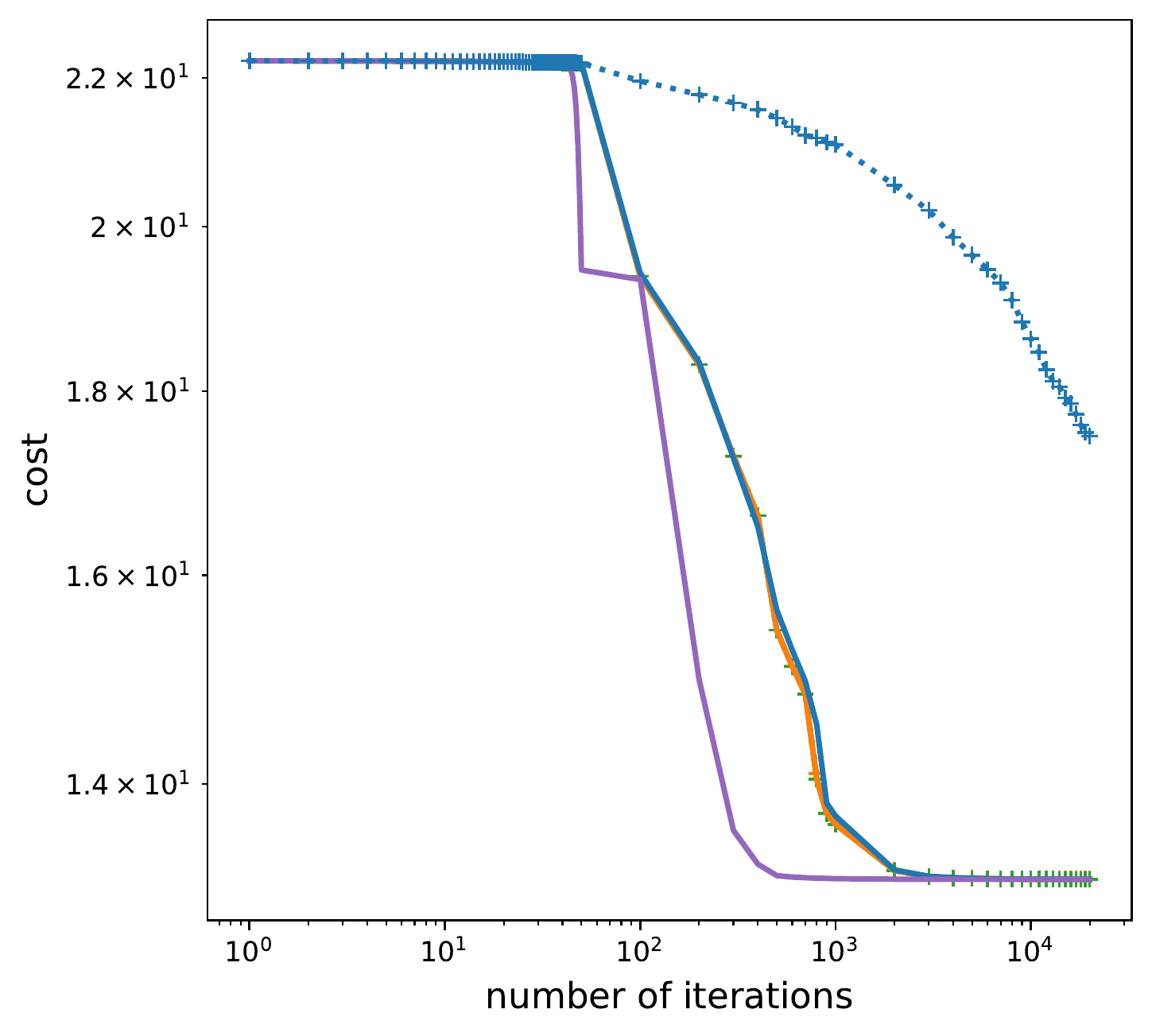}
    \caption{$\mu_2$, squared weight function \label{fig:cvnoisesq1Dmu2}}
  \end{subfigure}
  \begin{subfigure}{0.45\textwidth}
    \includegraphics[width=\textwidth]{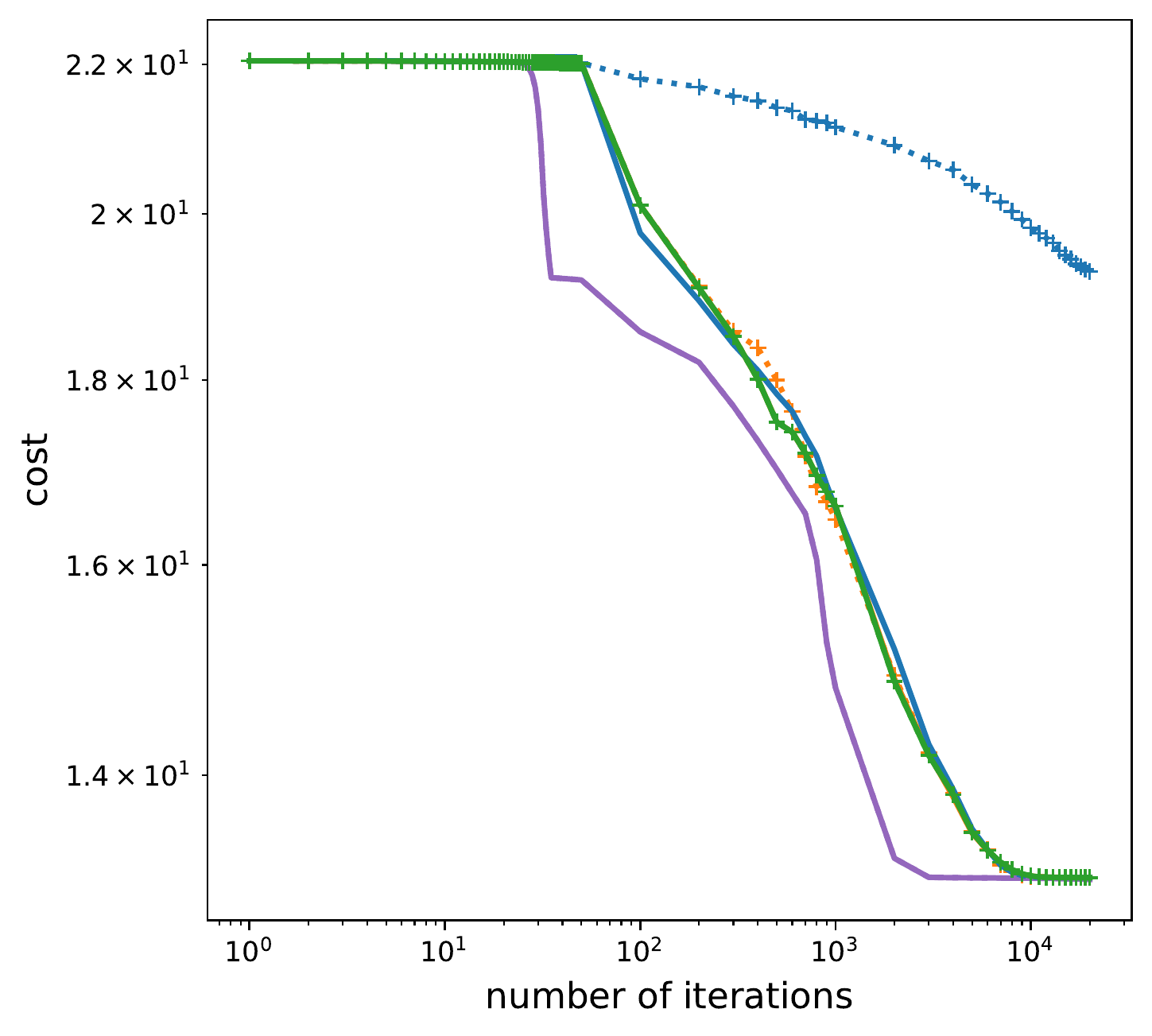}
    \caption{$\mu_3$, fixed weights \label{fig:cvnoisecst1Dmu3}}
  \end{subfigure}
  \begin{subfigure}{0.45\textwidth}
    \includegraphics[width=\textwidth]{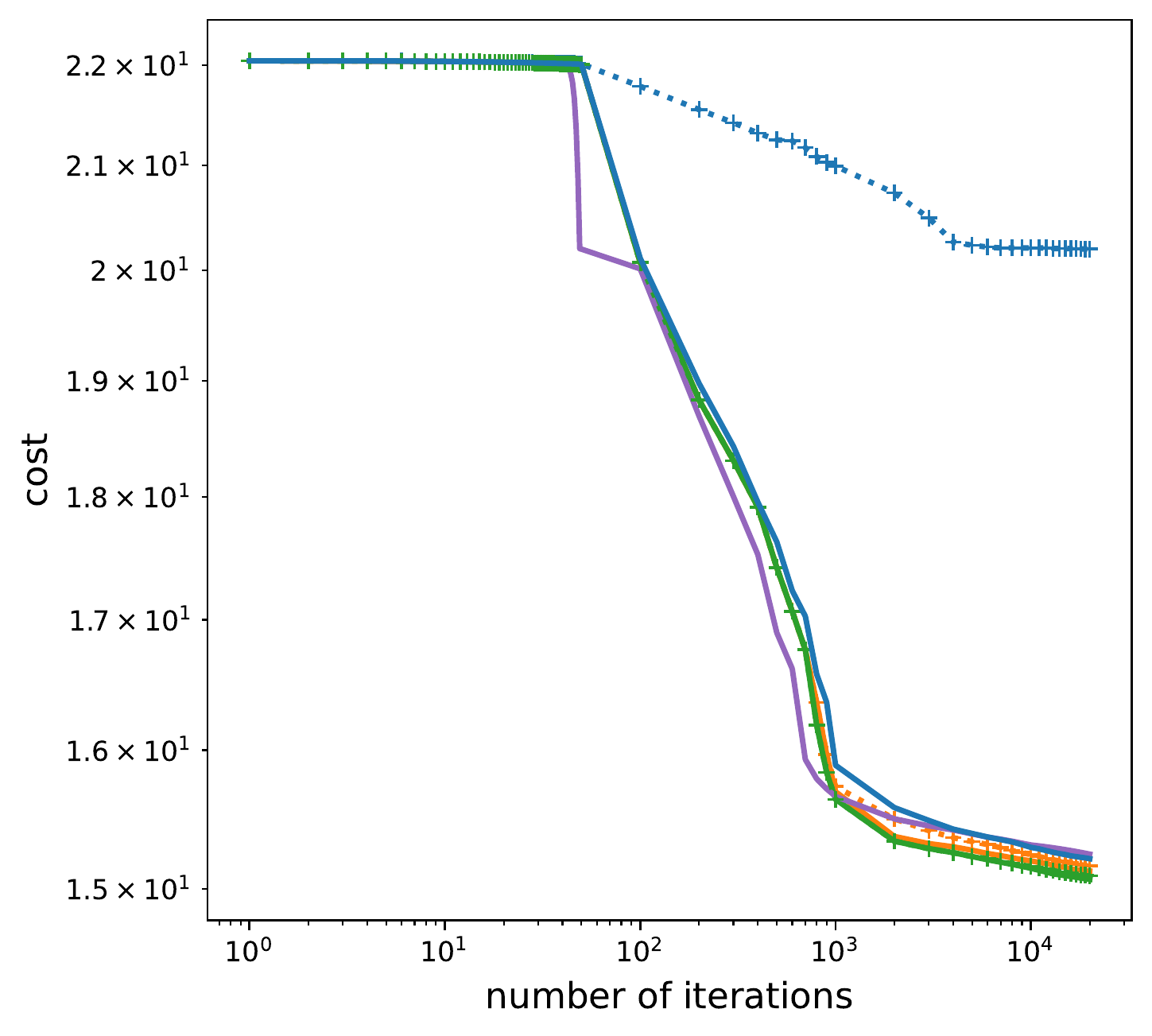}
    \caption{$\mu_3$, squared weight function \label{fig:cvnoisesq1Dmu3}}
  \end{subfigure}
  \caption{Evolution of the cost as a function of the number of iterations $n$ for various weight functions and values of $\beta_0$, for $\mu_2$ and $\mu_3$. Tests were performed with $M=5$, $N = 20$, $K=10000$ and $\Delta t_0 = 10^{-3}$. Blue curves are for $\beta_0 = 10^{-1.5}$, orange curves for $10^{-3.5}$, green curves for $10^{-5.5}$ and purple curves for $\beta_0 = 0$. Solid lines have a decrease of the noise in the squareroot of time whereas dotted lines with a ``+" marker have no decrease of the noise. \label{fig:cv1Dnoise}}
\end{figure}

\begin{figure}[htp]
  \centering
  \begin{subfigure}{0.45\textwidth}
    \includegraphics[width=\textwidth]{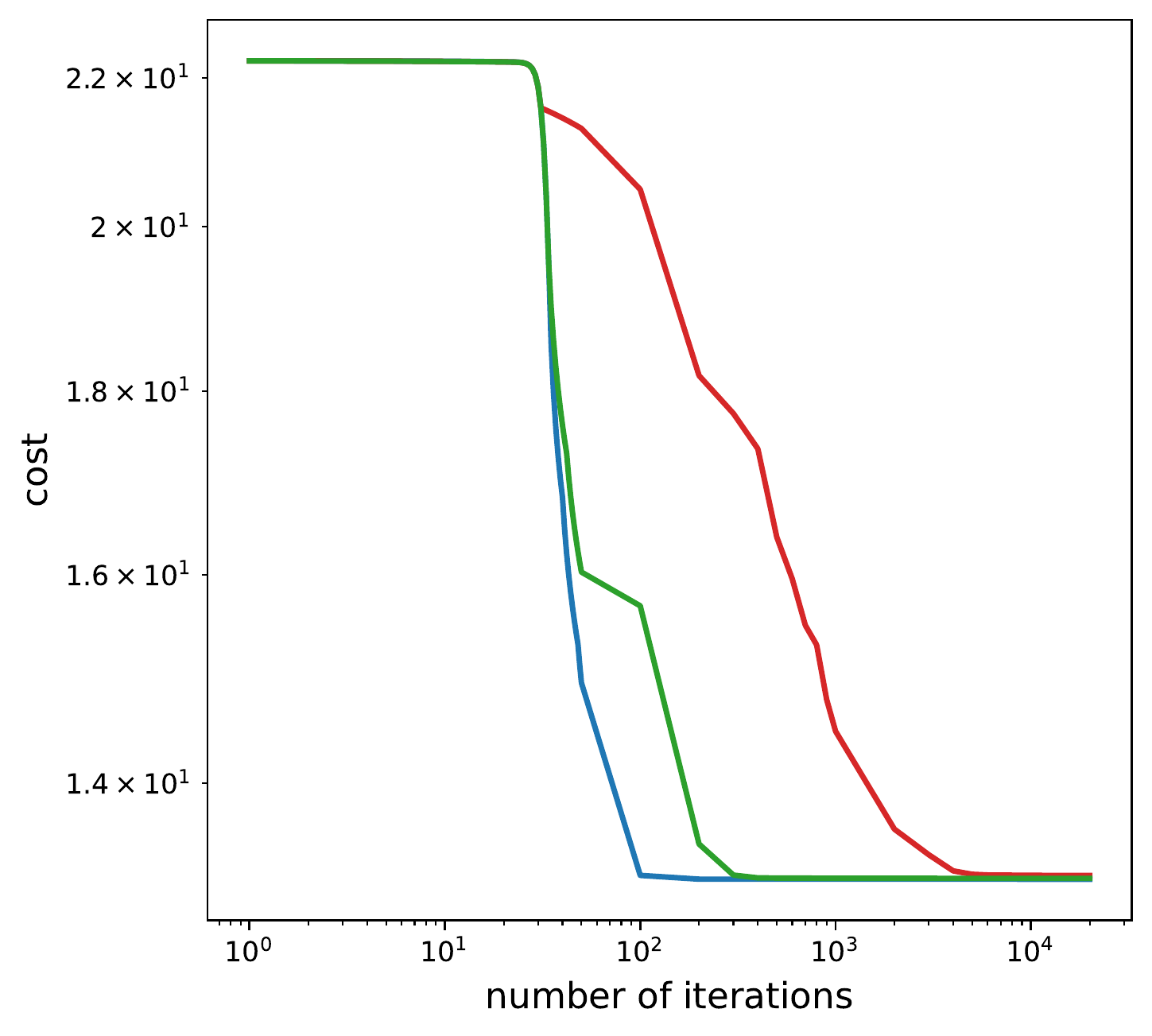}
    \caption{$\mu_2$, fixed weights \label{fig:cvNcst1Dmu2}}
  \end{subfigure}
  \begin{subfigure}{0.45\textwidth}
    \includegraphics[width=\textwidth]{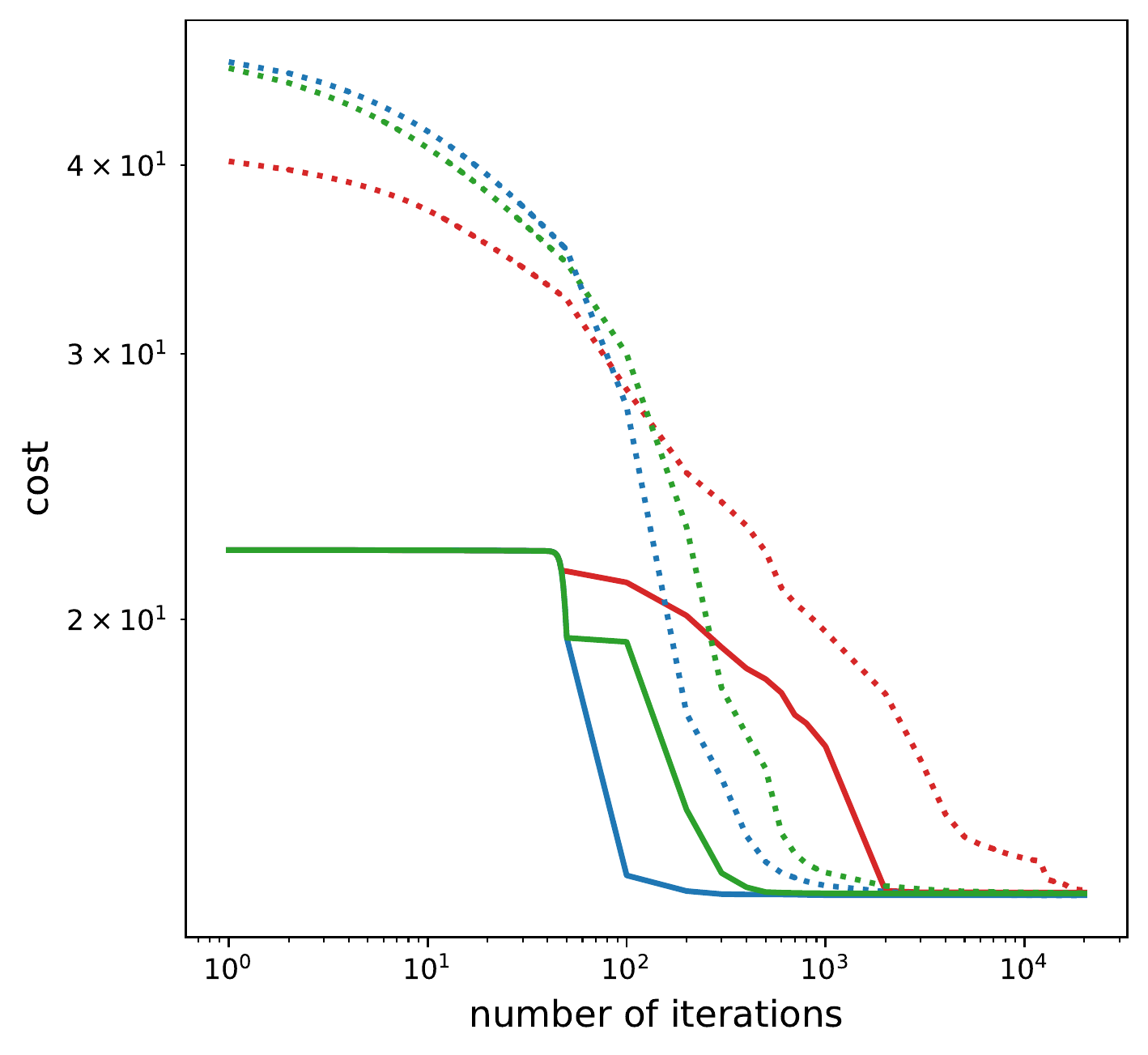}
    \caption{$\mu_2$, squared weight function \label{fig:cvNsq1Dmu2}}
  \end{subfigure}
  \begin{subfigure}{0.45\textwidth}
    \includegraphics[width=\textwidth]{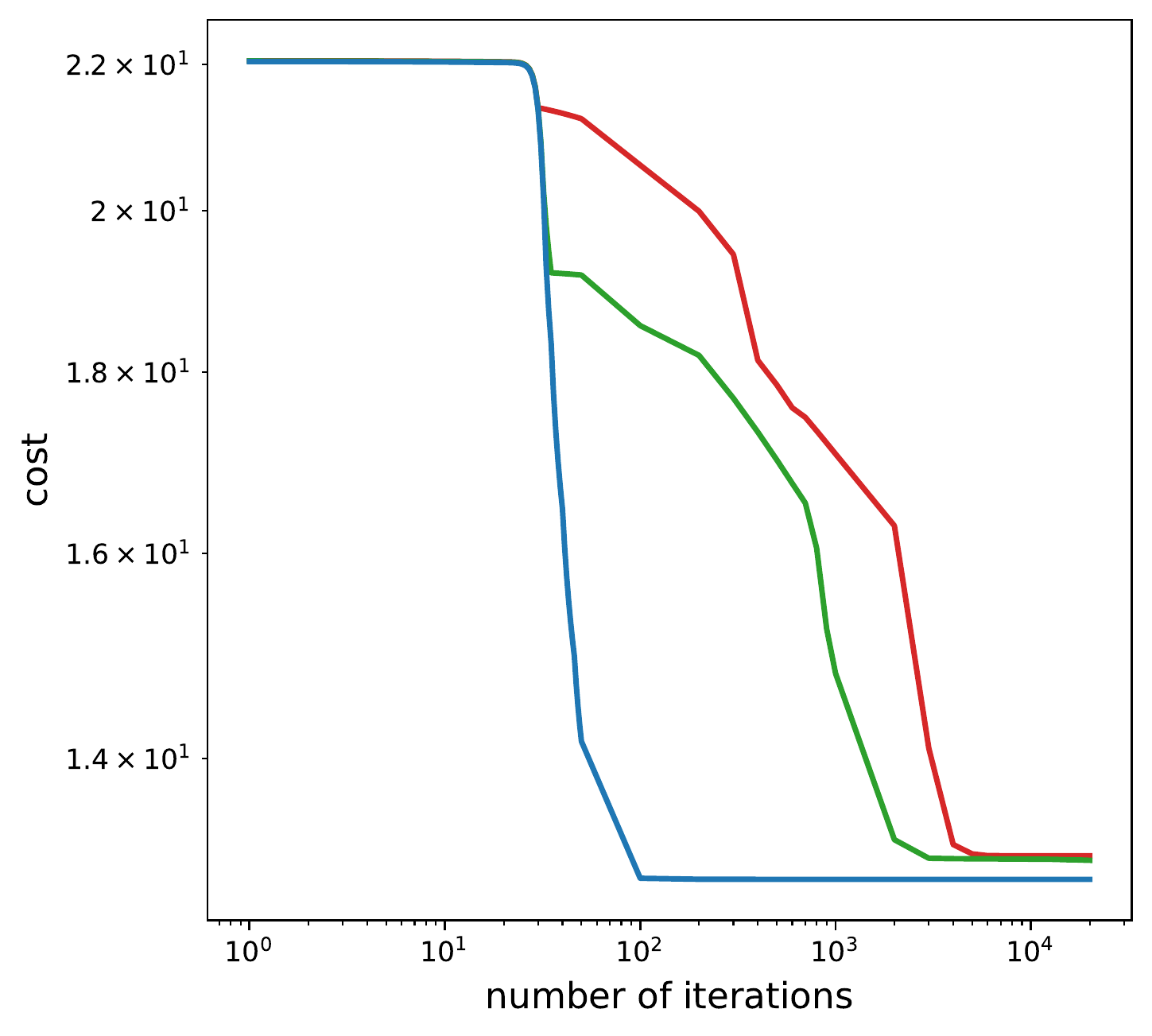}
    \caption{$\mu_3$, fixed weights \label{fig:cvNcst1Dmu3}}
  \end{subfigure}
  \begin{subfigure}{0.45\textwidth}
    \includegraphics[width=\textwidth]{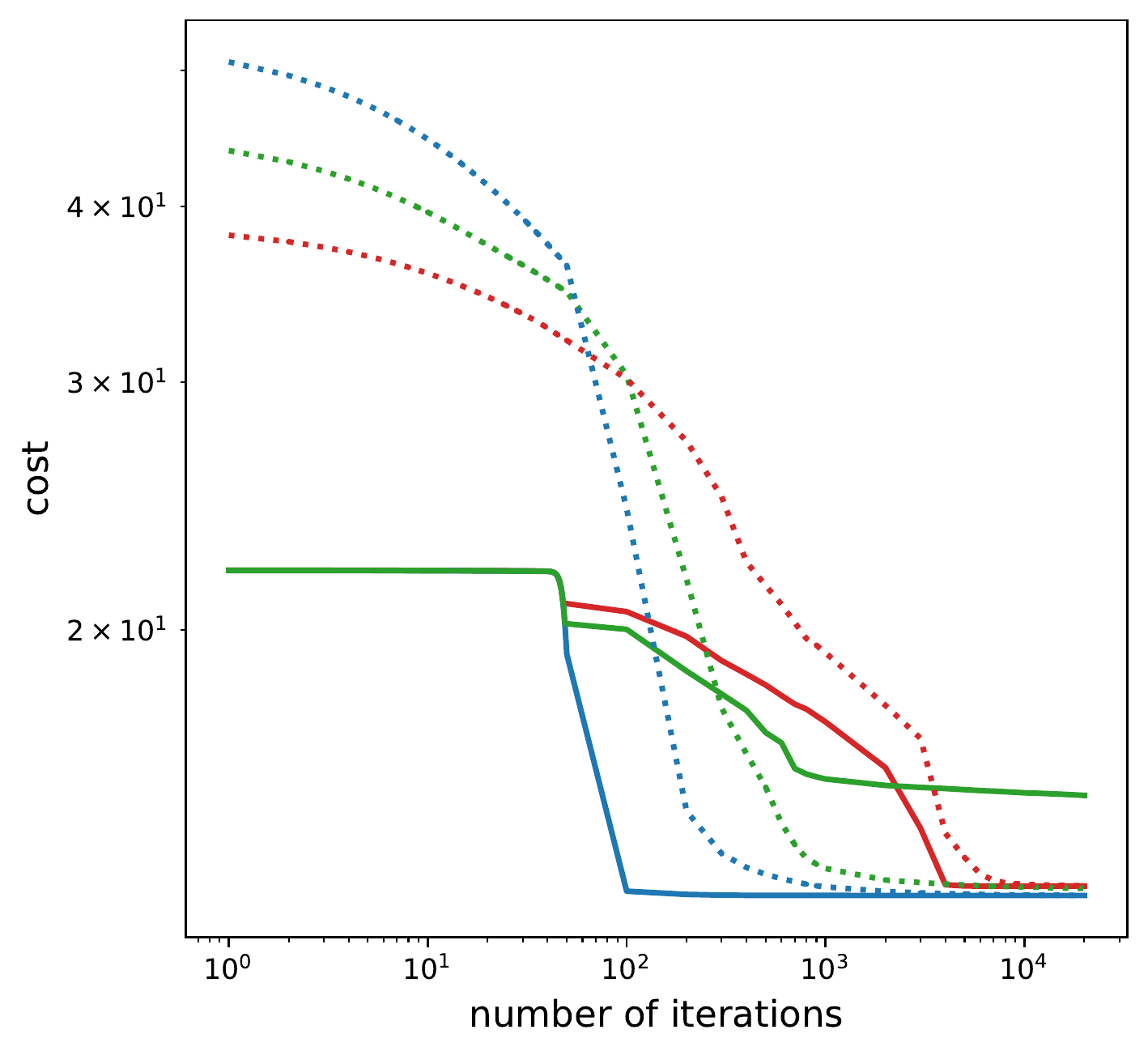}
    \caption{$\mu_3$, squared weight function \label{fig:cvNsq1Dmu3}}
  \end{subfigure}
  \caption{Evolution of the cost as a function of the number of iterations $n$ for various weight functions and values of $N$, for $\mu_2$ and $\mu_3$. Tests were performed with $M=5$, $\beta_0 = 0$, $K=10000$ and $\Delta t_0 = 10^{-3}$. Blue curves for $N =10$, green curves for $N=20$ and red curves for $N = 40$. Dotted lines correspond to tests initialized by Caratheodory-Tchakaloff subsampling whereas tests solid lines correspond to tests initialized by Runge-Kutta~3 method.\label{fig:cv1DN}}
\end{figure}

\begin{figure}[htp]
  \centering
  \begin{subfigure}{0.45\textwidth}
    \includegraphics[width=\textwidth]{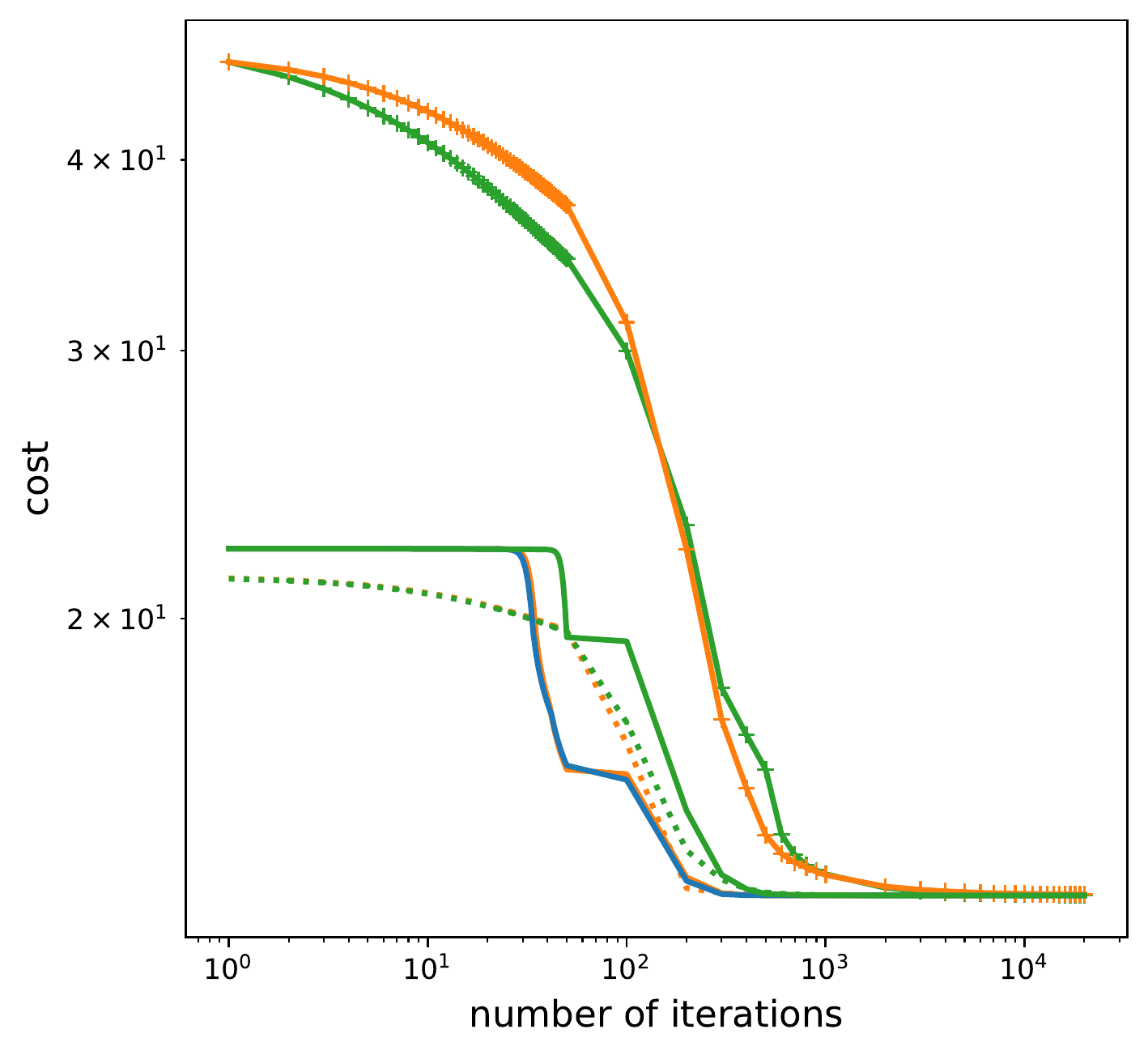}
    \caption{$\mu_2$, \label{fig:cvmetric1Dmu2}}
  \end{subfigure}
  \begin{subfigure}{0.45\textwidth}
    \includegraphics[width=\textwidth]{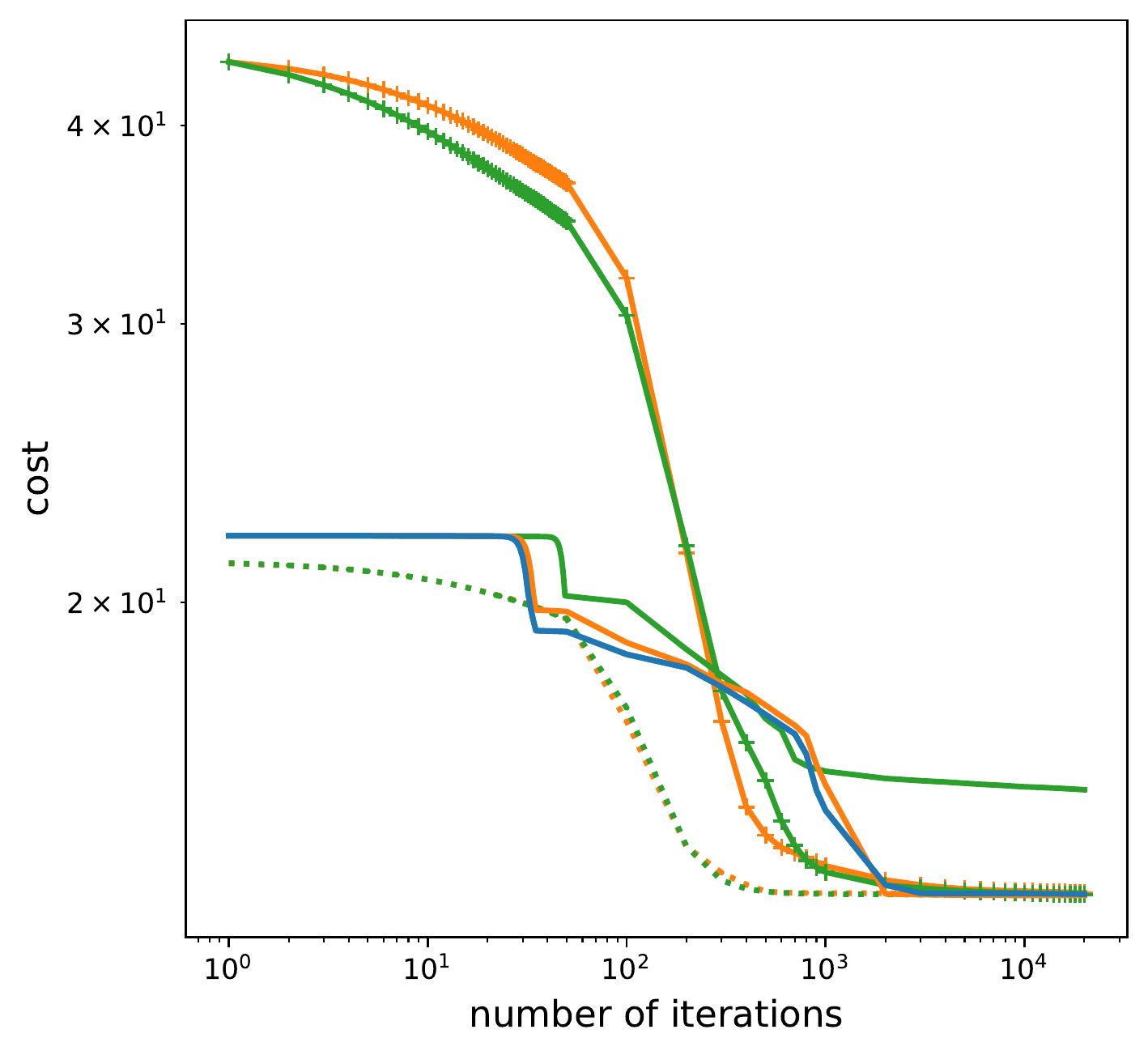}
    \caption{$\mu_3$, \label{fig:cvmetric1Dmu3}}
  \end{subfigure}
  \caption{Evolution of the cost as a function of the number of iterations $n$ for various weight functions, for $\mu_2$ and $\mu_3$. Tests were performed with $M=5$, $N=20$, $\beta_0 = 0$ and $\Delta t_0 = 10^{-3}$. Blue curves uses fixed weights, orange curves uses an exponential weight function and green curves a squared weight function. $K = 10000$ particles for solid lines and $K = 100$ particles for dotted lines. Optimization following a Caratheodory-Tchakaloff subsampling at initialization uses ``+" markers. \label{fig:cv1Dmetric}}
\end{figure}

\begin{figure}[!th]
  \centering
  \includegraphics[width=\textwidth]{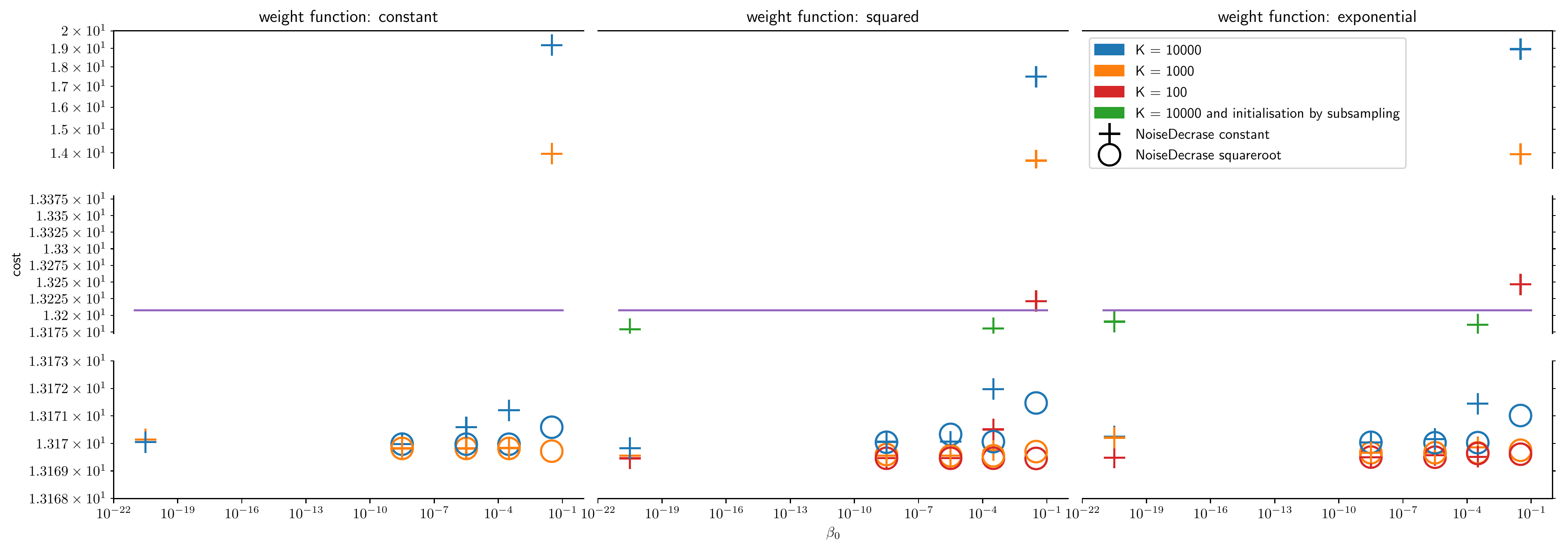}
  \caption{Lowest cost value reached during optimization by the constrained overdamped Langevin algorithm in function of the $\beta_0$, for various weight functions, values of $K$ and choices of $\mathrm{NoiseDecrease}$ functions. The purple line corresponds to the optimal transport cost.
  The marginal law is $\mu_2$, $N=20$, $M=5$, $\Delta t_0 = 10^{-3}$. \label{fig:metricnoiseoptimal1D}}
\end{figure}

\clearpage
\subsubsection{Minimal values of cost -- Figure~\ref{fig:metricnoiseoptimal1D}}

The goal of Figure~\ref{fig:metricnoiseoptimal1D} is to compare the minimal values of the cost obtained by the algorithm for different parameters together with its analytic value.
We observe that considering adaptive weights enable to reach lower optimal costs than with fixed weights, but the relative difference between the approximate minimal cost values is lower than 0.1\%.
When the noise level decreases in the square root of the number of iterations a lower optimal cost can be reached compared to a constant noise level.
In the variable weights cases, the lower $K$ the lower the optimal cost, but when the optimization starts with a Tchakaloff subsampling solution, for which the lowest cost reached is 0.3\% higher than with the Runge-Kutta~3 method.

\subsubsection{Optimal position of particles -- Figures \ref{fig:minimasmu11D}, \ref{fig:minimasmu21D} and \ref{fig:minimasmu31D}}

The aim of Figures \ref{fig:minimasmu11D}, \ref{fig:minimasmu21D} and \ref{fig:minimasmu31D} is to plot the positions of the particles obtained by the numerical procedure presented in Section~\ref{sect:numericaldiscretization} for respectively $\mu_1$, $\mu_2$ and $\mu_3$ and different values of $K$,
$N$, $\beta_0$, initialization methods, and in fixed and variable weights cases.
We numerically observe that the obtained particles are located close to the support of the exact optimal transport plan, and that the higher the value of $N$ the more precise the approximation of this transport map
is~\cite[Theorem 4.1]{alfonsi2019approximation}. Also, when $K= 10000$ and even more when $\beta_0 = 10^{3.5}$, particles are more spreaded around the transport map.

\begin{figure}[!th]
  \centering
  \begin{subfigure}{0.32\textwidth}
    \centering
    \includegraphics[width=\textwidth]{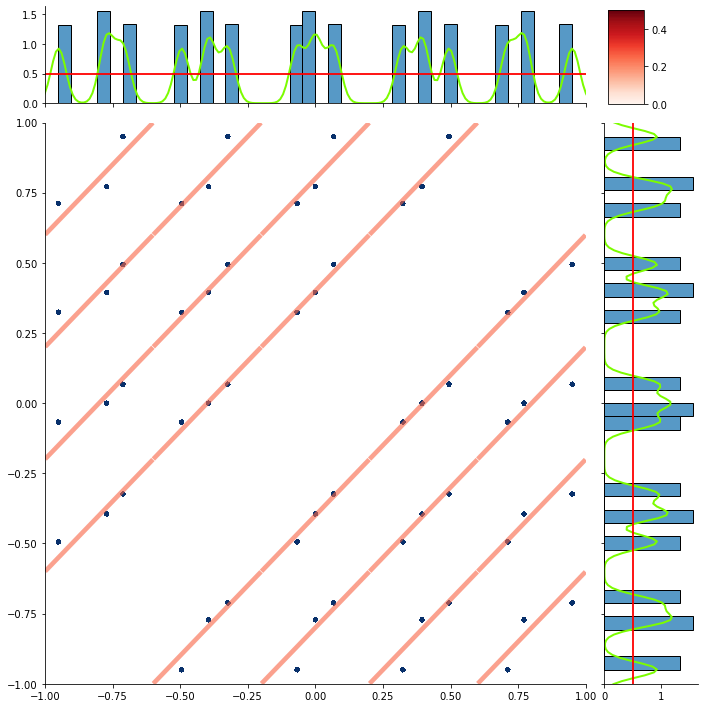}
    \caption{$N = 10$}
  \end{subfigure}
  \begin{subfigure}{0.32\textwidth}
    \centering
    \includegraphics[width=\textwidth]{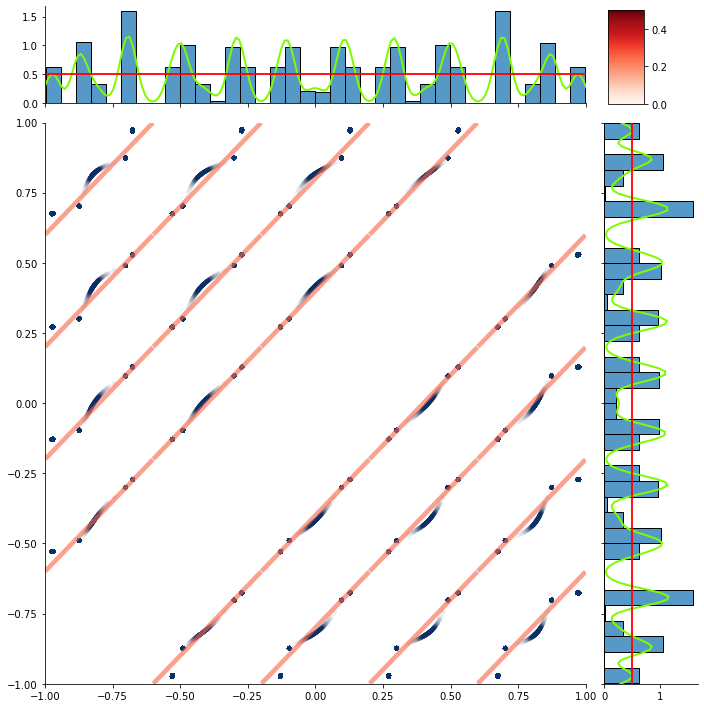}
    \caption{$N = 20$}
  \end{subfigure}
  \begin{subfigure}{0.32\textwidth}
    \centering
    \includegraphics[width=\textwidth]{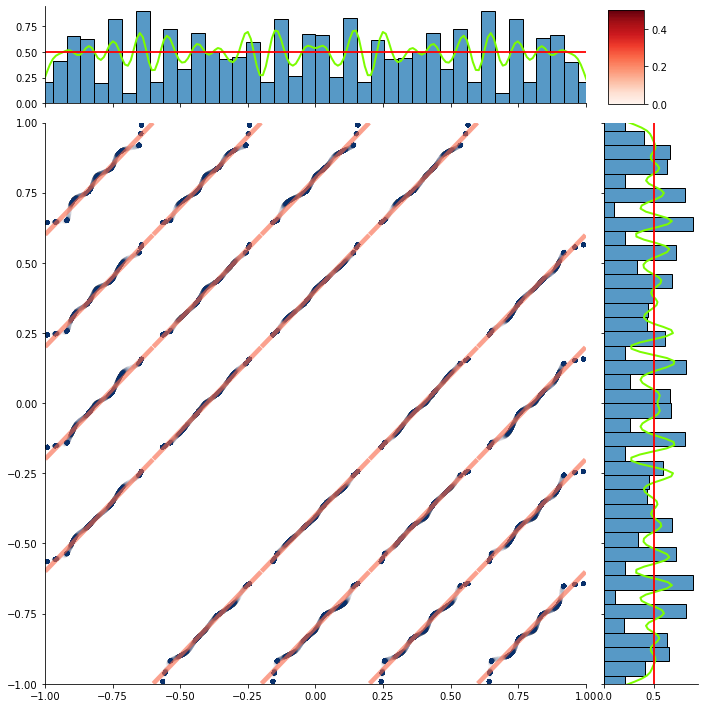}
    \caption{$N = 40$}
  \end{subfigure}
  \caption{
  Optimal transport with $\mu_1$ and $M=5$, $\Delta t_0=  10^{-3}$. In each plot, on the main graph {\small $\frac{1}{M(M-1)} \sum_{k=1}^K \sum_{m \neq m’ = 1}^M w_k \delta_{x_m^k, x_{m’}^k}$} is represented by blue particles. The darker the heavier the particle. Particles have some transparency which allows to see more clearly areas of high concentration. Red curves represent the functions $T^i$ for $i\in \{1,\dots,M-1\}$ defined in Theorem~\ref{thm:colombo}. The higher the density the darker. On side graphs are represented in blue a weighted histogram of the particles, in red the marginal law and in green a normal kernel density estimate based on the weighted particles (with a bandwidth rule based on Scott's rule with $d=0$). \label{fig:minimasmu11D}}
\end{figure}

\begin{figure}[!tp]
  \centering
  \begin{subfigure}{0.24\textwidth}
    \includegraphics[width=\textwidth]{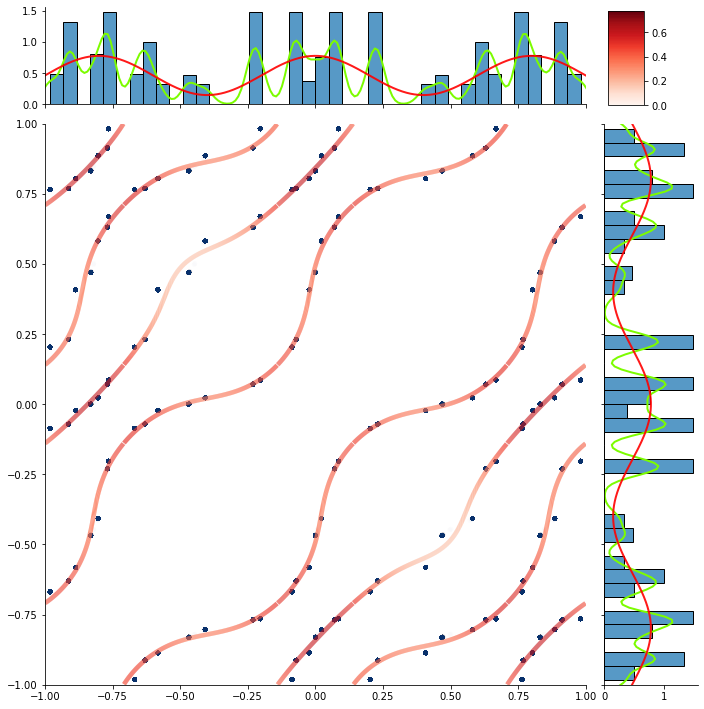}
    \caption{$N=20$, $\beta_0 = 0$, $K=10000$, fixed weights}
  \end{subfigure}
  \begin{subfigure}{0.24\textwidth}
    \includegraphics[width=\textwidth]{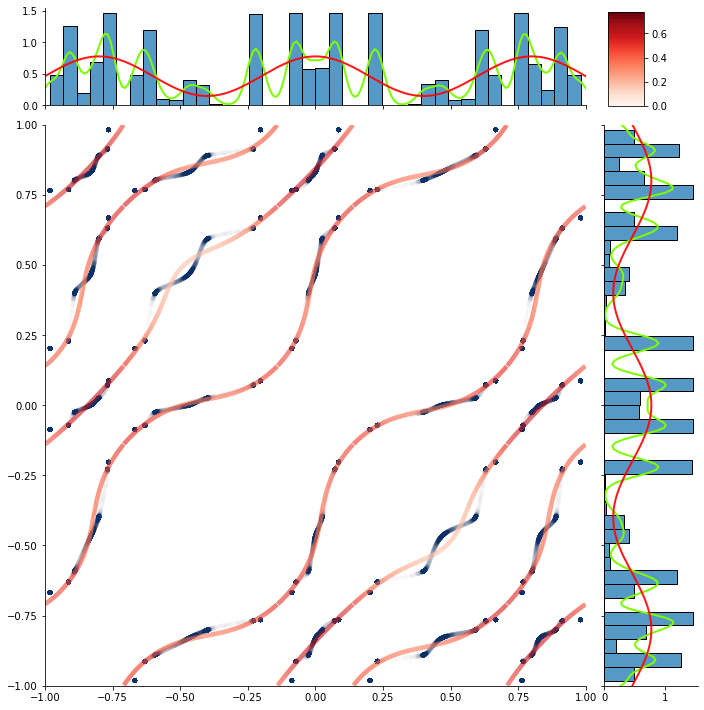}
    \caption{$N=20$, $\beta_0 = 10^{-3.5}$, $K=10000$, fixed weights}
  \end{subfigure}
  \begin{subfigure}{0.24\textwidth}
    \includegraphics[width=\textwidth]{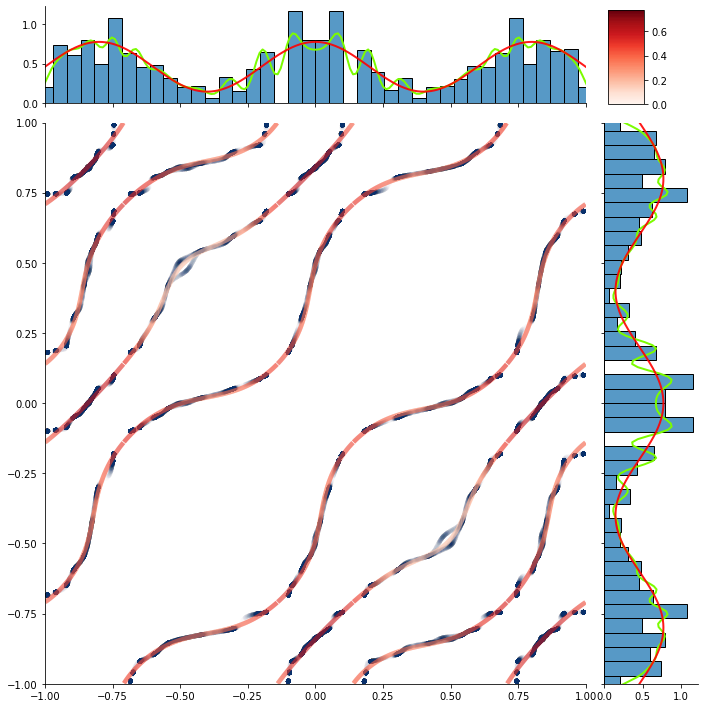}
    \caption{$N=40$, $\beta_0 = 0$, $K=10000$, fixed weights}
  \end{subfigure}
  \begin{subfigure}{0.24\textwidth}
    \includegraphics[width=\textwidth]{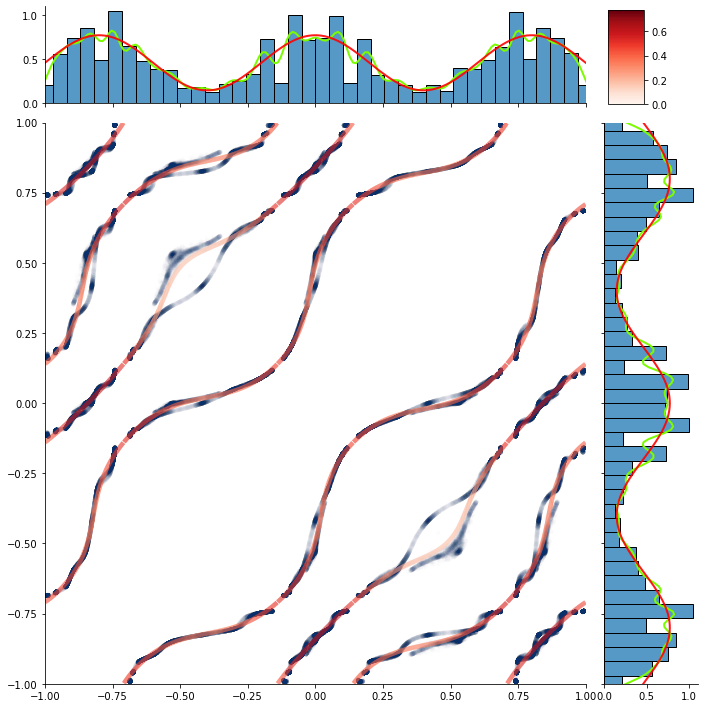}
    \caption{$N=40$, $\beta_0 = 10^{-3.5}$, $K=10000$, fixed weights}
  \end{subfigure}
  \begin{subfigure}{0.24\textwidth}
    \includegraphics[width=\textwidth]{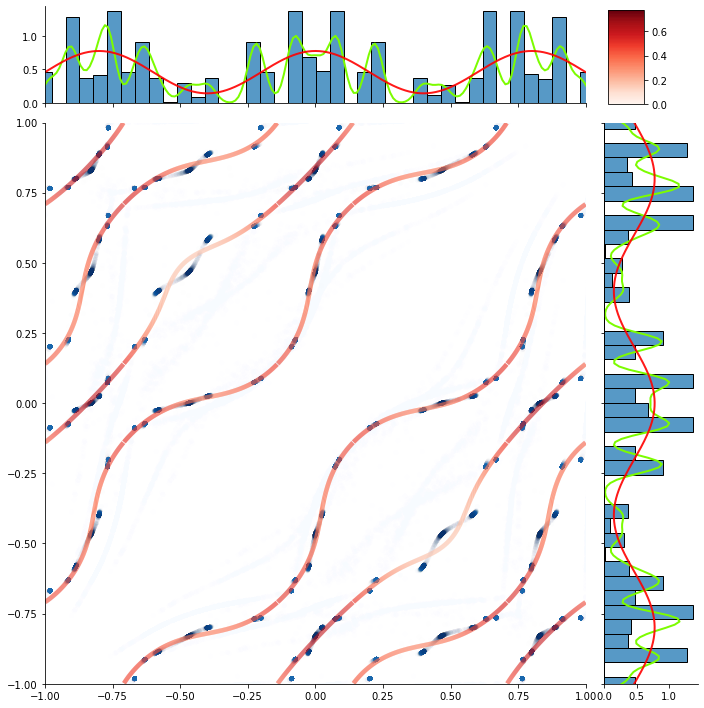}
    \caption{$N=20$, $\beta_0 = 0$, $K=10000$, squared weights}
  \end{subfigure}
  \begin{subfigure}{0.24\textwidth}
    \includegraphics[width=\textwidth]{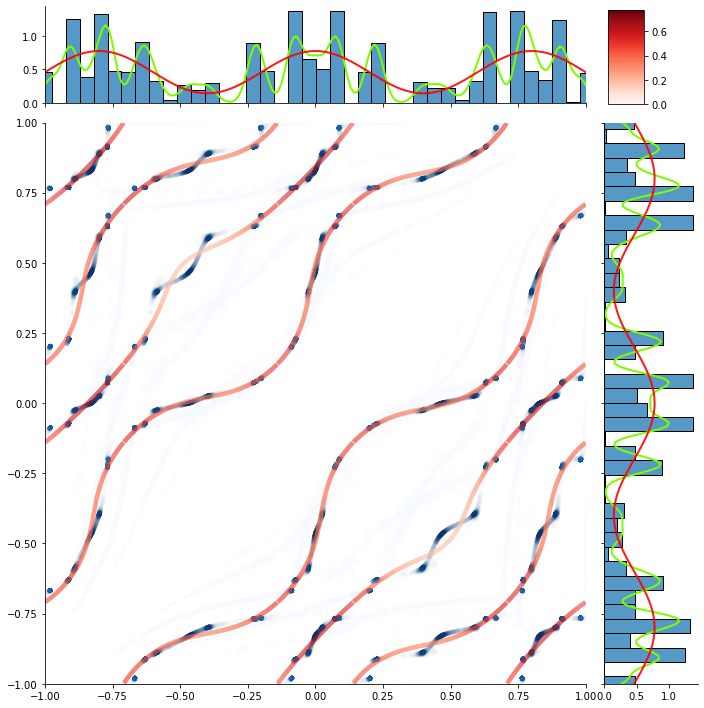}
    \caption{$N=20$, $\beta_0 = 10^{-3.5}$, $K=10000$, squared weights}
  \end{subfigure}
  \begin{subfigure}{0.24\textwidth}
    \includegraphics[width=\textwidth]{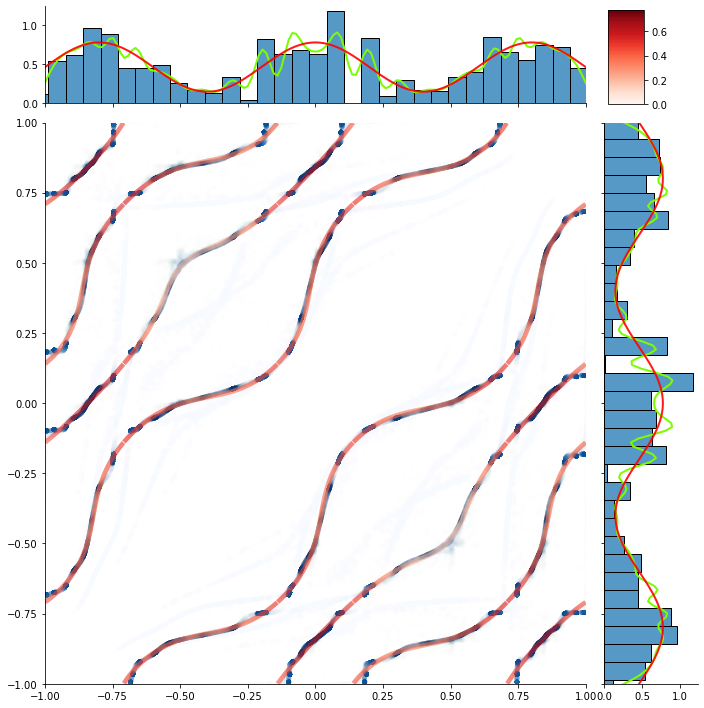}
    \caption{$N=40$, $\beta_0 = 0$, $K=10000$, squared weights}
  \end{subfigure}
  \begin{subfigure}{0.24\textwidth}
    \includegraphics[width=\textwidth]{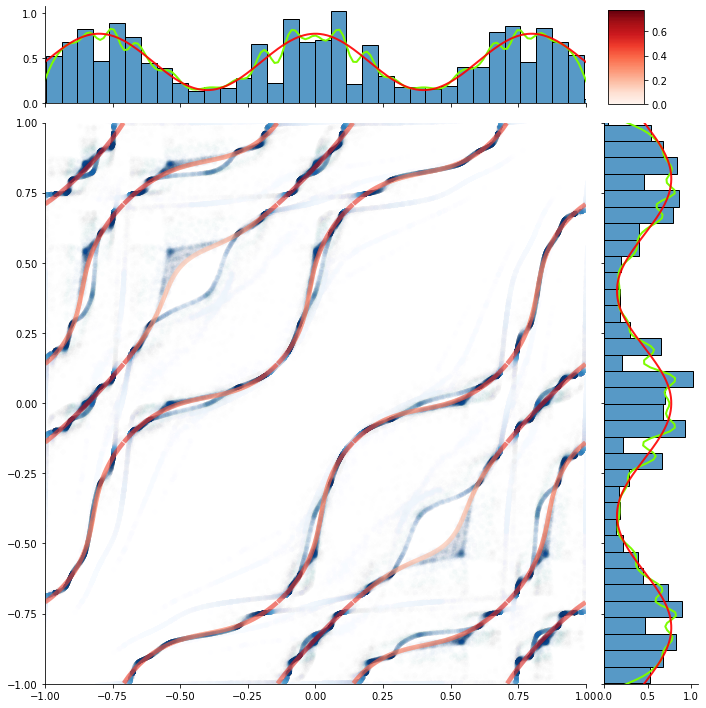}
    \caption{$N=40$, $\beta_0 = 10^{-3.5}$, $K=10000$, squared weights}
  \end{subfigure}
  \begin{subfigure}{0.24\textwidth}
    \includegraphics[width=\textwidth]{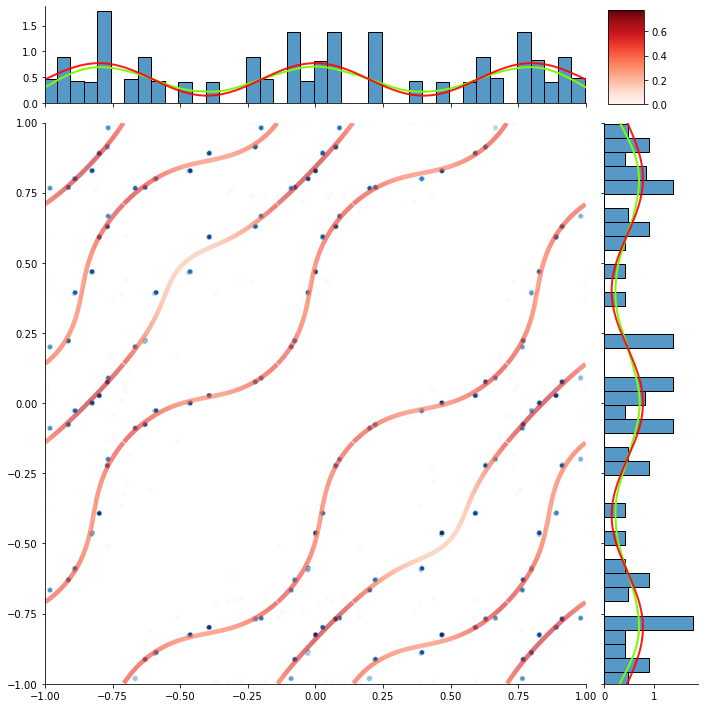}
    \caption{$N=20$, $\beta_0 = 0$, $K=100$, squared weights}
  \end{subfigure}
  \begin{subfigure}{0.24\textwidth}
    \includegraphics[width=\textwidth]{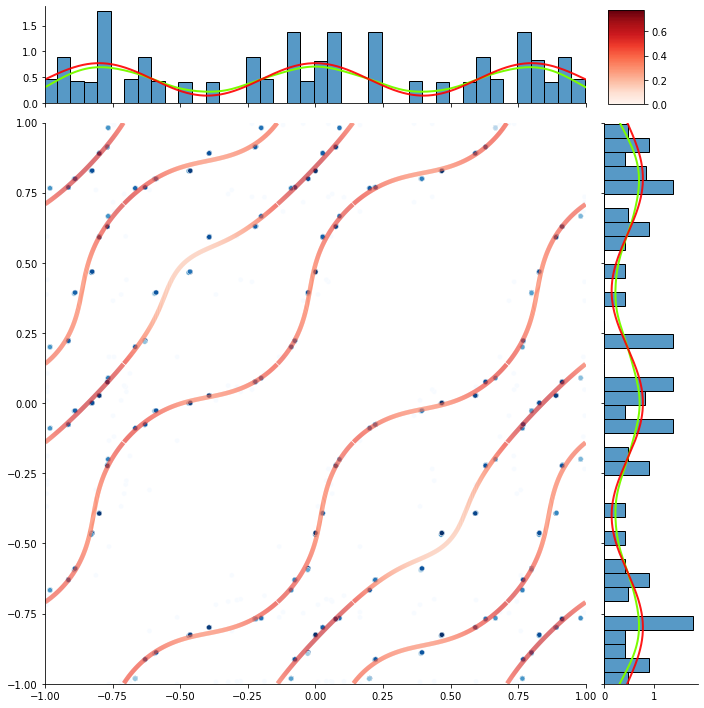}
    \caption{$N=20$, $\beta_0 = 10^{-3.5}$, $K=100$, squared weights}
  \end{subfigure}
  \begin{subfigure}{0.24\textwidth}
    \includegraphics[width=\textwidth]{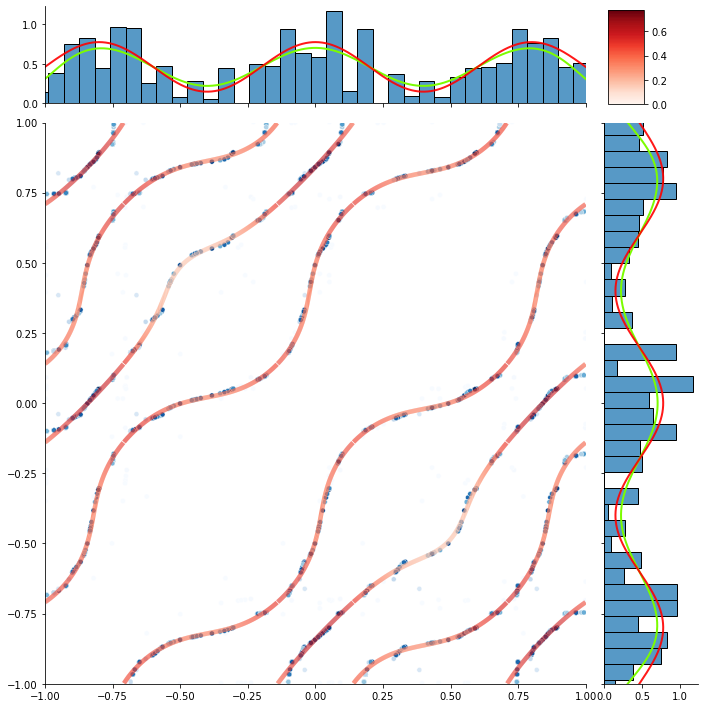}
    \caption{$N=40$, $\beta_0 = 0$, $K=100$, squared weights}
  \end{subfigure}
  \begin{subfigure}{0.24\textwidth}
    \includegraphics[width=\textwidth]{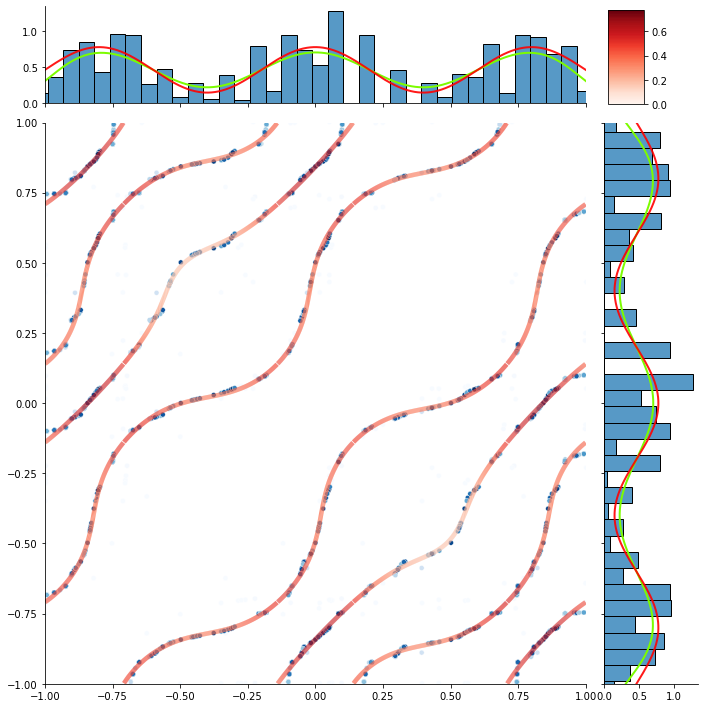}
    \caption{$N=40$, $\beta_0 = 10^{-3.5}$, $K=100$, squared weights}
  \end{subfigure}
  \begin{subfigure}{0.24\textwidth}
    \includegraphics[width=\textwidth]{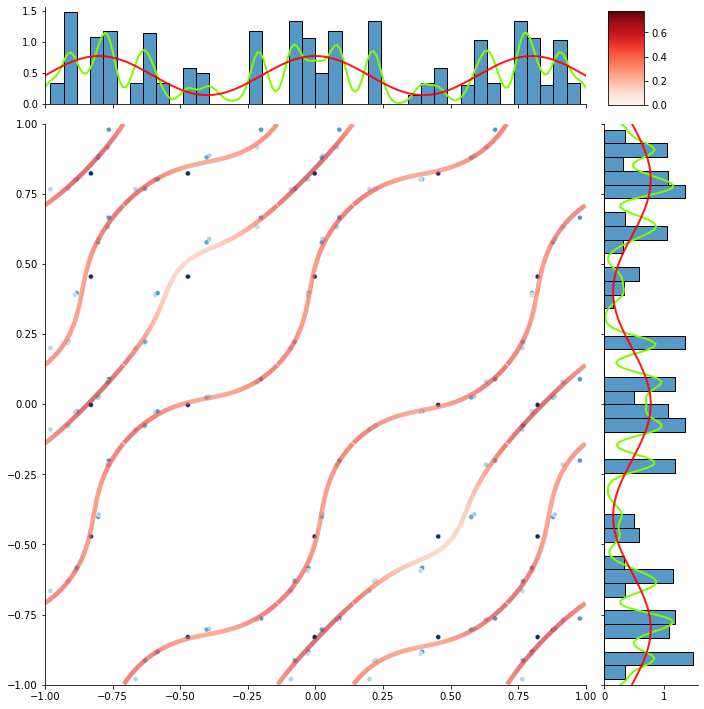}
    \caption{$N=20$, $\beta_0 = 0$, $K=10000$, squared weights, with initial subsampling}
  \end{subfigure}
  \begin{subfigure}{0.24\textwidth}
    \includegraphics[width=\textwidth]{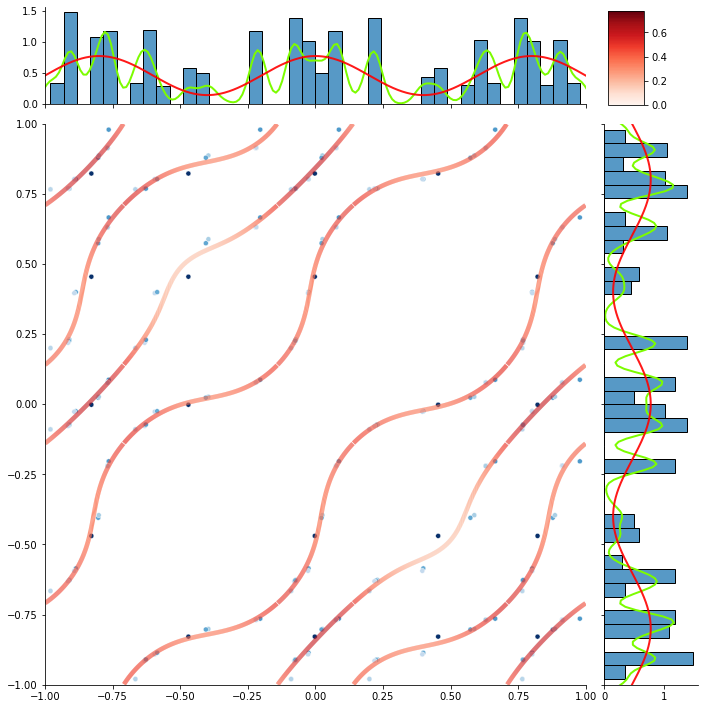}
    \caption{$N=20$, $\beta_0 = 10^{-3.5}$, $K=10000$, squared weights, with initial subsampling}
  \end{subfigure}
  \begin{subfigure}{0.24\textwidth}
    \includegraphics[width=\textwidth]{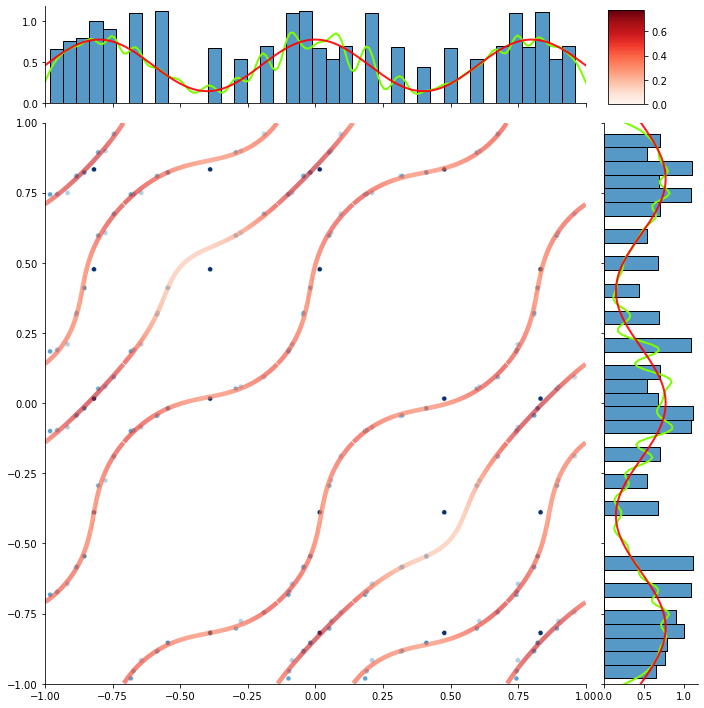}
    \caption{$N=40$, $\beta_0 = 0$, $K=10000$, squared weights, with initial subsampling}
  \end{subfigure}
  \begin{subfigure}{0.24\textwidth}
    \includegraphics[width=\textwidth]{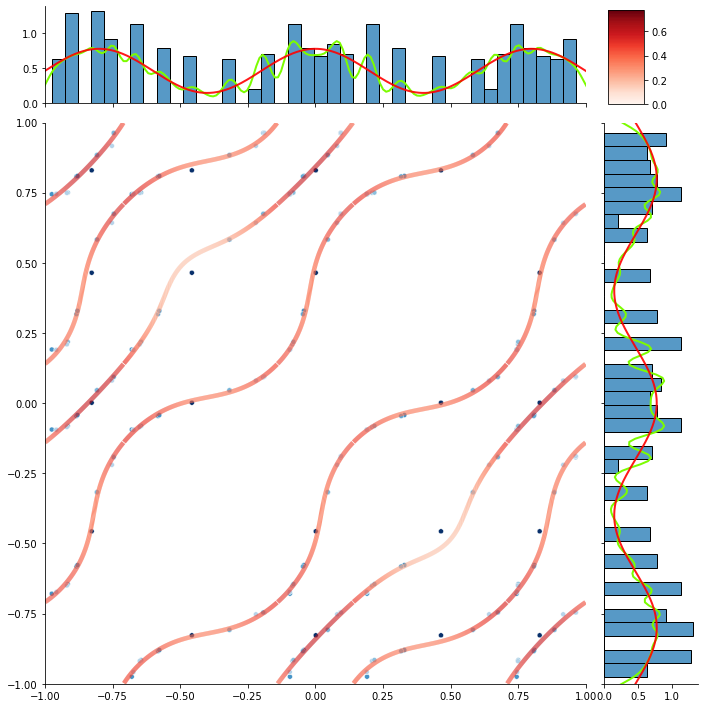}
    \caption{$N=40$, $\beta_0 = 10^{-3.5}$, $K=10000$, squared weights, with initial subsampling}
  \end{subfigure}
  \caption{Optimal transport with $\mu_2$ and $M=5$, $\Delta t_0=  10^{-3}$. In each plot, on the main graph {\small $\frac{1}{M(M-1)} \sum_{k=1}^K \sum_{m \neq m’ = 1}^M w_k \delta_{x_m^k, x_{m’}^k}$} is represented by blue particles. The darker the heavier the particle. Particles have some transparency which allows to see more clearly areas of high concentration. Red curves represent the functions $T^i$ for $i\in \{1,\dots,M-1\}$ defined in Theorem~\ref{thm:colombo}. The higher the density the darker. On side graphs are represented in blue a weighted histogram of the particles, in red the marginal law and in green a normal kernel density estimate based on the weighted particles (with a bandwidth rule based on Scott's rule with $d=0$). \label{fig:minimasmu21D}}
\end{figure}

\begin{figure}[!tp]
  \centering
  \begin{subfigure}{0.24\textwidth}
    \includegraphics[width=\textwidth]{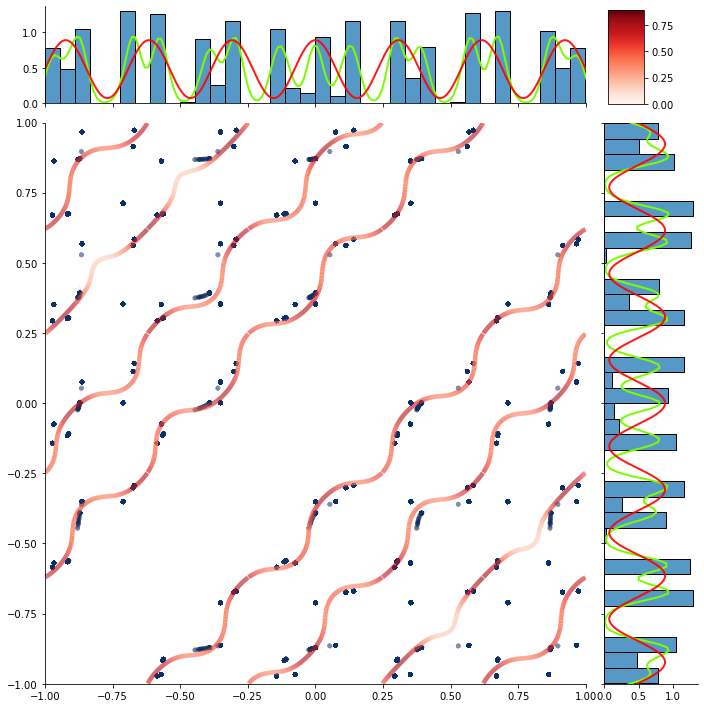}
    \caption{$N=20$, $\beta_0 = 0$, $K=10000$,  fixed}
  \end{subfigure}
  \begin{subfigure}{0.24\textwidth}
    \includegraphics[width=\textwidth]{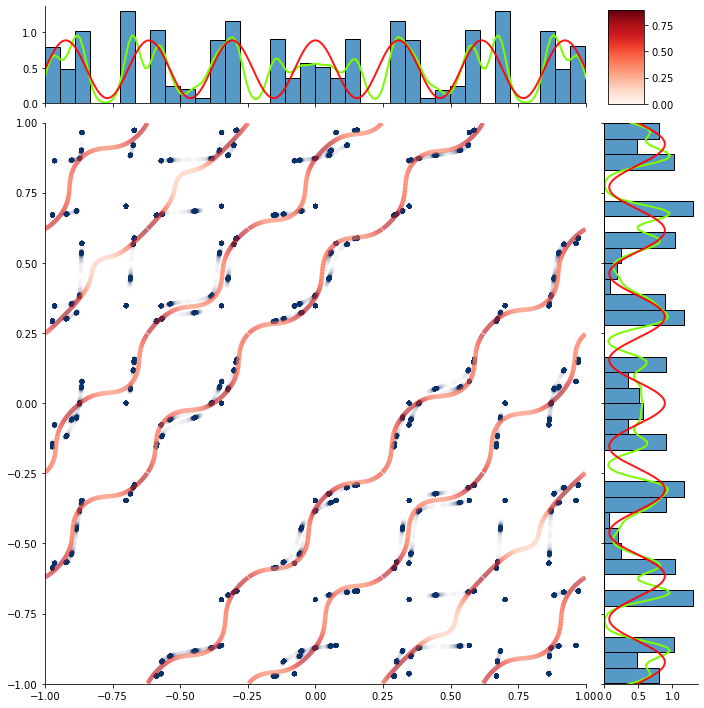}
    \caption{$N=20$, $\beta_0 = 10^{-3.5}$, $K=10000$,  fixed}
  \end{subfigure}
  \begin{subfigure}{0.24\textwidth}
    \includegraphics[width=\textwidth]{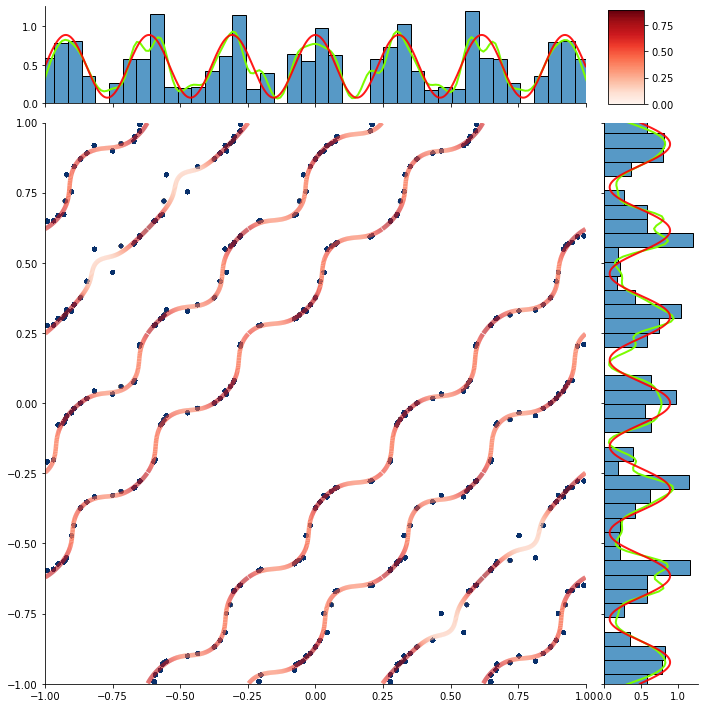}
    \caption{$N=40$, $\beta_0 = 0$, $K=10000$,  fixed}
  \end{subfigure}
  \begin{subfigure}{0.24\textwidth}
    \includegraphics[width=\textwidth]{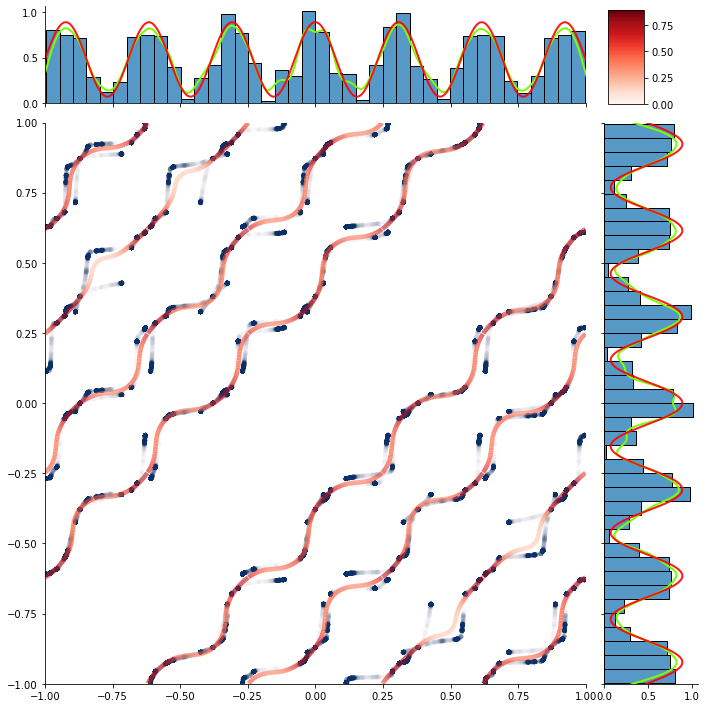}
    \caption{$N=40$, $\beta_0 = 10^{-3.5}$, $K=10000$,  fixed}
  \end{subfigure}
  \begin{subfigure}{0.24\textwidth}
    \includegraphics[width=\textwidth]{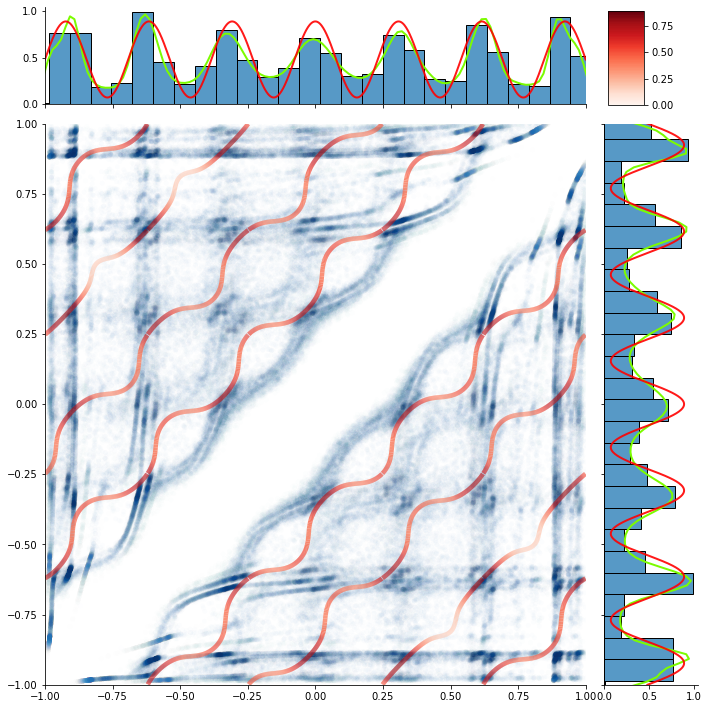}
    \caption{$N=20$, $\beta_0 = 0$, $K=10000$,  squared}
  \end{subfigure}
  \begin{subfigure}{0.24\textwidth}
    \includegraphics[width=\textwidth]{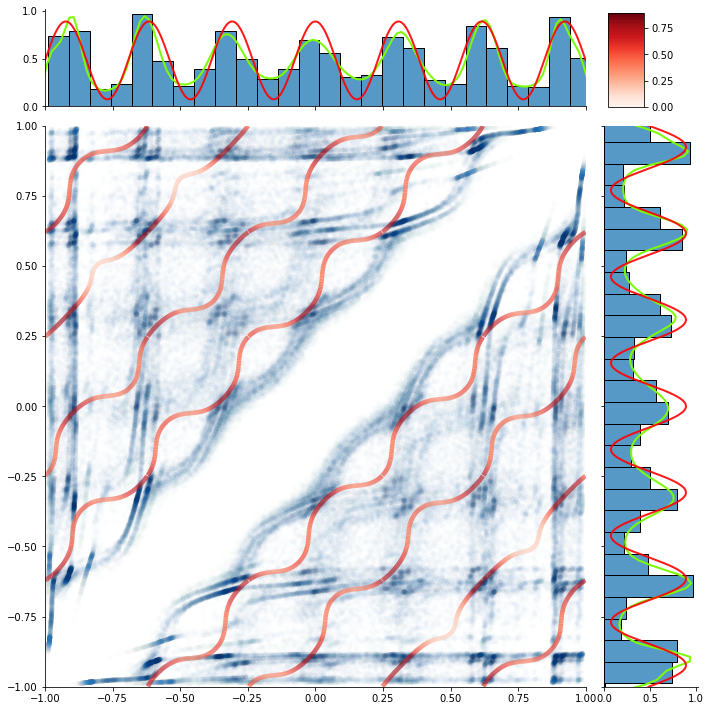}
    \caption{$N=20$, $\beta_0 = 10^{-3.5}$, $K=10000$,  squared}
  \end{subfigure}
  \begin{subfigure}{0.24\textwidth}
    \includegraphics[width=\textwidth]{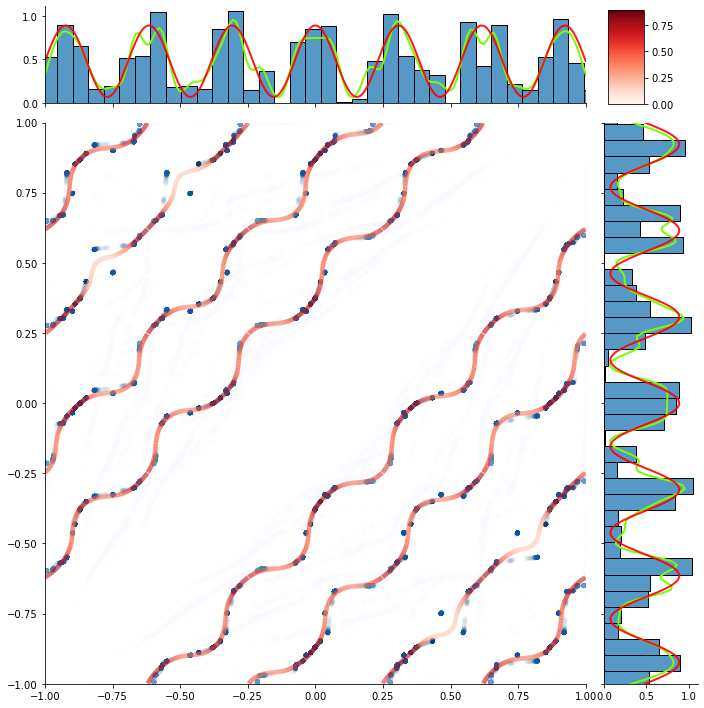}
    \caption{$N=40$, $\beta_0 = 0$, $K=10000$,  squared}
  \end{subfigure}
  \begin{subfigure}{0.24\textwidth}
    \includegraphics[width=\textwidth]{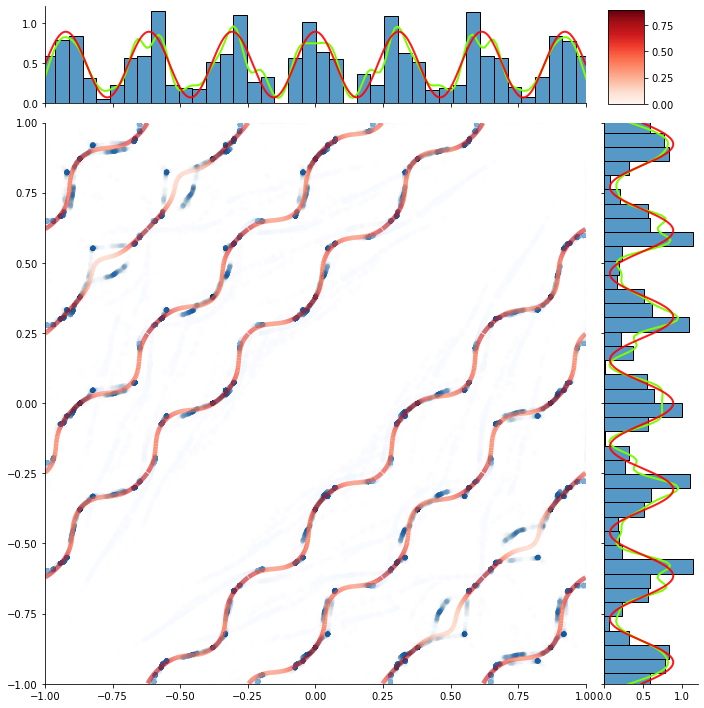}
    \caption{$N=40$, $\beta_0 = 10^{-3.5}$, $K=10000$,  squared}
  \end{subfigure}
  \begin{subfigure}{0.24\textwidth}
    \includegraphics[width=\textwidth]{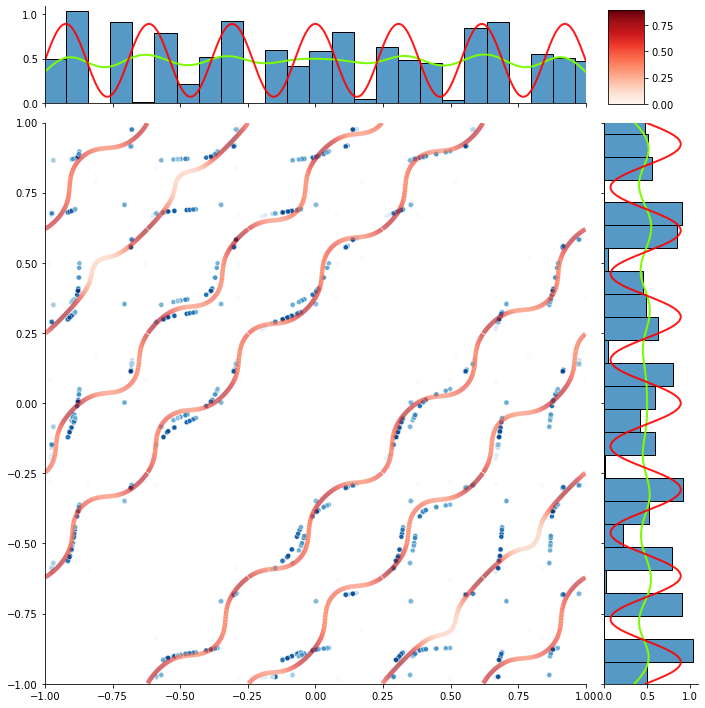}
    \caption{$N=20$, $\beta_0 = 0$, $K=100$,  squared}
  \end{subfigure}
  \begin{subfigure}{0.24\textwidth}
    \includegraphics[width=\textwidth]{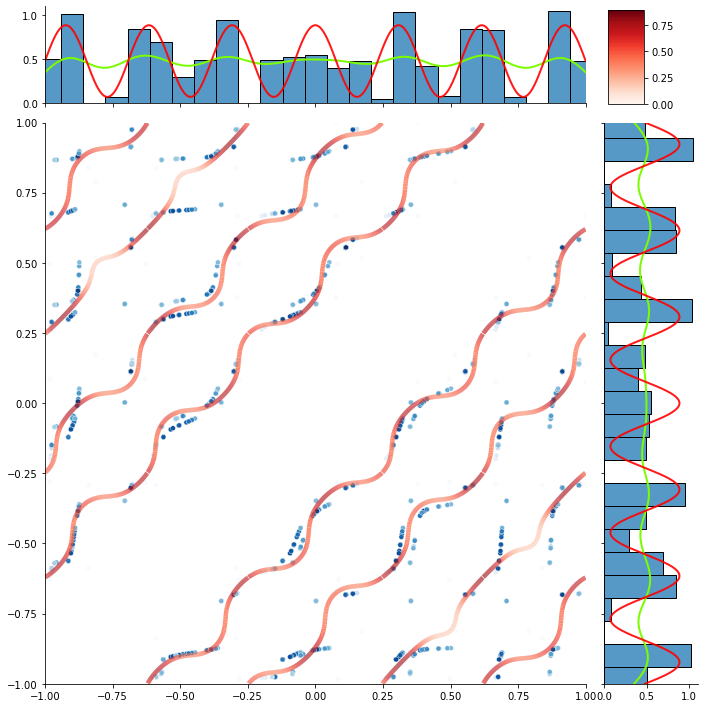}
    \caption{$N=20$, $\beta_0 = 10^{-3.5}$, $K=100$,  squared}
  \end{subfigure}
  \begin{subfigure}{0.24\textwidth}
    \includegraphics[width=\textwidth]{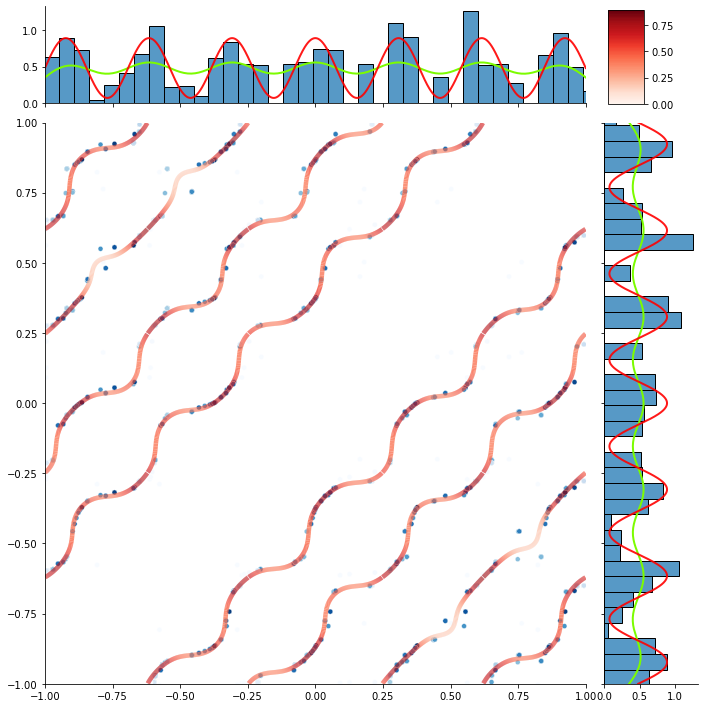}
    \caption{$N=40$, $\beta_0 = 0$, $K=100$,  squared}
  \end{subfigure}
  \begin{subfigure}{0.24\textwidth}
    \includegraphics[width=\textwidth]{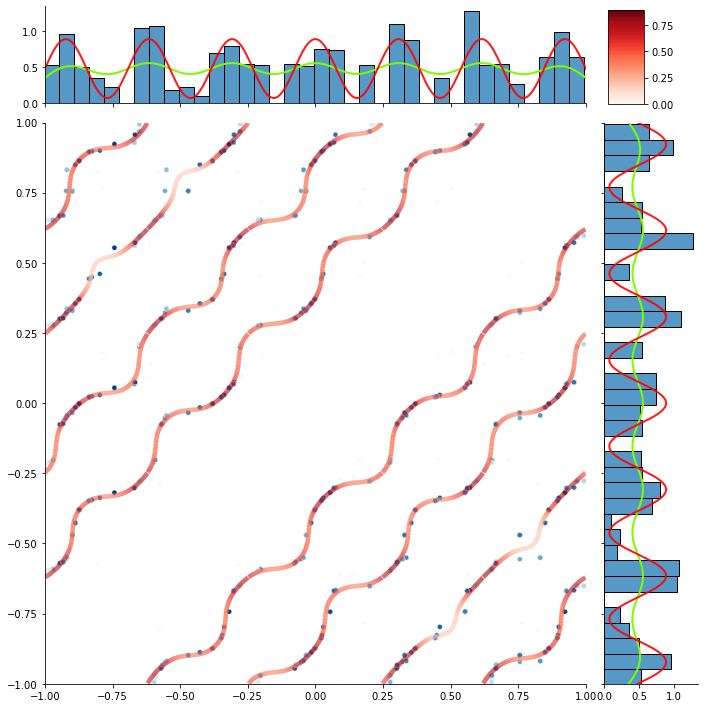}
    \caption{$N=40$, $\beta_0 = 10^{-3.5}$, $K=100$,  squared}
  \end{subfigure}
  \begin{subfigure}{0.24\textwidth}
    \includegraphics[width=\textwidth]{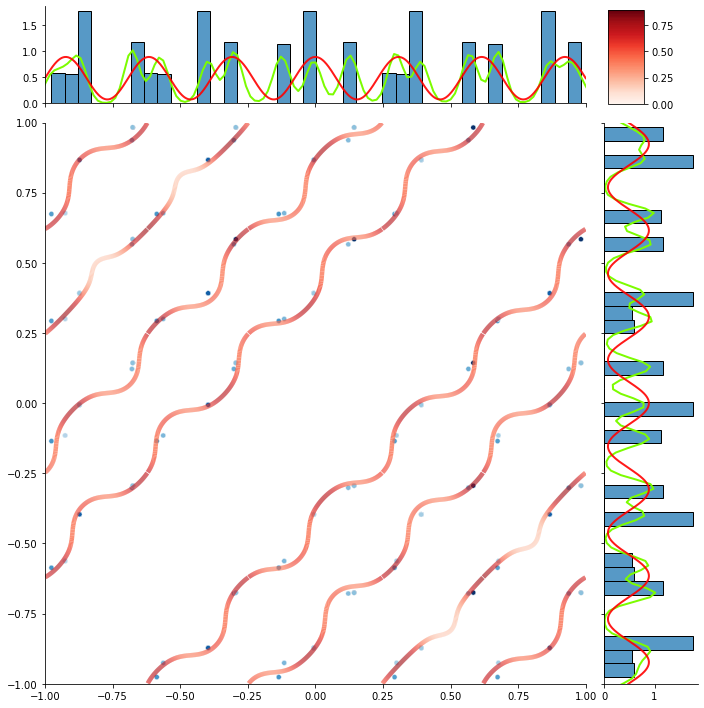}
    \caption{$N=20$, $\beta_0 = 0$, $K=10000$,  squared, with initial subsampling}
  \end{subfigure}
  \begin{subfigure}{0.24\textwidth}
    \includegraphics[width=\textwidth]{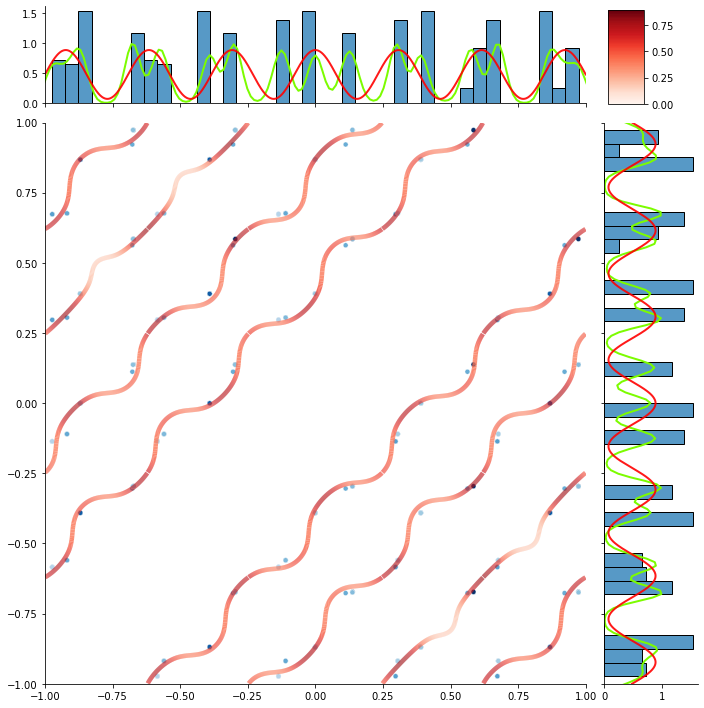}
    \caption{$N=20$, $\beta_0 = 10^{-3.5}$, $K=10000$,  squared, with initial subsampling}
  \end{subfigure}
  \begin{subfigure}{0.24\textwidth}
    \includegraphics[width=\textwidth]{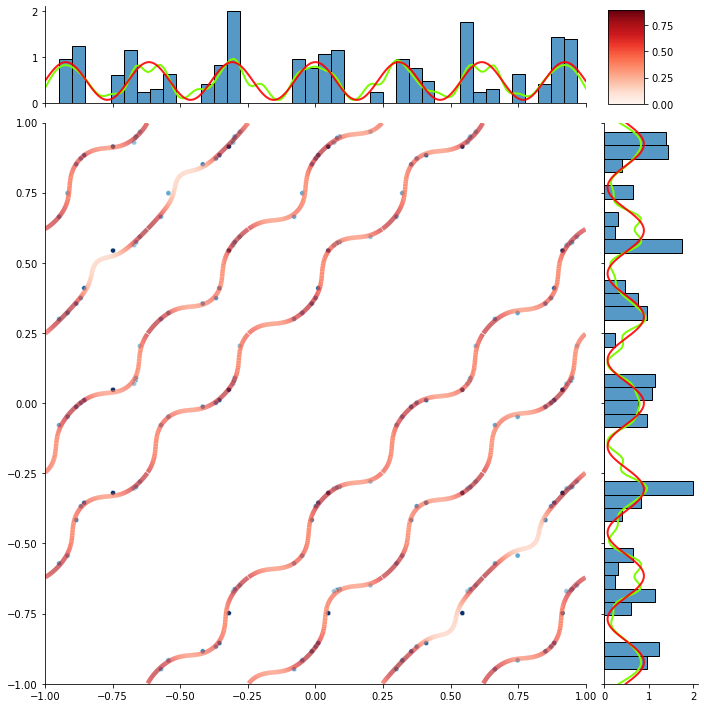}
    \caption{$N=40$, $\beta_0 = 0$, $K=10000$,  squared, with initial subsampling}
  \end{subfigure}
  \begin{subfigure}{0.24\textwidth}
    \includegraphics[width=\textwidth]{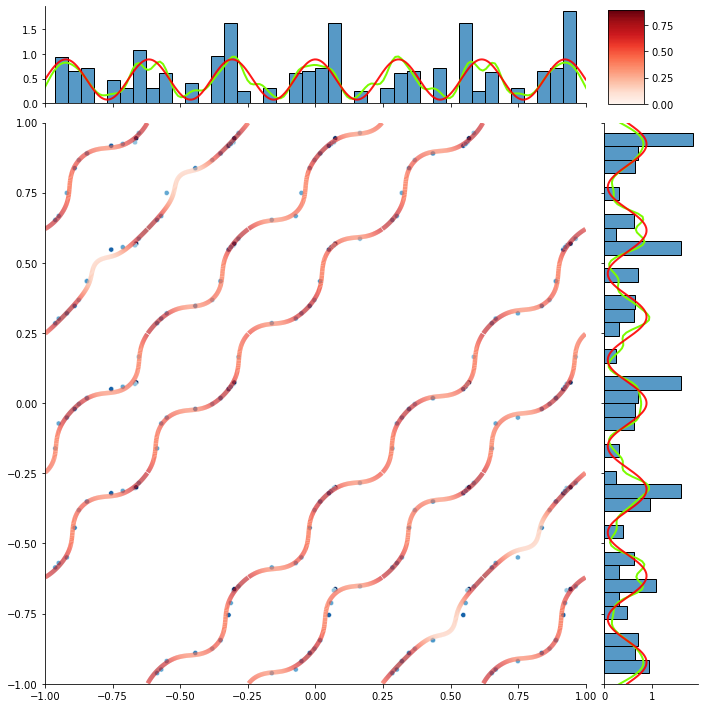}
    \caption{$N=40$, $\beta_0 = 10^{-3.5}$, $K=10000$,  squared, with initial subsampling}
  \end{subfigure}
  \caption{Optimal transport with $\mu_3$ and $M=5$, $\Delta t_0=  10^{-3}$. In each plot, on the main graph {\small $\frac{1}{M(M-1)} \sum_{k=1}^K \sum_{m \neq m’ = 1}^M w_k \delta_{x_m^k, x_{m’}^k}$} is represented by blue particles. The darker the heavier the particle. Particles have some transparency which allows to see more clearly areas of high concentration. Red curves represent the functions $T^i$ for $i\in \{1,\dots,M-1\}$ defined in Theorem~\ref{thm:colombo}. The higher the density the darker. On side graphs are represented in blue a weighted histogram of the particles, in red the marginal law and in green a normal kernel density estimate based on the weighted particles (with a bandwidth rule based on Scott's rule with $d=0$). \label{fig:minimasmu31D}}
\end{figure}

\clearpage
\subsection{Three-dimensional test cases ($d=3$)}\label{sect:3D}

\subsubsection{Tests design}\label{sect:3Dtestdesign}

The numerical experiments were realized with four different marginal laws that are named afterwards as follows:
\begin{align}
  \mu_1 &\sim \mathcal{N} \left(0_3,\mathrm{Id}_3\right), \\
  \mu_2 &\sim \frac 2 3 \mathcal{N}(0_3,\begin{pmatrix}1 & 0.5 & 0.75 \\ 0.5 & 2 & 1.5\\ 0.75 & 1.5 & 3 \end{pmatrix}) + \frac 1 3 \mathcal{N} (\begin{pmatrix}2 \\ 2 \\ 2  \end{pmatrix},\begin{pmatrix} 1 & 0.8 & 0.22 \\ 0.8 & 2 & 1.8 \\ 0.22& 1.8 & 3  \end{pmatrix}),  \\
  \mu_3 &\sim \frac 1 {10} \mathcal{N} (0_3, C) + \frac 1 5 \mathcal{N}(\begin{pmatrix} 4 \\ 0 \\0 \end{pmatrix}, C)
  + \frac 1 5 \mathcal{N}(\begin{pmatrix} 8 \\ 0 \\0 \end{pmatrix}, C)
  + \frac 1 5 \mathcal{N}(\begin{pmatrix} 12 \\ 0 \\0 \end{pmatrix}, C) \nonumber \\
  &\quad+ \frac 1 5 \mathcal{N}(\begin{pmatrix} 16 \\ 0 \\0 \end{pmatrix}, C)
  + \frac 1 {10} \mathcal{N}(\begin{pmatrix} 20 \\ 0 \\0 \end{pmatrix}, C), \qquad \text{with} \quad C = \begin{pmatrix}1 & 0.5 & 0.75 \\ 0.5 & 2 & 1.5\\ 0.75 & 1.5 & 3 \end{pmatrix}, \\
  \mu_4 &\sim {\cal U} \left({\cal B}(0,1)\right).
\end{align}
And, for $i = 1, 2, 3, 4$, using as test functions\footnotemark tensor products of 1D orthonormal polynomials $(P^{\mu_i, j}_l)_{\substack{1 \le j \le 3,\\ l \in \N}}$, defined as, for $j=1, 2, 3$, $l \in \N$,
\begin{align}
  &\mathrm{degree}\left(P^{\mu_i, j}_l\right) = l, && \forall l' < l, \, \int_{\R^3} P^{\mu_i, j}_l(x_j) P^{\mu_i, j}_{l'}(x_j) d \mu_i (x_1, x_2, x_3) = \frac{1}{(l+1)^2}\delta_{l,l'}.
\end{align}

\footnotetext{These polynomials were chosen after a few numerical tests on some optimization procedures for their better convergence properties than the polynomials they were compared to. Their tensorised form both eases the computation of the moments and allows some parallelisation. Note also that the matrix $\nabla \Gamma^K(Y^K)$ is a multivariate Vandermonde matrix. We checked numerically its invertibility throughout the optimization process.}

As for a finite number of multivariate polynomials (and under a suitable control of mixed derivatives), the hyperbolic cross~\cite{MR3887571} seems to behave better than using all polynomials up to a given degree, we used, for a number of test functions $N$ appropriately chosen, the polynomials $P^{\mu_i, 1}_{l_1}\otimes P^{\mu_i, 2}_{l_2}\otimes P^{\mu_i, 3}_{l_3}$, where
\begin{equation}
  (l_1 + 1)(l_2 + 1)(l_3 + 1) \le L_N,
\end{equation}
where $L_N$ is defined such that $\#\{(l_1, l_2, l_3) | (l_1 + 1)(l_2 + 1)(l_3 + 1) \le L_N \} = N$. The map between maximum degree of the polynomials ($L_N - 1$) and $N$ is shown in Table \ref{tbl:hyperboliccross}.

\begin{table}[htp]
\begin{center}
  \begin{tabular}{c|cccccc}
  $L_N - 1$ & 6 & 7 & 8 & 9 & 10 & 11 \\
  \hline
  $N$ & 28 & 38 & 44 & 53 & 56 & 74
  \end{tabular}
 \end{center}
 \caption{Map between the maximum degree of 1D polynomials and the number of test functions using hyperbolic cross in 3D. \label{tbl:hyperboliccross}}
\end{table}

In the numerical examples presented afterwards, as all weights are fixed to $\frac 1 K$, there is no need to use the polynomial of degrees $(l_1, l_2, l_3) = (0, 0,0)$, hence values of $N$ decreased by 1 compared to the values of Table \ref{tbl:hyperboliccross}.

\begin{remark}
  One of the main advantages of using sums of Normal functions (or a uniform measure on a ball) as marginal laws and polynomials as test functions is that their exists in that case close formulas for the computation of the moments. From our experiments in dimension 1, the precision of the computation of the moments is important both for the solution of the MCOT problem to be well-defined (and thus for the algorithm to converge) -- numerically computed moments, though not exact, must allow the existence of $Y^K \in ((\R^d)^M)^K$ such that $\|\Gamma^K(Y^K) \|_\infty \le \epsilon$ for $\epsilon$ the machine-precision; and for the convergence as $N$ increases of the MCOT cost towards the OT cost -- numerically computed moments not precise enough might hide this convergence. Numerical quadratures in 3D could be implemented for dealing with more general marginal laws and test functions, however, their computation and convergence speed put it beyond the scope of this article.
\end{remark}

\paragraph{Mean-Covariance.}

Tests were also performed using as test functions the mean and covariance matrix for $\mu_1$ and $\mu_2$, in order to notice on examples how much those test functions do constrain an optimal transport problem.  Note that this problem of optimal transport when the mean and covariance structure are given may be interesting per se, when only partial information on the distribution is known. We have indicated in Table~\ref{tbl:meancovar} the optimal costs obtained with our algorithm for $\mu_1$ and $\mu_2$ with mean-covariance constraints ($N=9$) and with many moment constraints ($N=52$). We observe on our examples a relative difference around 15-20\%.

\begin{table}[htp]
\begin{center}
  \begin{tabular}{c|cccc}
   & $\mu_1$, $M = 10$ & $\mu_1$, $M = 100$ & $\mu_2$, $M = 10$ & $\mu_2$, $M = 100$ \\
   \hline
  $N = 9$ & 10.65 & 1395 & 8.007 & 1074 \\
  $N = 52$ & 12.50 & 1599 & 9.107 & 1201
  \end{tabular}
 \end{center}
 \caption{Optimal value of the cost obtained for $\mu_1$ and $\mu_2$ with mean-covariance constraints ($N=9$) and with many moment constraints ($N=52$). \label{tbl:meancovar}}
\end{table}

\paragraph{Cost.}
In order to avoid too high values of the cost function, we used in all experiments a regularized  Coulomb  cost $c(x_1, \dots, x_M) = \sum_{m \neq m' = 1}^M \frac{1}{\epsilon + |x_m - x_{m'}|}$, with $\epsilon = 10^{-3}$ and $\forall i= 1,\dots, M, x_m \in \R^3$.

\paragraph{Fixed weights.}
After several tests comparing fixed and variable weights (with various weight functions), we observe that in dimension 3, for the marginal laws considered, both initialization and optimization using variable weights were much slower than using fixed weights. Therefore, all following tests have been performed using fixed weights. Heuristically, when using variable weights, some particles tend to have large weights and are strongly constrained while other ones become lightweight and do not move much since the gradient on positions is proportional to weights.

\subsubsection{Initialization and constraints enforcement -- Figure \ref{fig:initN}}\label{sect:3Dinit}

Initialization was performed by a sampling $K$ particles according to the marginal law, and then using the Runge-Kutta method showed in Section~\ref{sec:init} to bring the particles on the submanifold of the constraints ${\cal M}^K$. This method has been tested for various values of $N$ and $K$, presented respectively in Figures \ref{fig:initN}.

As $N$ increases (Figure \ref{fig:initN}), the submanifold of the constraints becomes harder to reach using the Runge-Kutta method (similarly to the 1D case), 
and large values of $N$ ($ L_N \ge  11$) could not be attained in the time of the numerical experiment (remind that the number of computations involved at each iteration grows linearly with $N$). In the case of each marginal laws ($\mu_1$ and $\mu_2$) for which the tests have been performed, despite the assymmetry of $\mu_2$, the dependence on $N$ of the convergence speed appears to be similar.


Note also that as we use symmetrised test functions (regarding the marginal laws) with fixed weights, the number of independant coordinates involved in the Runge-Kutta method to satisfy the constraints is linear in $KM$ ($M$ being the number of marginal laws). Thus, solving the problem of finding a starting point on the submanifold with $100$ marginal laws and $10^3$ particles is numerically the same as the one with $10$ marginal laws and $10^4$ particles. Although in the case where weights are variable this remark can not be applied, as coordinates on different marginal laws of the same particle share the same weight, increasing the number of marginal laws relaxes the problem of finding a starting point on the submanifold.

\begin{figure}[!t]
  \centering
  \begin{subfigure}{0.48\textwidth}
    \includegraphics[width=\textwidth]{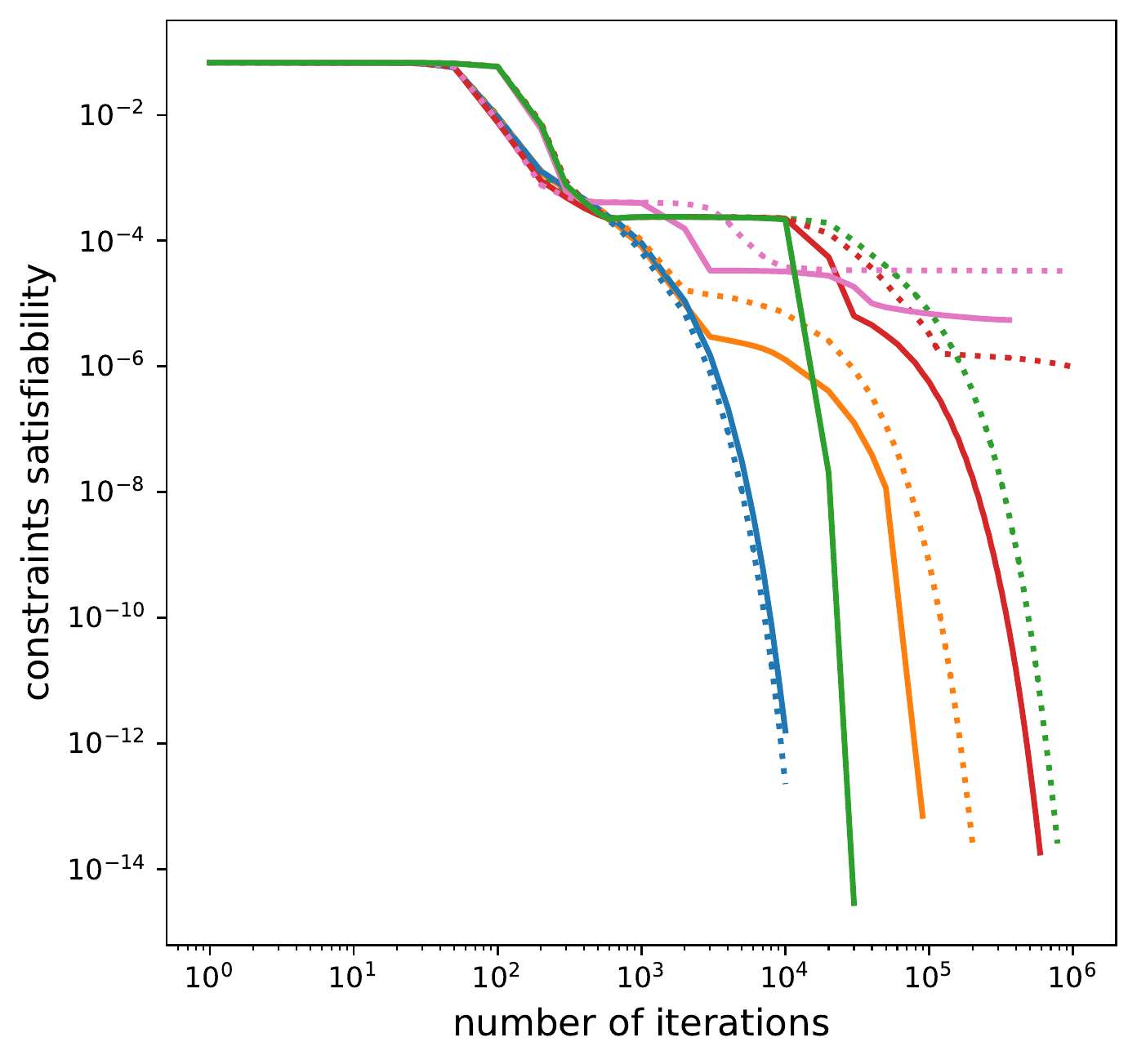}
    \caption{$\mu_1$, $M=10$ \label{fig:initNmu110}}
  \end{subfigure}
  \begin{subfigure}{0.48\textwidth}
    \includegraphics[width=\textwidth]{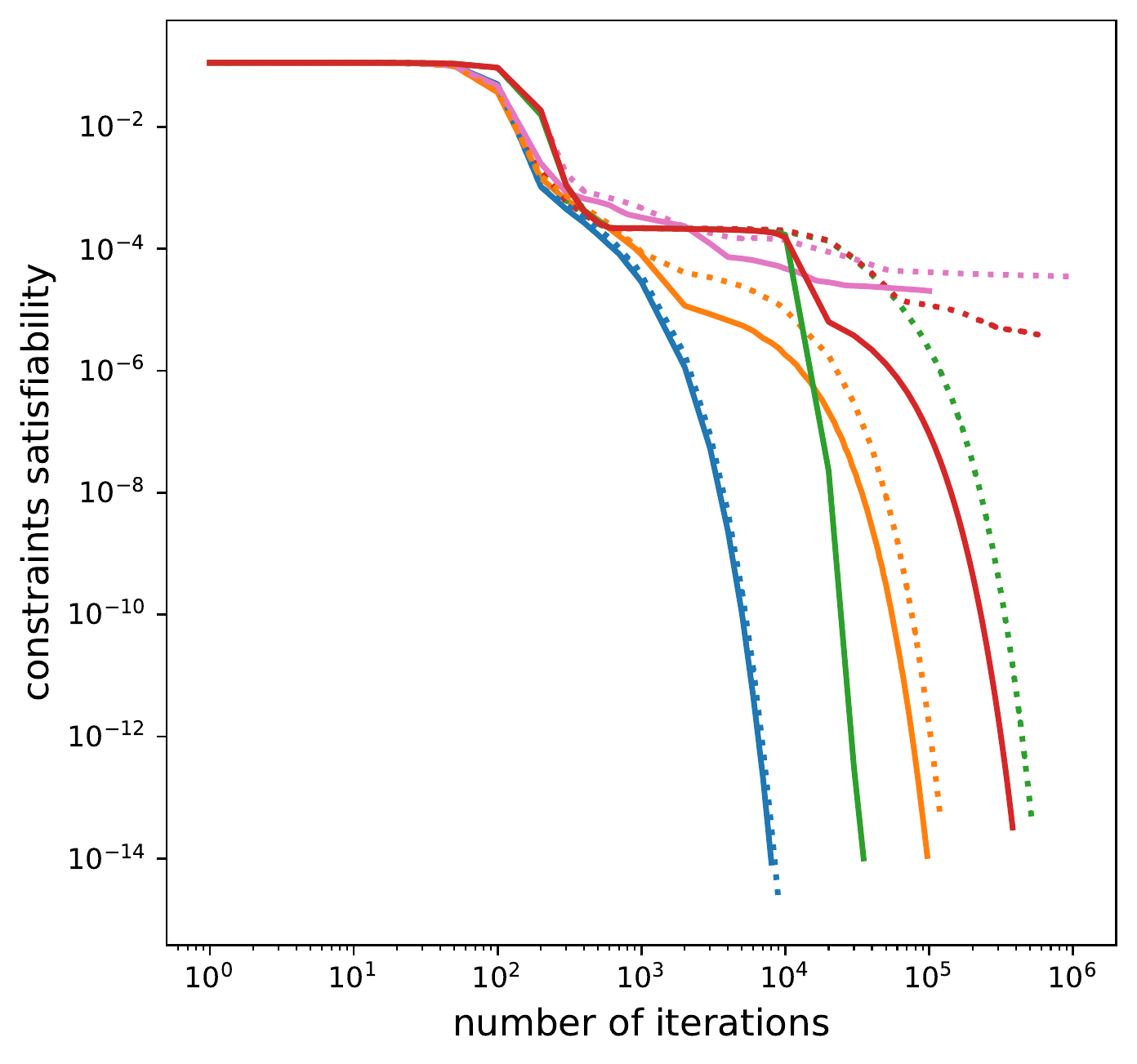}
    \caption{$\mu_2$, $M=10$ \label{fig:initNmu210}}
  \end{subfigure}
  \begin{subfigure}{0.48\textwidth}
    \includegraphics[width=\textwidth]{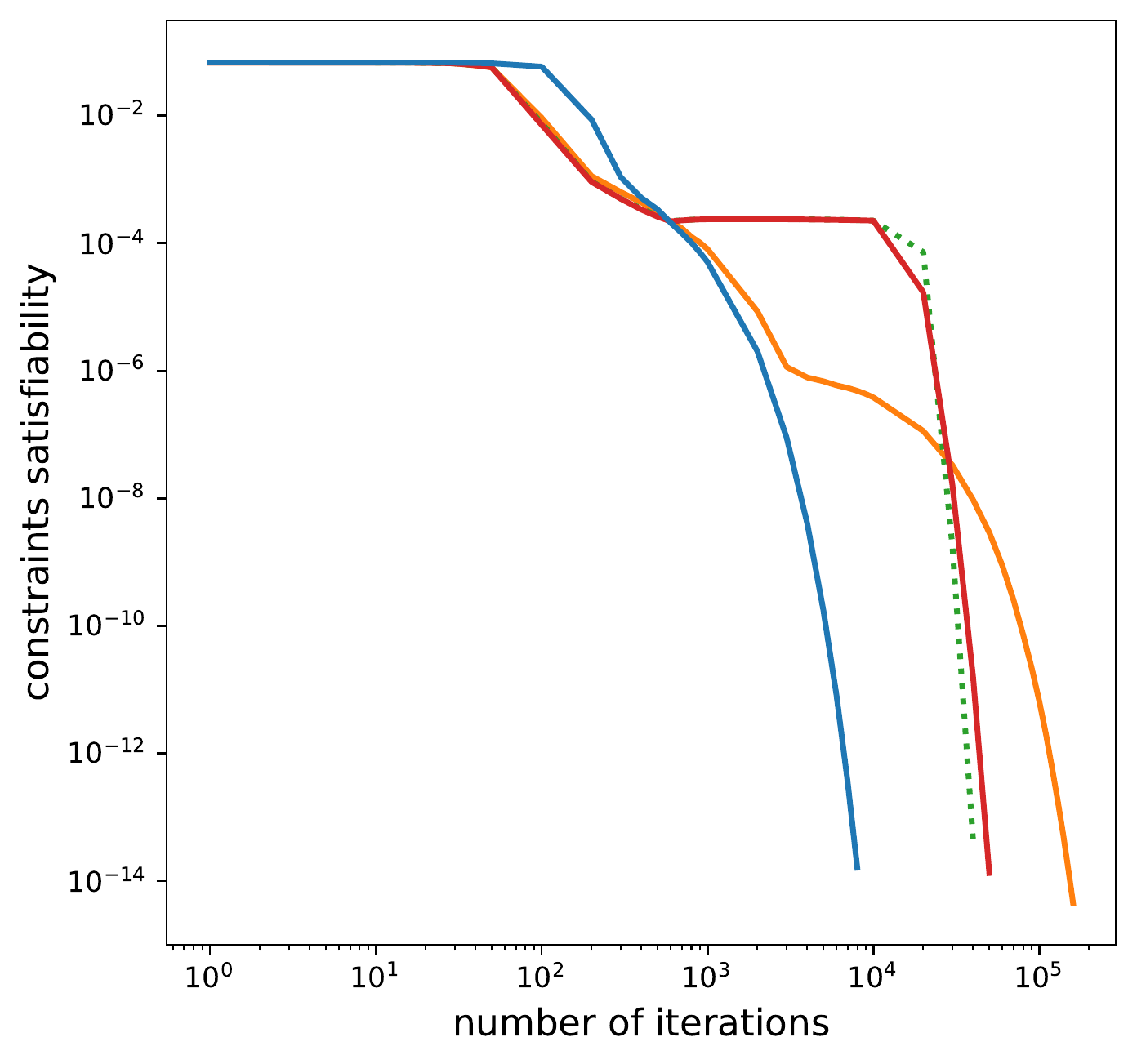}
    \caption{$\mu_1$, $M=100$ \label{fig:initNmu1100}}
  \end{subfigure}
  \begin{subfigure}{0.48\textwidth}
    \includegraphics[width=\textwidth]{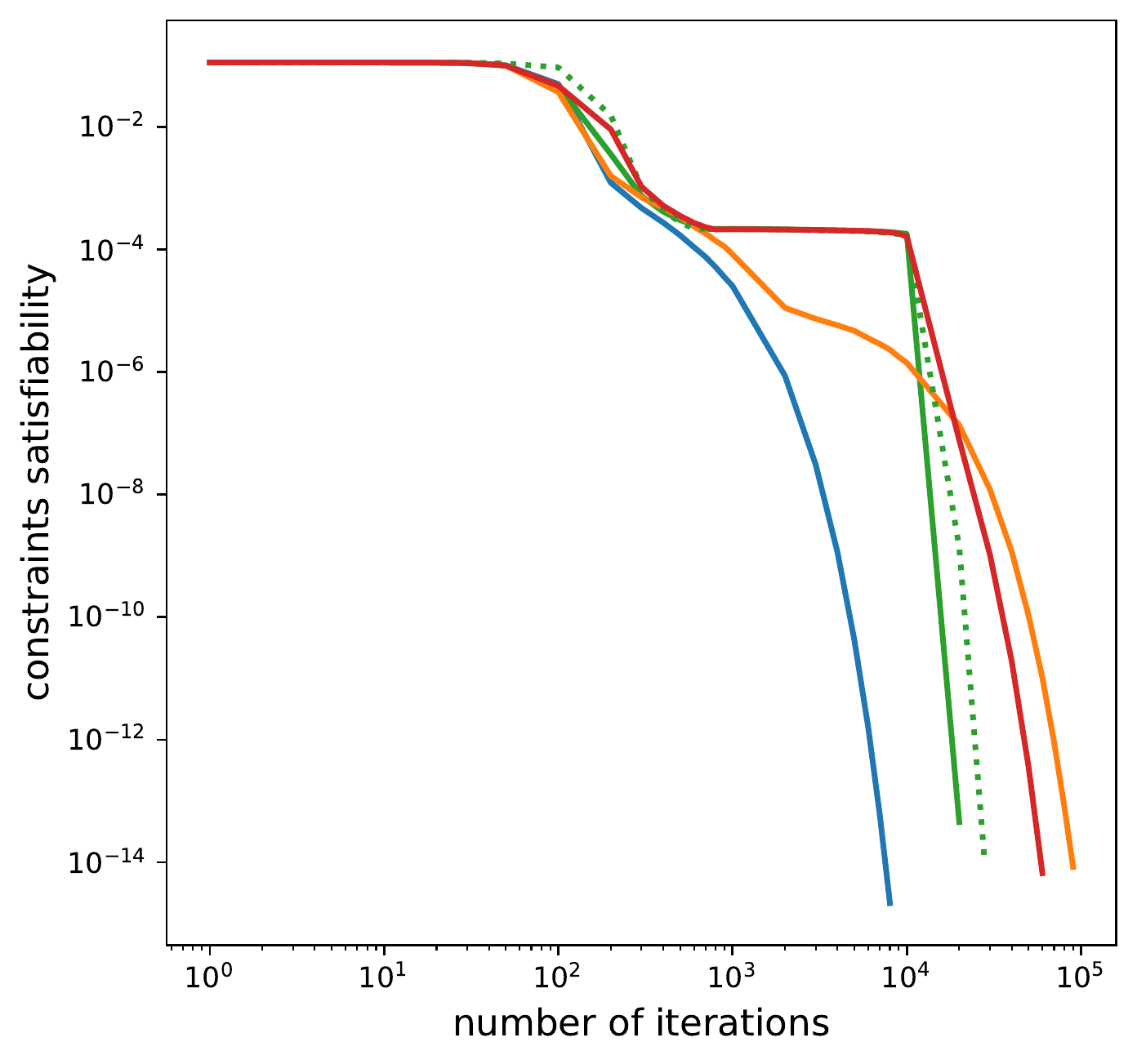}
    \caption{$\mu_2$, $M=100$ \label{fig:initNmu2100}}
  \end{subfigure}
  \caption{Evolution of $\| \Gamma^K(Y^K_m) \|_\infty$ for values of $N$ ranging from 27 to 52 and $K$ between 1000 (dotted lines) and 10000 (solid lines) as a function of the number of iterations $m$ of the Runge-Kutta~3 procedure. $\Delta t_0 = 10^{-4}$. Blue curves are for $N=27$, orange ones for $N = 37$, green ones for $N = 43$, red ones for $N = 52$ and pink ones for $N = 73$.
  \label{fig:initN}}
\end{figure}

\subsubsection{Optimization procedure -- Figures \ref{fig:noise}, \ref{fig:optimN}, \ref{fig:duringoptim1} and \ref{fig:duringoptim2}}

The aim of Figures \ref{fig:noise} and \ref{fig:optimN} is to plot the evolution of $V^K(Y^K_n)$ as a function of $n$ the number of iterations of the constrained overdamped Langevin algorithm presented in Section~\ref{sect:numericaldiscretization}
for various values of $N$ and values of $\beta_0$. As we observed (Figure \ref{fig:noise}), and similarly to the tests in dimension 1, that tests with $\beta_0 = 0$ converges faster than $\beta_0 > 0$, we kept $\beta_0 = 0$ for all the other tests. The convergence of the cost for various values of $N$ and $K$, various number of marginal laws and for $\mu_1$ and $\mu_2$ is presented in Figure \ref{fig:optimN}. And a presentation of how particles move during the optimization procedure can be seen in Figures \ref{fig:duringoptim1} and \ref{fig:duringoptim2}.


\begin{figure}[tp]
  \centering
  \includegraphics[width=0.48\textwidth]{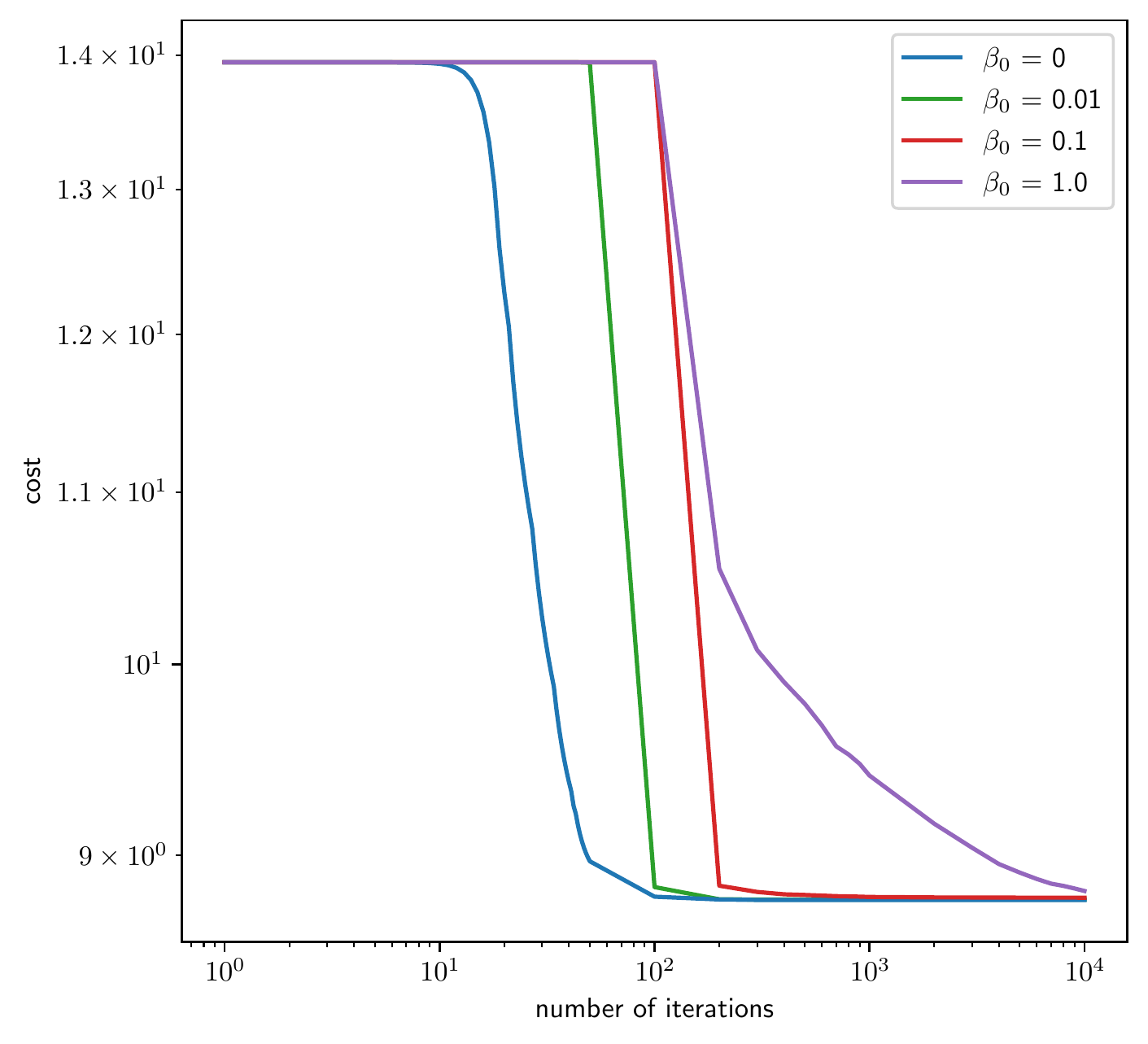}
  \caption{Evolution of the cost as a function of the number of iterations $n$ for various values of $\beta_0$. The marginal law is $\mu_2$, $\beta_0$ varies from $0$ to $1$ and the other parameters are $\Delta t_0 = 10^{-4}$, noise level decreases as the squareroot of the number of iterations, $N=27$, $M=10$, $K = 160$.\label{fig:noise}}
\end{figure}


On all subgraphs of Figure \ref{fig:optimN}, one can observe that the optimization procedure reaches a cost close to the optimal one for the MCOT problem in 50-200 iterations, when $K$ is large enough for a given $N$ (e.g. $K=1000$ is sufficient when $N=27$ but not when $N=43$).
As $N$ increases the value of the optimal costs does as well, which is expected, as MCOT problems get more and more constrained. As $K$ increases, the value of the cost computed converges towards the MCOT cost.
Indeed, the  slight decrease of the computed MCOT cost at the 20000$^\mathrm{th}$ iteration as $K$ increases that can be observed in Table \ref{tbl:minimascost} from $K=320$ to $K=10000$ suggests that their exists $K_0 \in \N$ such that for $K \ge K_0$, the gain in an increase in $K$ reflects weakly on the MCOT cost computed.

On Figures \ref{fig:duringoptim1} and \ref{fig:duringoptim2} is plotted the evolution of some symmetrized visualizations of the process during the optimization for an MCOT problem on $\mu_1$. Although at each iteration it satisfies the moment constraints, it deviates from a Normal sample rapidly and tends to concentrate on some points (a bit like in Tchakaloff's theorem and \cite[Theorem 3.1]{alfonsi2019approximation}).

\begin{figure}[tp]
  \centering
  \begin{subfigure}{0.48\textwidth}
    \includegraphics[width=\textwidth]{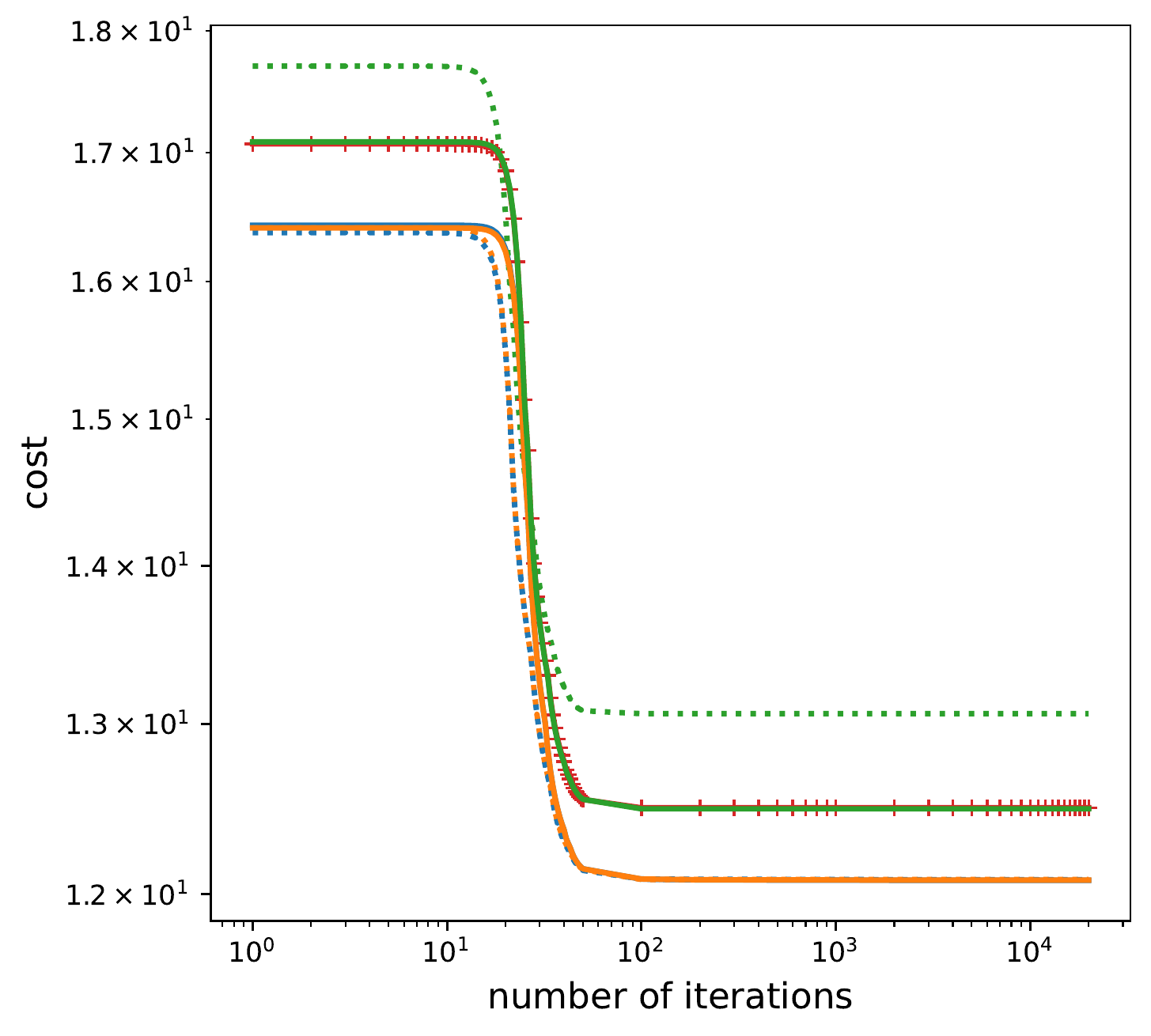}
    \caption{$\mu_1$, $M = 10$ \label{fig:optimNmu110}}
  \end{subfigure}
  \begin{subfigure}{0.48\textwidth}
    \includegraphics[width=\textwidth]{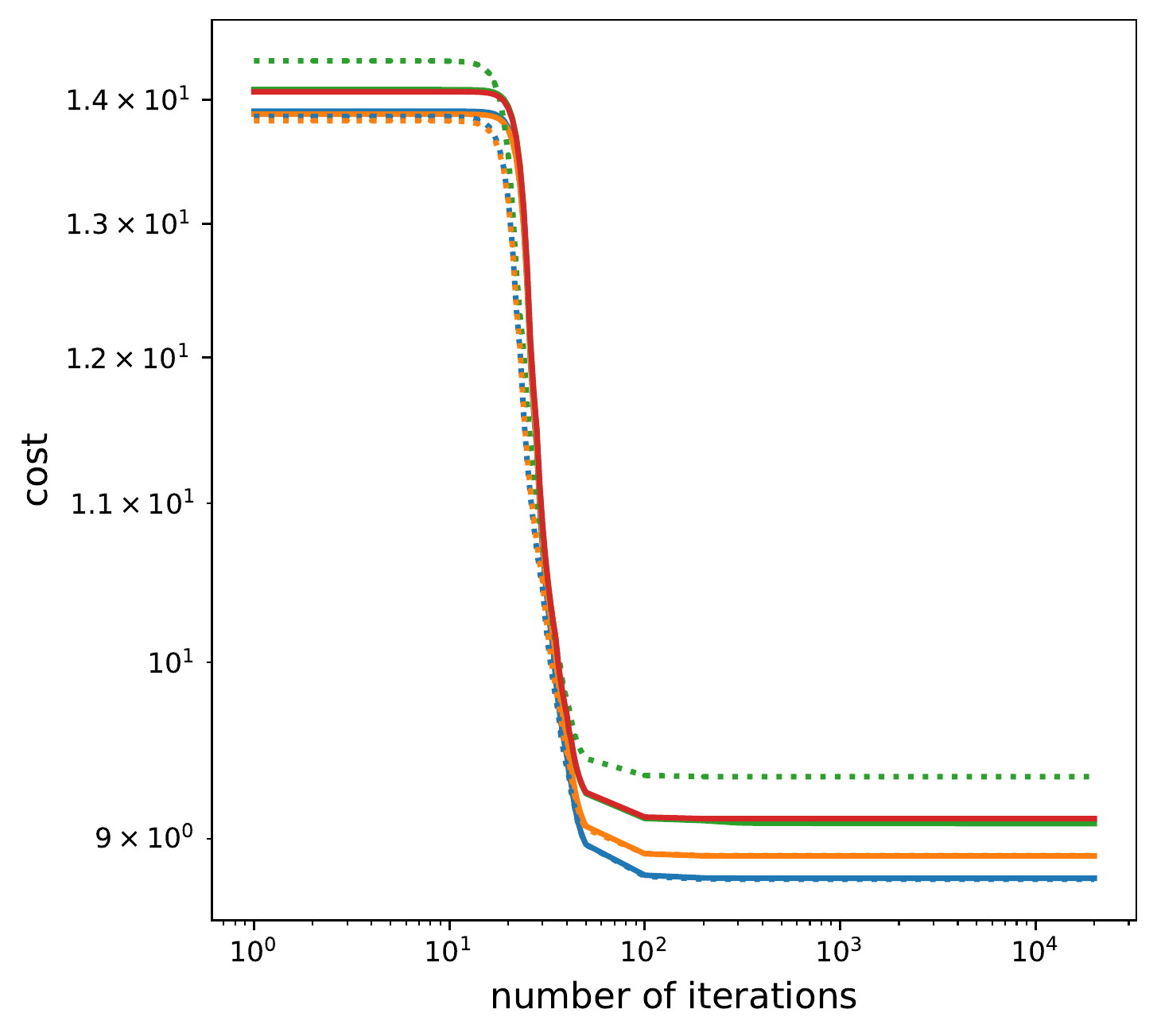}
    \caption{$\mu_2$, $M = 10$ \label{fig:optimNmu210}}
  \end{subfigure}
  \begin{subfigure}{0.48\textwidth}
    \includegraphics[width=\textwidth]{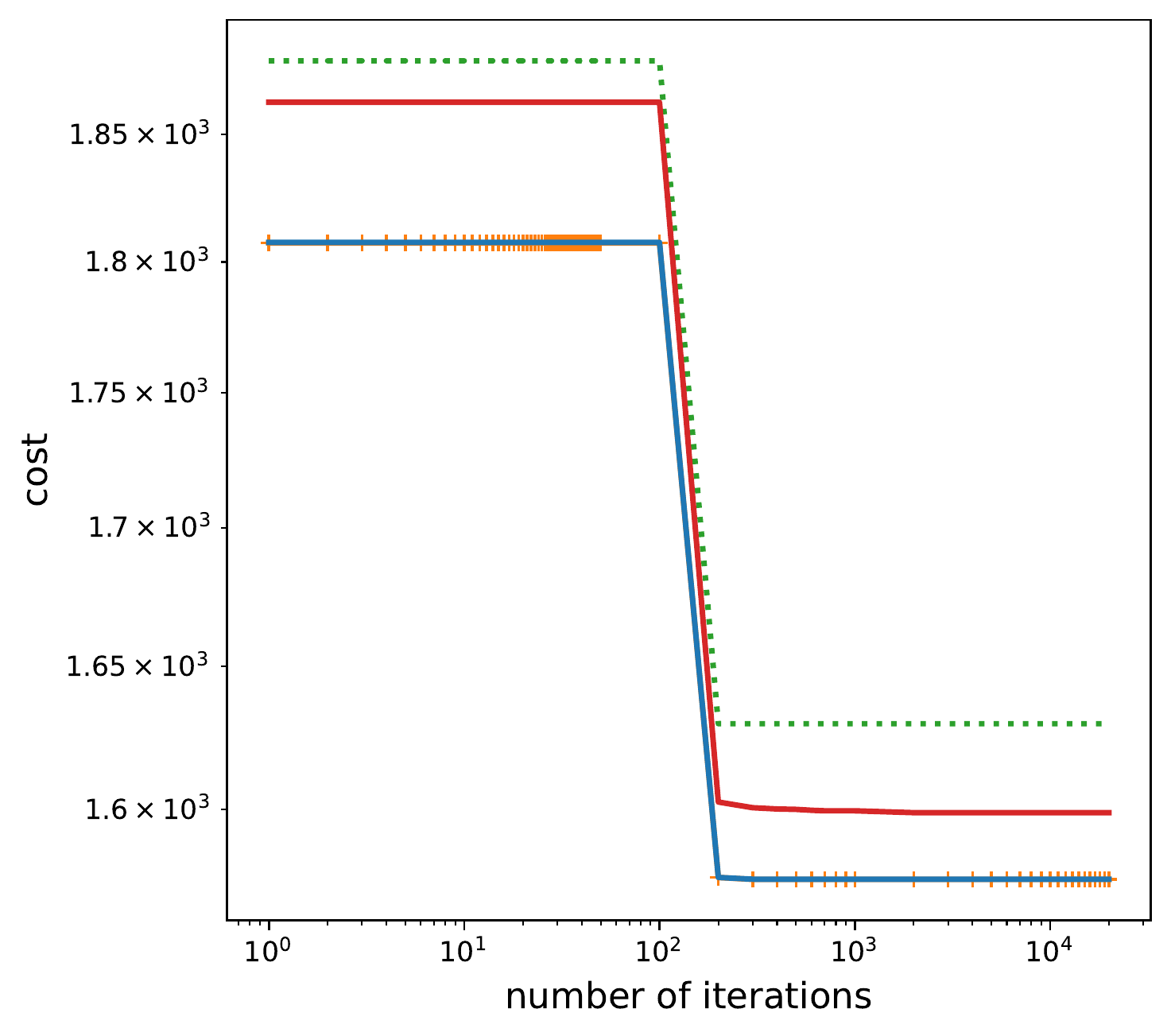}
    \caption{$\mu_1$, $M= 100$ \label{fig:optimNmu1100}}
  \end{subfigure}
  \begin{subfigure}{0.48\textwidth}
    \includegraphics[width=\textwidth]{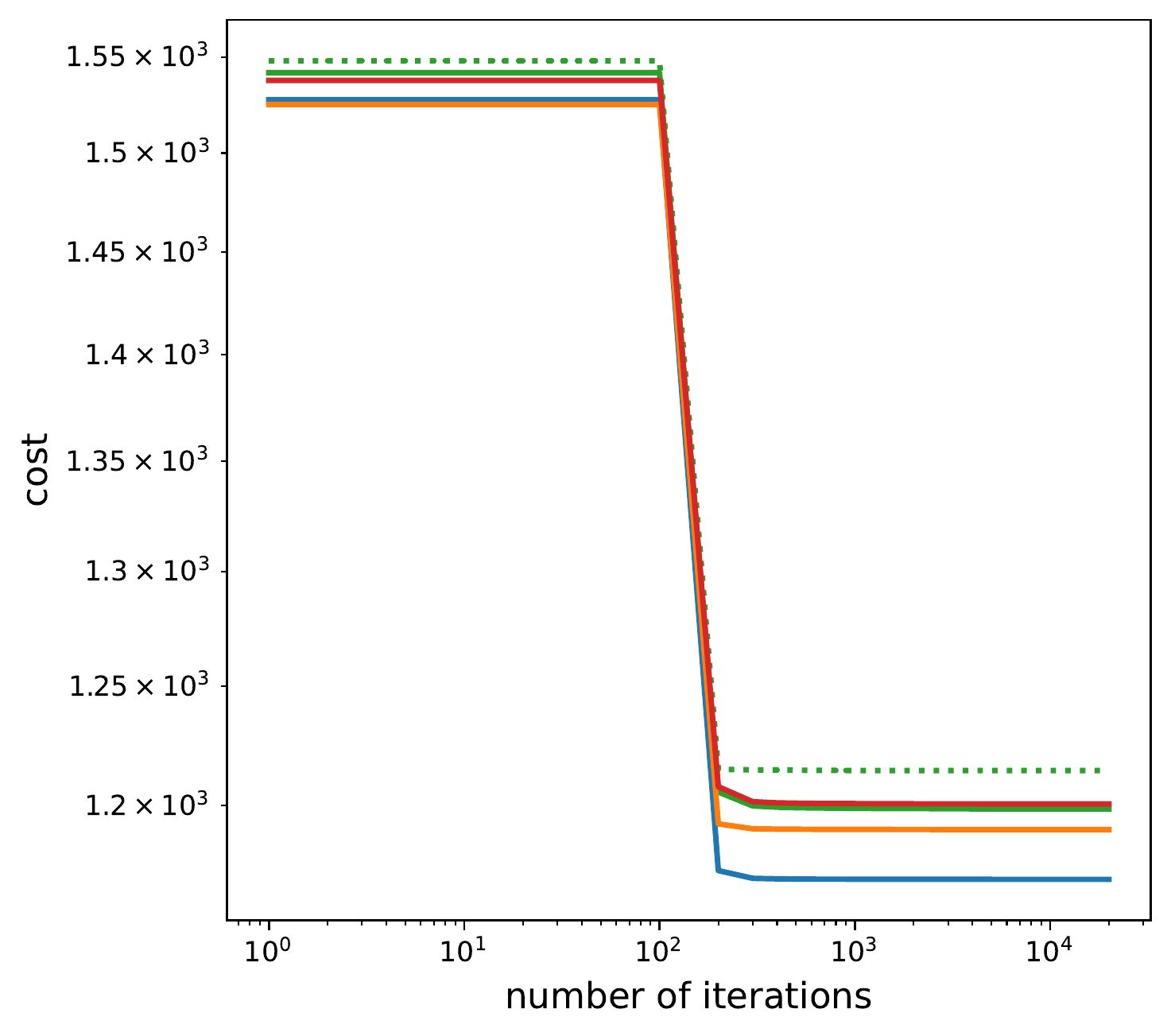}
    \caption{$\mu_2$, $M = 100$ \label{fig:optimNmu2100}}
  \end{subfigure}
  \caption{Evolution of the cost as a function of the number of iterations $n$ for various values of $N$ and $K$ from 1000 (dotted lines) to 10000 (solid lines). $\Delta t_0 = 10^{-4}$, $\beta_0  = 0$. Blue curves are for $N=27$, orange ones for $N = 37$, green ones for $N = 43$, red ones for $N = 52$ and pink ones for $N = 73$. On Figures \ref{fig:optimNmu110} and \ref{fig:optimNmu1100}, ``+" signs are added to better distinguish overlaid curves.\label{fig:optimN}}
\end{figure}


\begin{figure}[htp]
  \centering
  \begin{multicols}{3}
    \captionsetup[sub]{labelformat=anumber}
    \begin{subfigure}{0.32\textwidth}
      \includegraphics[width=\textwidth]{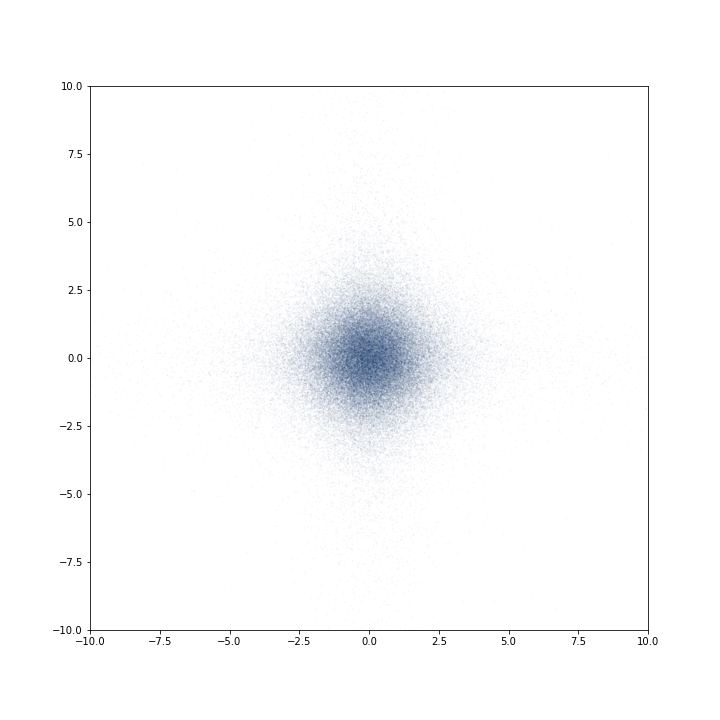}
      \caption{plane XY, iteration 1 \label{fig:duringoptimXY1}}
    \end{subfigure}\par
    \begin{subfigure}{0.32\textwidth}
      \includegraphics[width=\textwidth]{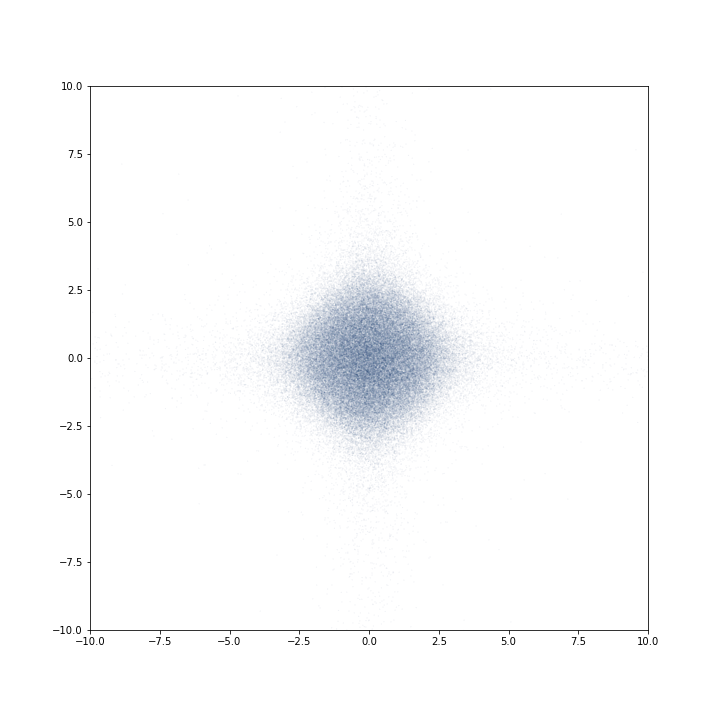}
      \caption{plane XY, iteration 30 \label{fig:duringoptimXY30}}
    \end{subfigure}\par
    \begin{subfigure}{0.32\textwidth}
      \includegraphics[width=\textwidth]{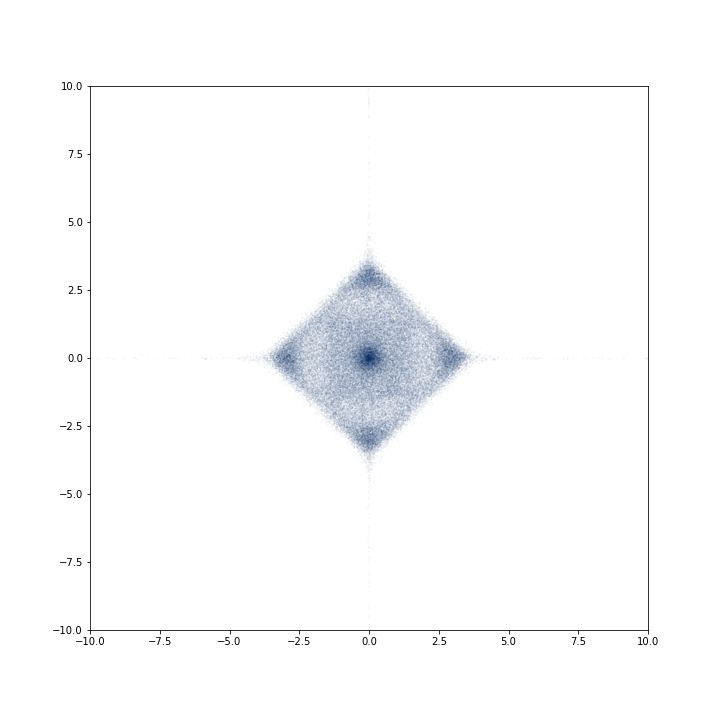}
      \caption{plane XY, iteration 50 \label{fig:duringoptimXY50}}
    \end{subfigure}\par

    \captionsetup[sub]{labelformat=bnumber}
    \setcounter{subfigure}{0}
    \begin{subfigure}{0.32\textwidth}
      \includegraphics[width=\textwidth]{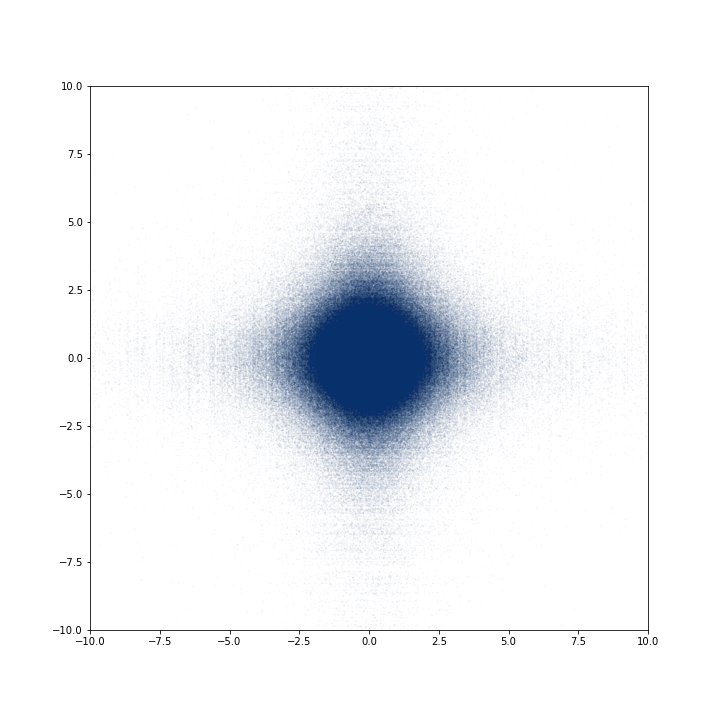}
      \caption{X axis, iteration 1 \label{fig:duringoptimX1}}
    \end{subfigure}\par
    \begin{subfigure}{0.32\textwidth}
      \includegraphics[width=\textwidth]{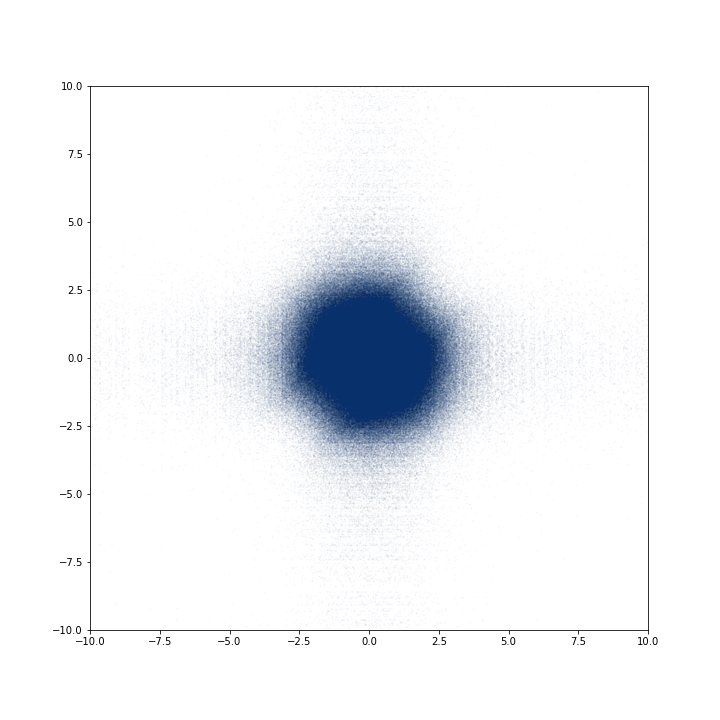}
      \caption{X axis, iteration 30 \label{fig:duringoptimX30}}
    \end{subfigure}\par
    \begin{subfigure}{0.32\textwidth}
      \includegraphics[width=\textwidth]{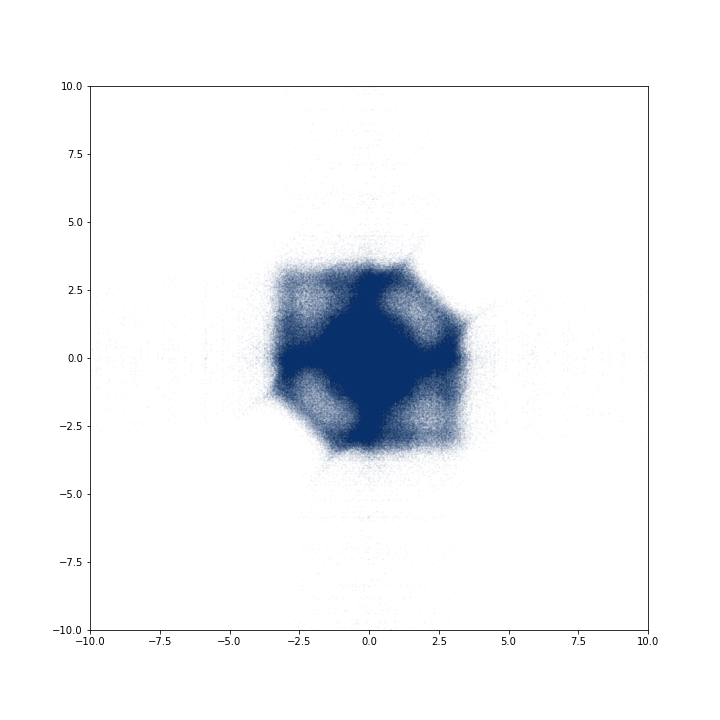}
      \caption{X axis, iteration 50 \label{fig:duringoptimX50}}
    \end{subfigure}\par

    \captionsetup[sub]{labelformat=cnumber}
    \setcounter{subfigure}{0}
    \begin{subfigure}{0.32\textwidth}
      \includegraphics[width=\textwidth]{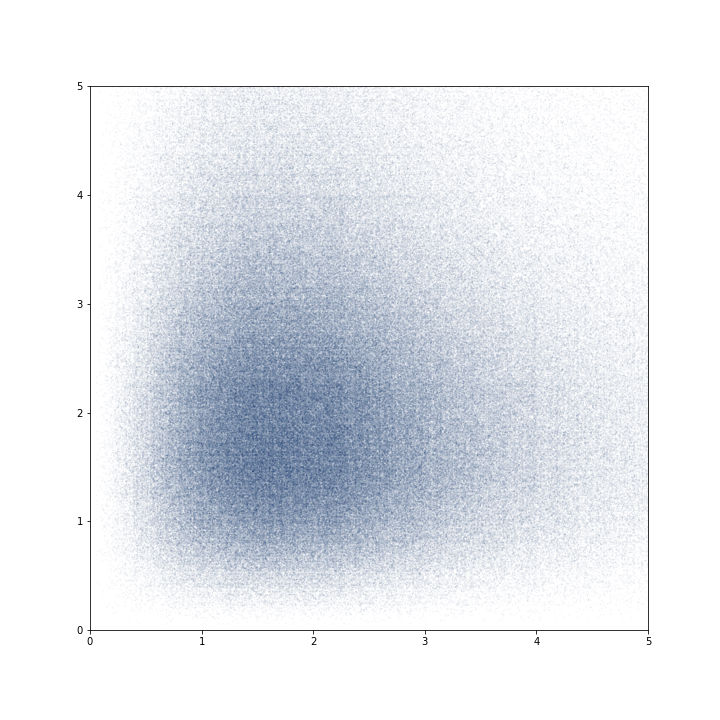}
      \caption{radial, iteration 1 \label{fig:duringoptimR1}}
    \end{subfigure}\par
    \begin{subfigure}{0.32\textwidth}
      \includegraphics[width=\textwidth]{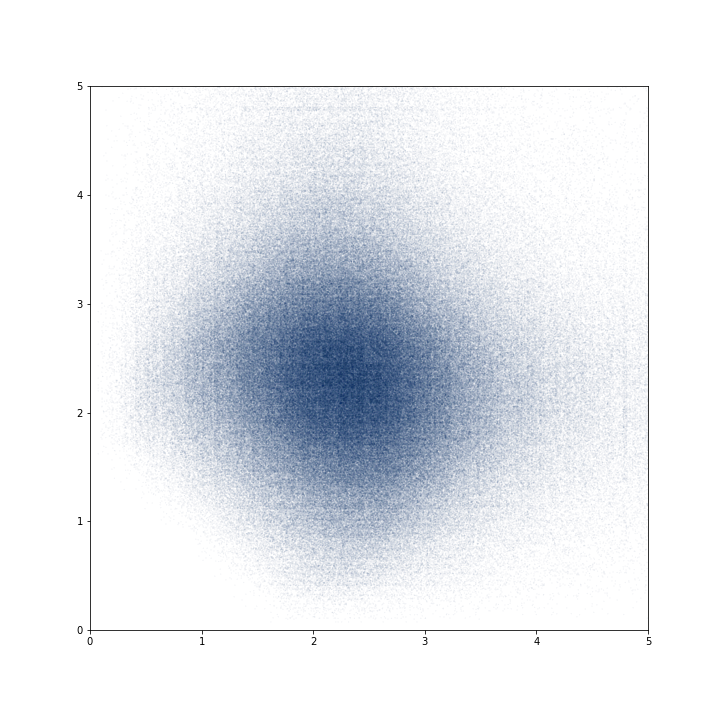}
      \caption{radial, iteration 30 \label{fig:duringoptimR30}}
    \end{subfigure}\par
    \begin{subfigure}{0.32\textwidth}
      \includegraphics[width=\textwidth]{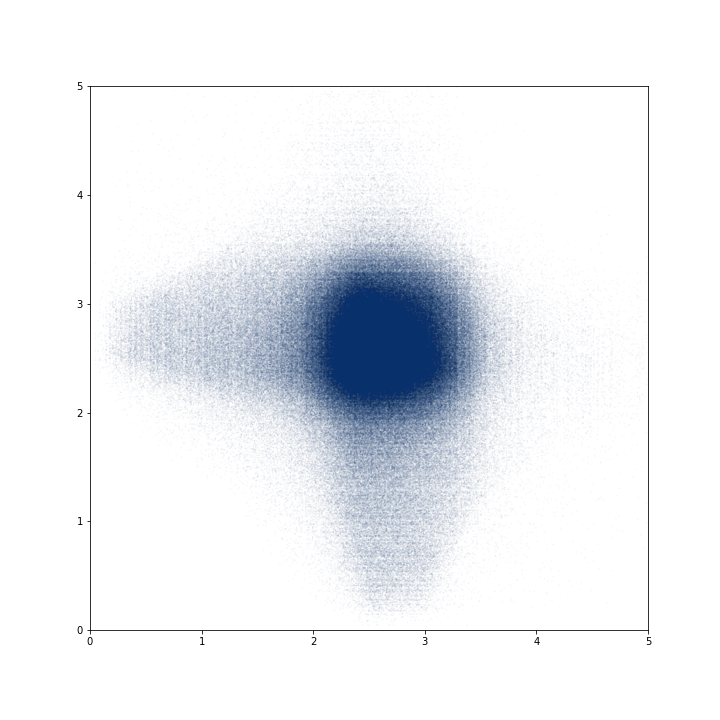}
      \caption{radial, iteration 50 \label{fig:duringoptimR50}}
    \end{subfigure}
  \end{multicols}
  \caption{Transport along optimization for $\mu_1$, $M=10$, $K=10000$, $N=27$, $\beta_0 = 0$, $\Delta t_0 = 10^{-4}$.
  In figures of column (a) is showed {\small $\frac 1 {MK} \sum_{k=1}^{K} \sum_{m = 1}^{M} \delta_{x^k_{m,1},x^k_{m,2}}$}. In figures of column (b) is showed
  {\small $\frac 1 {M(M-1)K} \sum_{k=1}^{K} \sum_{m \neq m' = 1}^{M} \delta_{x^k_{m,1},x^k_{m',1}}$}. In figures of column (c) is showed {\small $\frac 1 {M(M-1)K} \sum_{k=1}^{K} \sum_{m \neq m' = 1}^{M} \delta_{|x^k_m|,|x^k_{m'}|}$}, where {\small $|x^k_{m}| = \sqrt{\sum_{i=1}^3(x^k_{m,i})^2}$}. The evolution of the corresponding cost can be seen in Figure \ref{fig:optimNmu110}.
  \label{fig:duringoptim1}}
\end{figure}

\begin{figure}[htp]
  \centering
  \begin{multicols}{3}
    \captionsetup[sub]{labelformat=anumber}
    \begin{subfigure}{0.32\textwidth}
      \setcounter{subfigure}{3}
      \includegraphics[width=\textwidth]{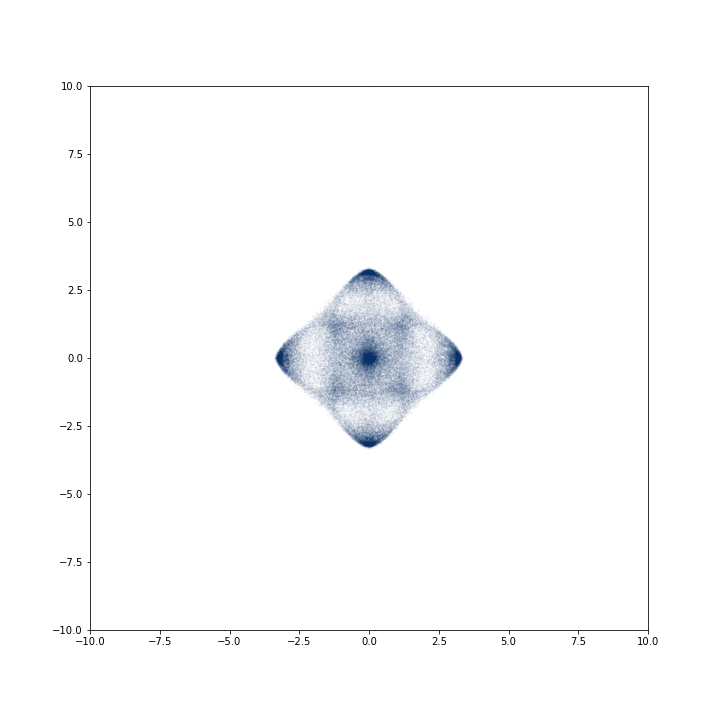}
      \caption{plane XY, iteration 100 \label{fig:duringoptimXY100}}
    \end{subfigure}\par
    \begin{subfigure}{0.32\textwidth}
      \includegraphics[width=\textwidth]{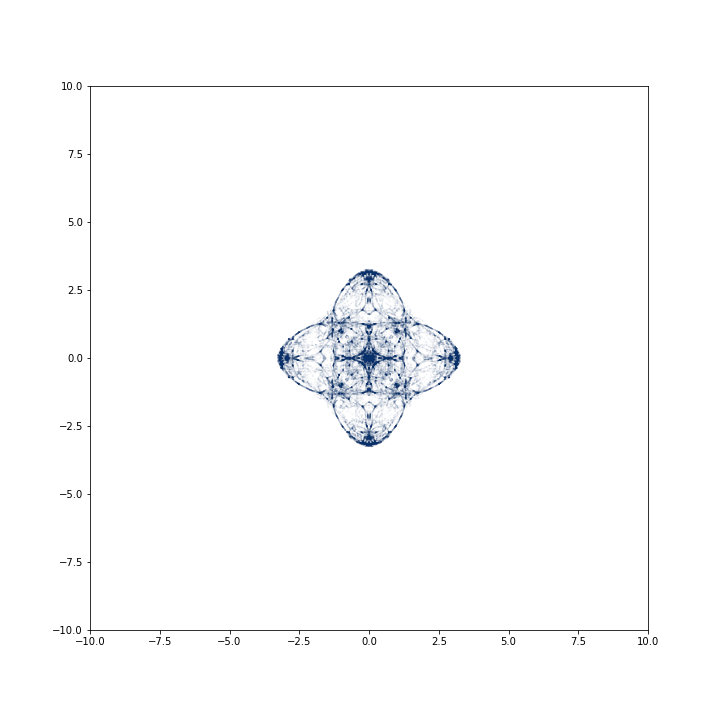}
      \caption{plane XY, iteration 1000 \label{fig:duringoptimXY1000}}
    \end{subfigure}\par
    \begin{subfigure}{0.32\textwidth}
      \includegraphics[width=\textwidth]{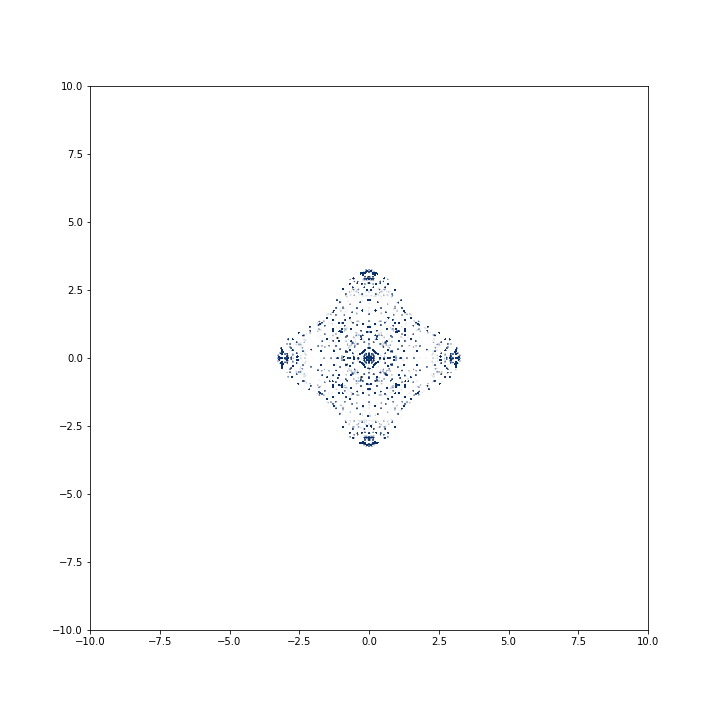}
      \caption{plane XY, iteration 20000 \label{fig:duringoptimXY20000}}
    \end{subfigure}\par

    \captionsetup[sub]{labelformat=bnumber}
    \setcounter{subfigure}{3}
    \begin{subfigure}{0.32\textwidth}
      \includegraphics[width=\textwidth]{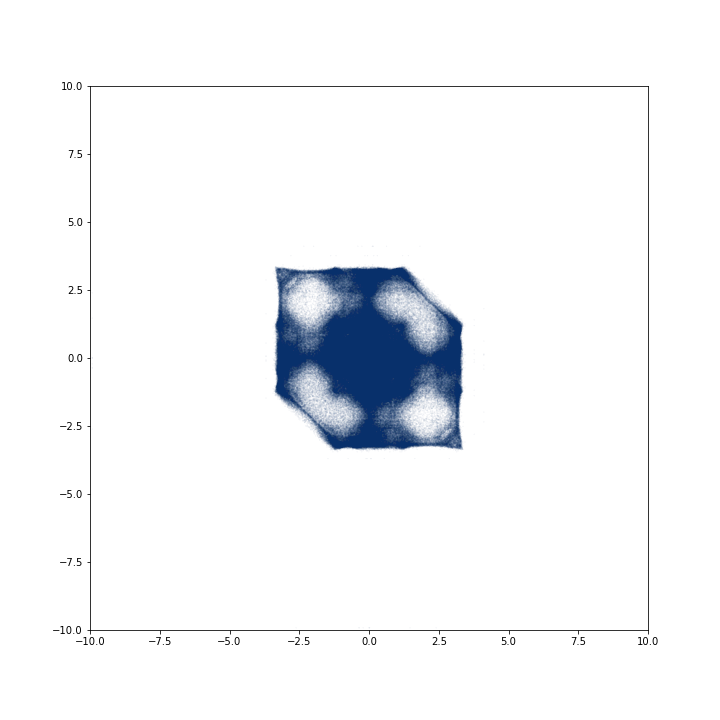}
      \caption{X axis, iteration 100 \label{fig:duringoptimX100}}
    \end{subfigure}\par
    \begin{subfigure}{0.32\textwidth}
      \includegraphics[width=\textwidth]{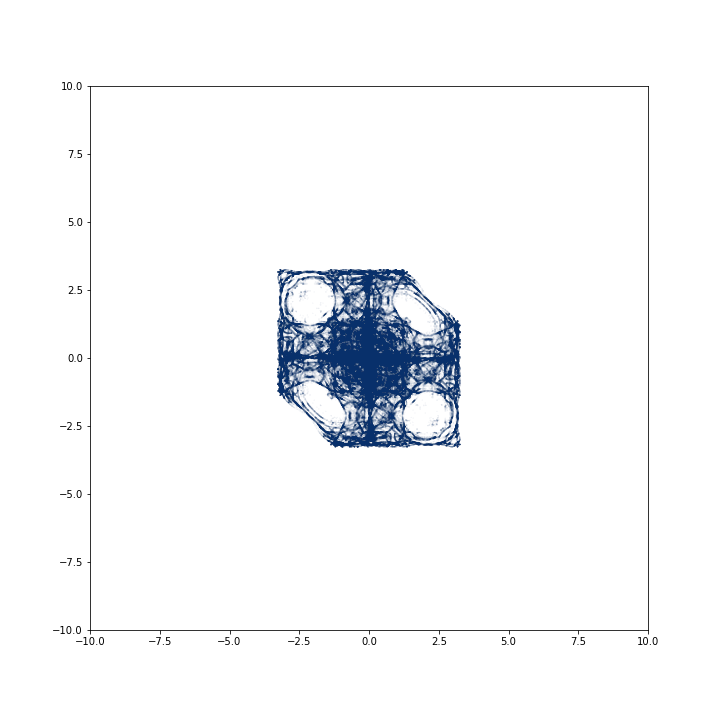}
      \caption{X axis, iteration 1000 \label{fig:duringoptimX1000}}
    \end{subfigure}\par
    \begin{subfigure}{0.32\textwidth}
      \includegraphics[width=\textwidth]{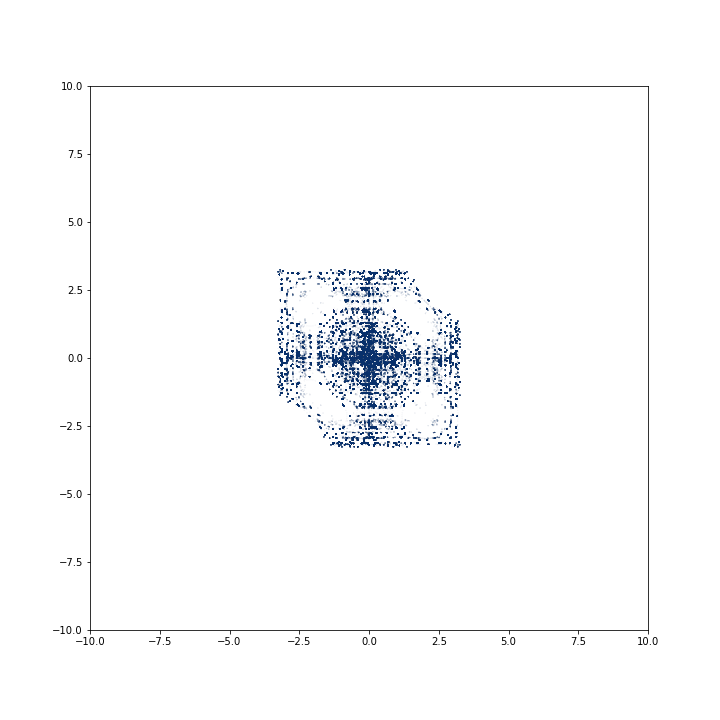}
      \caption{X axis, iteration 20000 \label{fig:duringoptimX20000}}
    \end{subfigure}\par

    \captionsetup[sub]{labelformat=cnumber}
    \setcounter{subfigure}{3}
    \begin{subfigure}{0.32\textwidth}
      \includegraphics[width=\textwidth]{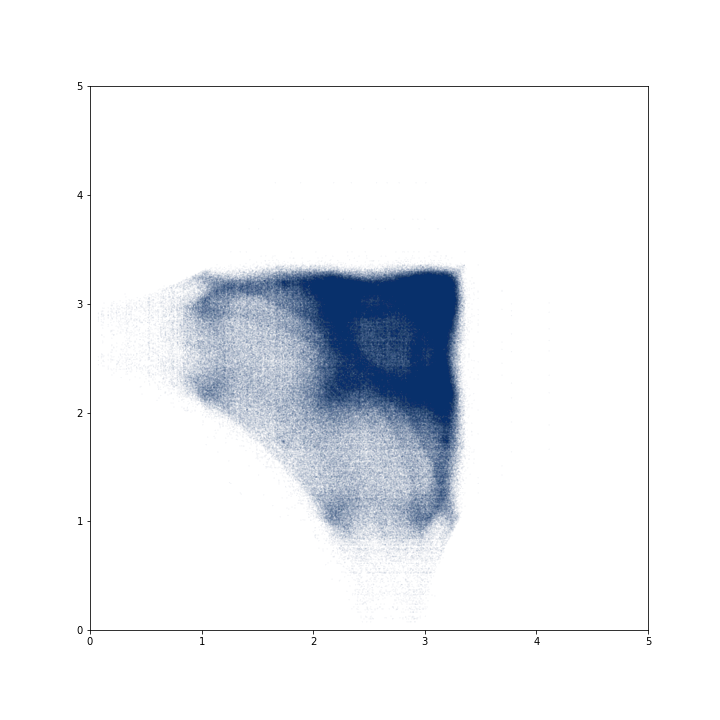}
      \caption{radial, iteration 100 \label{fig:duringoptimR100}}
    \end{subfigure}\par
    \begin{subfigure}{0.32\textwidth}
      \includegraphics[width=\textwidth]{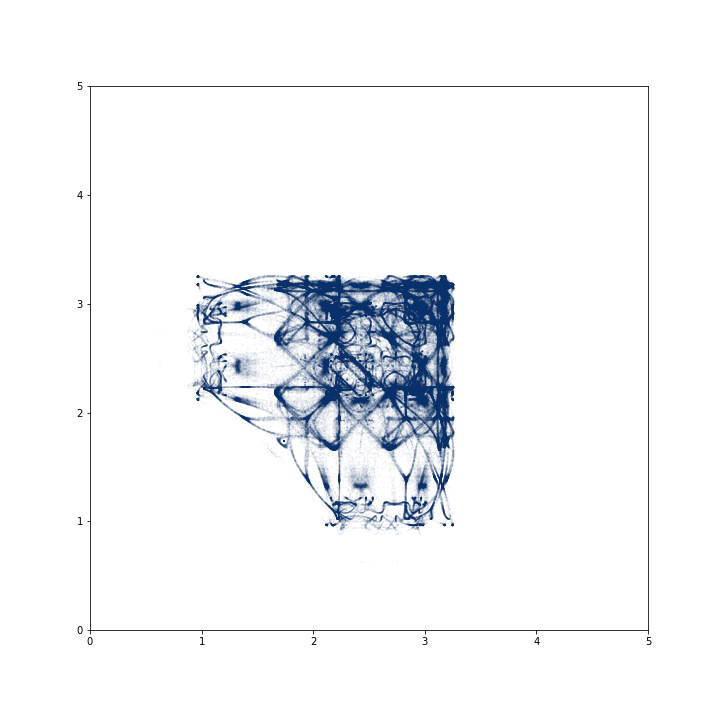}
      \caption{radial, iteration 1000 \label{fig:duringoptimR1000}}
    \end{subfigure}\par
    \begin{subfigure}{0.32\textwidth}
      \includegraphics[width=\textwidth]{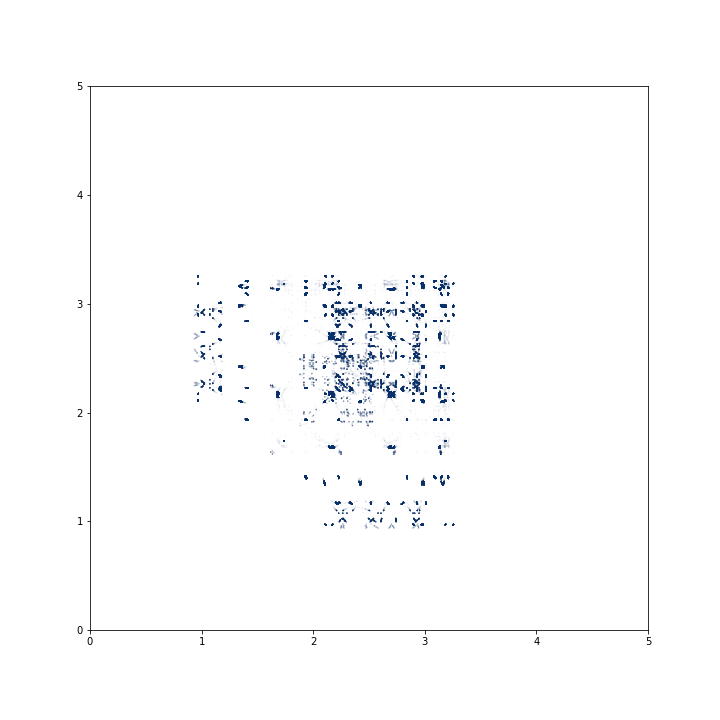}
      \caption{radial, iteration 20000 \label{fig:duringoptimR20000}}
    \end{subfigure}
  \end{multicols}
  \caption{Transport along optimization for $\mu_1$, $M=10$, $K=10000$, $N=27$, $\beta_0 = 0$, $\Delta t_0 = 10^{-4}$.
  In figures of column (a) is showed {\small $\frac 1 {MK} \sum_{k=1}^{K} \sum_{m = 1}^{M} \delta_{x^k_{m,1},x^k_{m,2}}$}. In figures of column (b) is showed
  {\small $\frac 1 {M(M-1)K} \sum_{k=1}^{K} \sum_{m \neq m' = 1}^{M} \delta_{x^k_{m,1},x^k_{m',1}}$}. In figures of column (c) is showed {\small $\frac 1 {M(M-1)K} \sum_{k=1}^{K} \sum_{m \neq m' = 1}^{M} \delta_{|x^k_m|,|x^k_{m'}|}$}, where {\small $|x^k_{m}| = \sqrt{\sum_{i=1}^3(x^k_{m,i})^2}$}. The evolution of the corresponding cost can be seen in Figure \ref{fig:optimNmu110}.
  \label{fig:duringoptim2}}
\end{figure}

%

\subsubsection{Minimas -- Figures \ref{fig:minimas1}, \ref{fig:minimas2}, \ref{fig:minimasnonsym} and \ref{fig:minimasnonsymXY}}\label{sect:minimas}

As $K$ increases, the symmetrized minimizers of Figures \ref{fig:minimas1} and \ref{fig:minimas2} tends to be visually more and more concentrated on some particular points. According to Table \ref{tbl:minimascost}, higher values of $K$ tends to have lower costs.

Some symmetrized visualizations of minimizers for MCOT problems for the non-symmetrical measures $\mu_2$ and $\mu_3$ are  presented in Figures \ref{fig:minimasnonsym} and \ref{fig:minimasnonsymXY}. In those cases, the 1D couplings obtained on each axis (X, Y or Z) are not the same (Figure \ref{fig:minimasnonsym}). A higher number of marginal laws $M$ seems to spread more the particules, although their 1D coupling still shows particles highly concentrated around a few values in the considered examples. Higher values of $N$ increases the concentration of the particles around fewer values in the $\mu_3$ examples. The planar representation of the minimizers for large $M$ (Figure \ref{fig:minimasnonsymXY}), shows that particules are not distributed spatially as a Normal function and tend to concentrate on some 1D curves (for the considered 2D projections) with a higher spreading than for lower values of $M$.

\begin{table}[tp!]
\begin{center}
  \begin{tabular}{c|llllll}
  $K$ & 40 & 80 & 160 & 320 & 1000 & 10000 \\
  \hline
  cost & 12.2558198 & 12.1747815 & 12.1457150 & 12.0916662 & 12.0821615 & 12.0785749 \\
  lower cost & 12.1981977 & 12.0864398 & 12.0862042 & 12.0855486 & 12.0821615 & 12.0785745
  \end{tabular}
 \end{center}
 \caption{Values of the regularized Coulomb cost (see here-named paragraph in Section \ref{sect:3Dtestdesign}) for the MCOT problem with $\mu_1$, $M=10$, $N=27$, $\Delta t_0 = 10^{-4}$, $\beta_0  = 0$ and $K$ ranging from 40 to 10000. The \emph{cost} line corresponds to the value of the regularized cost associated to the minimizing process at iteration 20000 (which also corresponds to the minimizers represented in the graphs of Figures \ref{fig:minimas1} and \ref{fig:minimas2}). The \emph{lower cost} line corresponds to the lower value of the regularized cost encountered by the minmizing process before or at iteration 20000 for each value of $K$. \label{tbl:minimascost}}
\end{table}

\begin{figure}[htp]
  \centering
  \begin{multicols}{3}
    \captionsetup[sub]{labelformat=anumber}
    \setcounter{subfigure}{0}
    \begin{subfigure}{0.32\textwidth}
      \includegraphics[width=\textwidth]{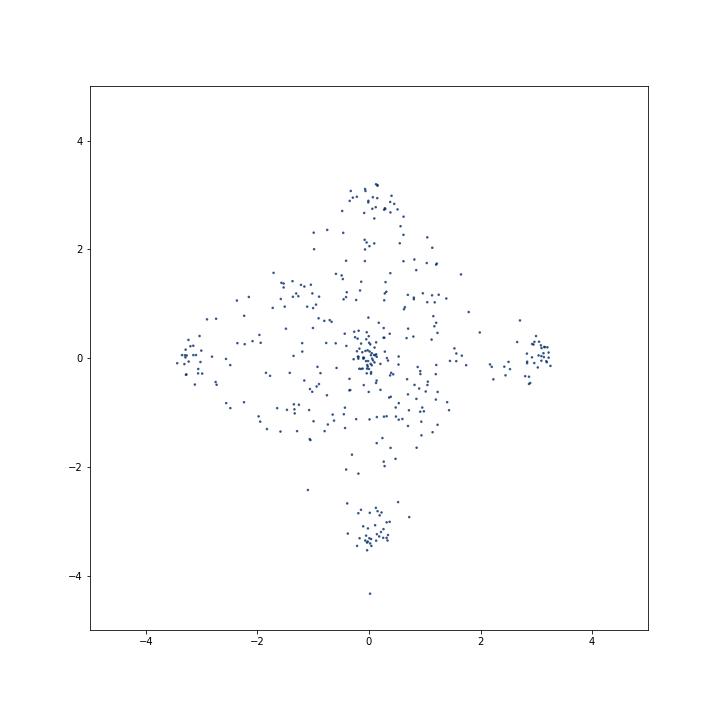}
      \caption{plane XY, K=40 \label{fig:minimasXYK40}}
    \end{subfigure}\par
    \begin{subfigure}{0.32\textwidth}
      \includegraphics[width=\textwidth]{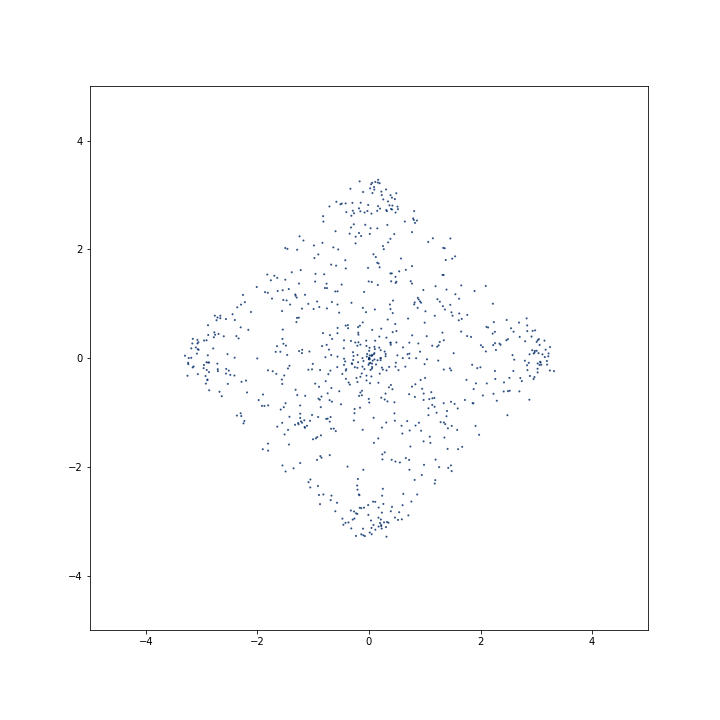}
      \caption{plane XY, K=80 \label{fig:minimasXYK80}}
    \end{subfigure}\par
    \begin{subfigure}{0.32\textwidth}
      \includegraphics[width=\textwidth]{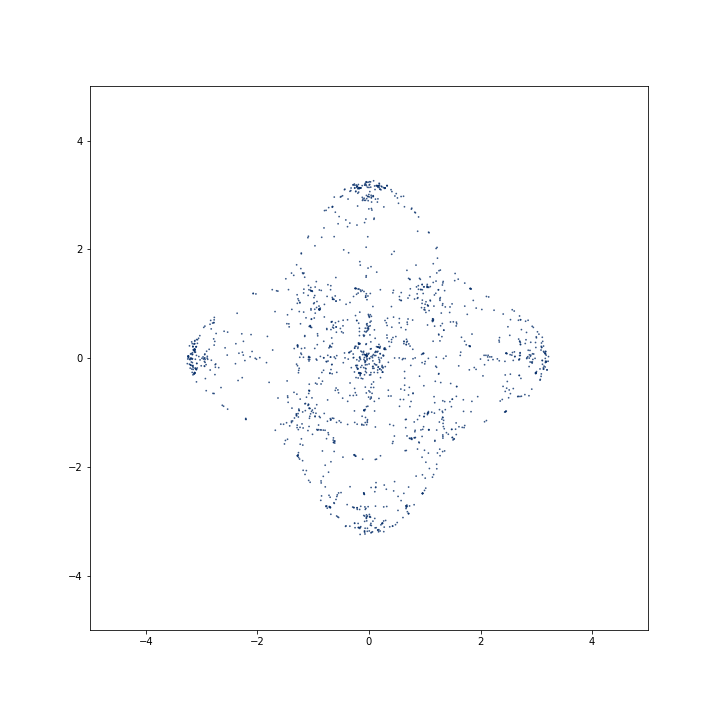}
      \caption{plane XY, K=160 \label{fig:minimasXYK160}}
    \end{subfigure}\par

    \captionsetup[sub]{labelformat=bnumber}
    \setcounter{subfigure}{0}
    \begin{subfigure}{0.32\textwidth}
      \includegraphics[width=\textwidth]{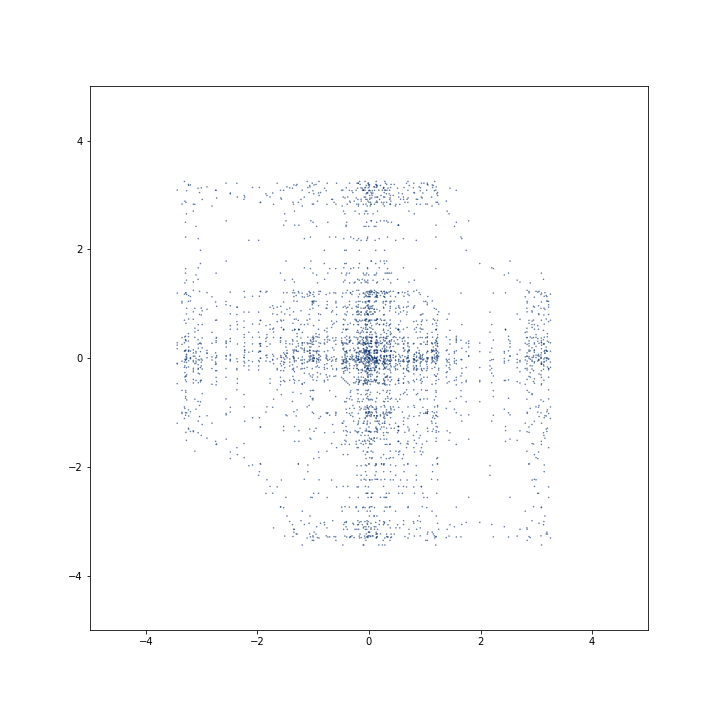}
      \caption{X axis, K=40 \label{fig:minimasXK40}}
    \end{subfigure}\par
    \begin{subfigure}{0.32\textwidth}
      \includegraphics[width=\textwidth]{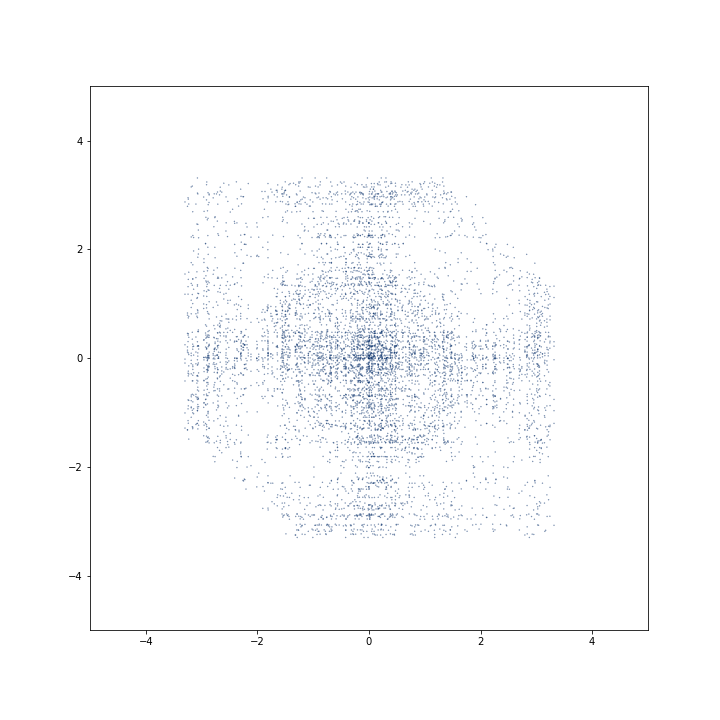}
      \caption{X axis, K=80 \label{fig:minimasXK80}}
    \end{subfigure}\par
    \begin{subfigure}{0.32\textwidth}
      \includegraphics[width=\textwidth]{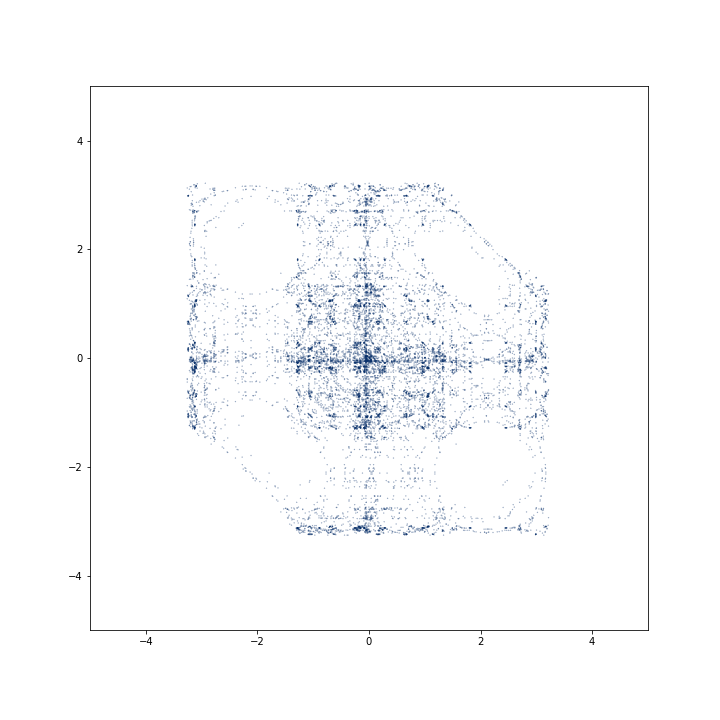}
      \caption{X axis, K=160 \label{fig:minimasXK160}}
    \end{subfigure}\par

    \captionsetup[sub]{labelformat=cnumber}
    \setcounter{subfigure}{0}
    \begin{subfigure}{0.32\textwidth}
      \includegraphics[width=\textwidth]{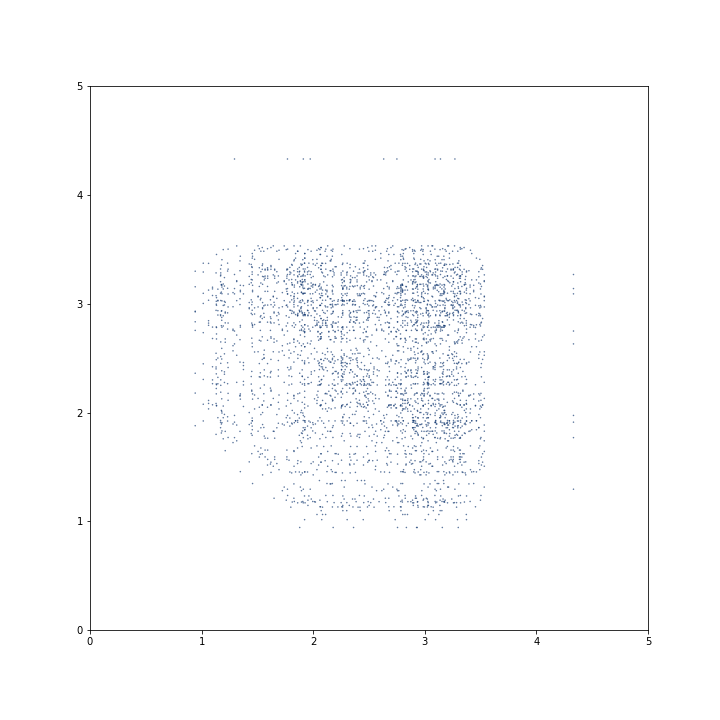}
      \caption{radial, K=40 \label{fig:minimasRK40}}
    \end{subfigure}\par
    \begin{subfigure}{0.32\textwidth}
      \includegraphics[width=\textwidth]{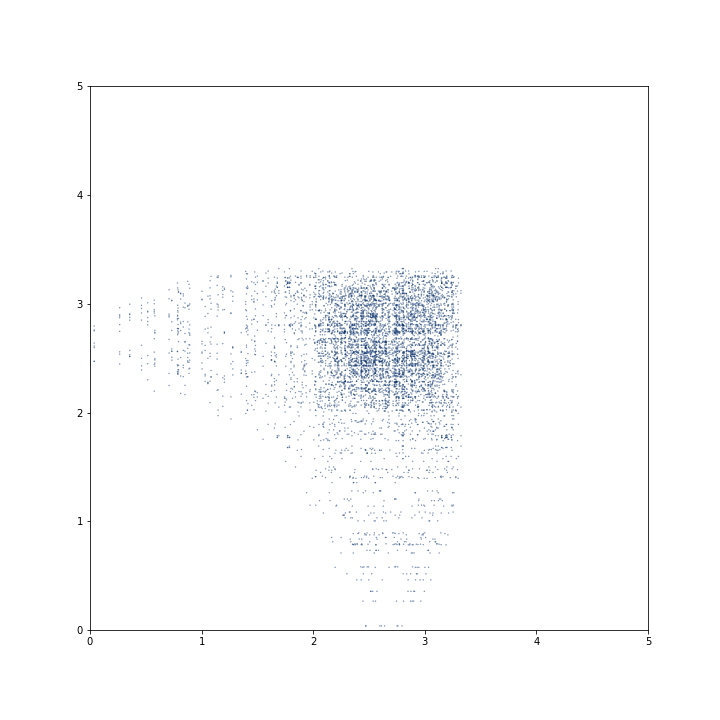}
      \caption{radial, K=80 \label{fig:minimasRK80}}
    \end{subfigure}\par
    \begin{subfigure}{0.32\textwidth}
      \includegraphics[width=\textwidth]{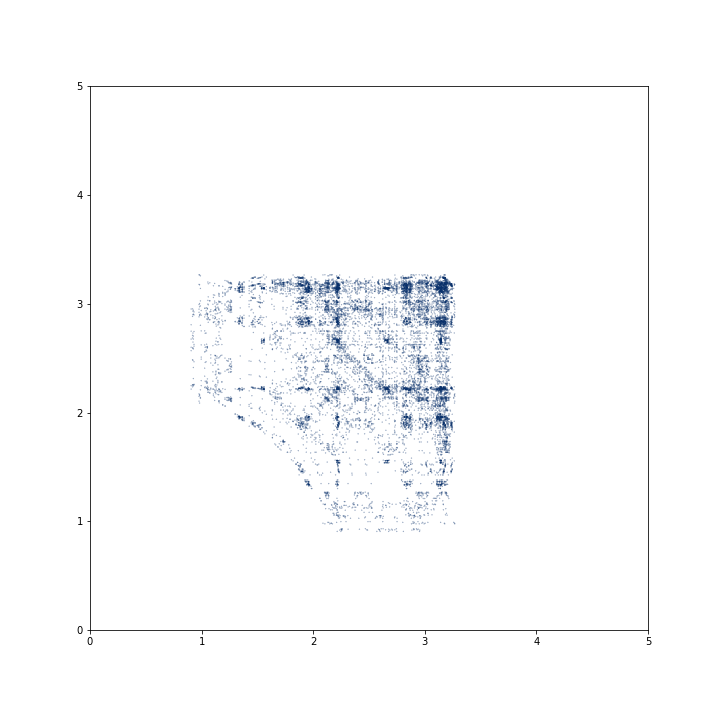}
      \caption{radial, K=160 \label{fig:minimasRK160}}
    \end{subfigure}
  \end{multicols}
  \caption{Optimal transport with $\mu_1$, $M=10$, $N=27$, $\beta_0  = 0$ and $\Delta t_0 = 10^{-4}$, for $K= 40, 80, 160$.
  In figures of column (a) is showed {\small $\frac 1 {MK} \sum_{k=1}^{K} \sum_{m = 1}^{M} \delta_{x^k_{m,1},x^k_{m,2}}$}. In figures of column (b) is showed
  {\small $\frac 1 {M(M-1)K} \sum_{k=1}^{K} \sum_{m \neq m' = 1}^{M} \delta_{x^k_{m,1},x^k_{m',1}}$}. In figures of column (c) is showed {\small $\frac 1 {M(M-1)K} \sum_{k=1}^{K} \sum_{m \neq m' = 1}^{M} \delta_{|x^k_{m}|,|x^k_{m'}|}$}, where {\small $|x^k_{m}| = \sqrt{\sum_{i=1}^3(x^k_{m,i})^2}$}. Corresponding costs can be found in Table \ref{tbl:minimascost}.
  \label{fig:minimas1}}
\end{figure}

\begin{figure}[htp]
  \centering
  \begin{multicols}{3}
    \captionsetup[sub]{labelformat=anumber}
    \begin{subfigure}{0.32\textwidth}
      \setcounter{subfigure}{3}
      \includegraphics[width=\textwidth]{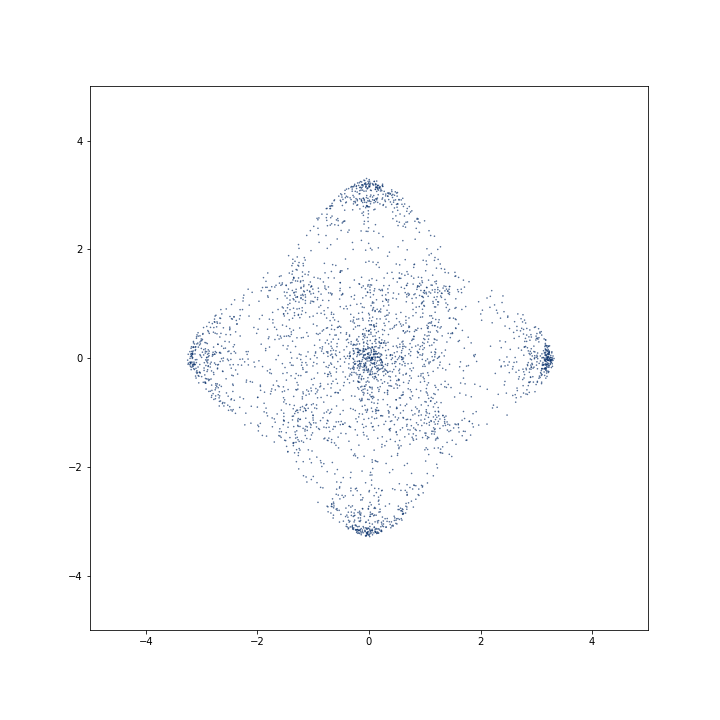}
      \caption{plane XY, K=320 \label{fig:minimasXYK320}}
    \end{subfigure}
    \begin{subfigure}{0.32\textwidth}
      \includegraphics[width=\textwidth]{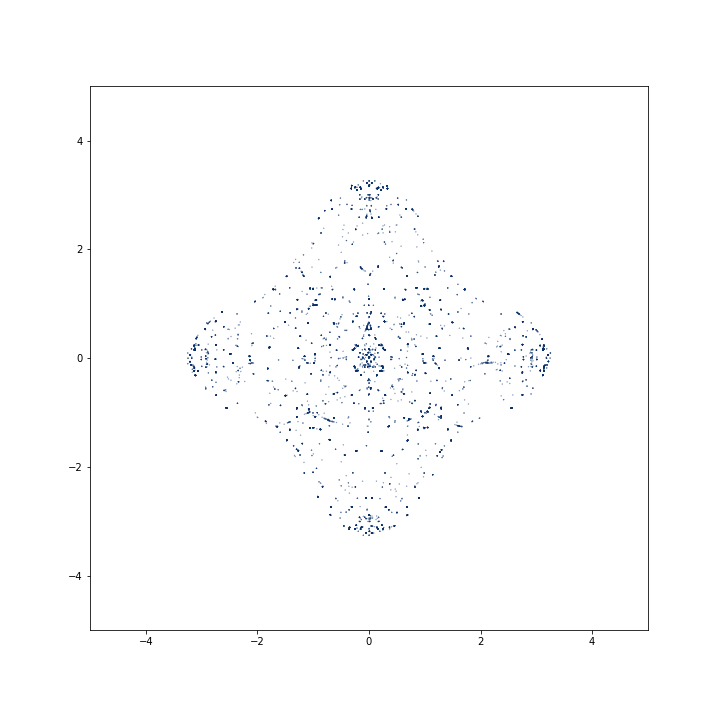}
      \caption{plane XY, K=1000 \label{fig:minimasXYK1000}}
    \end{subfigure}
    \begin{subfigure}{0.32\textwidth}
      \includegraphics[width=\textwidth]{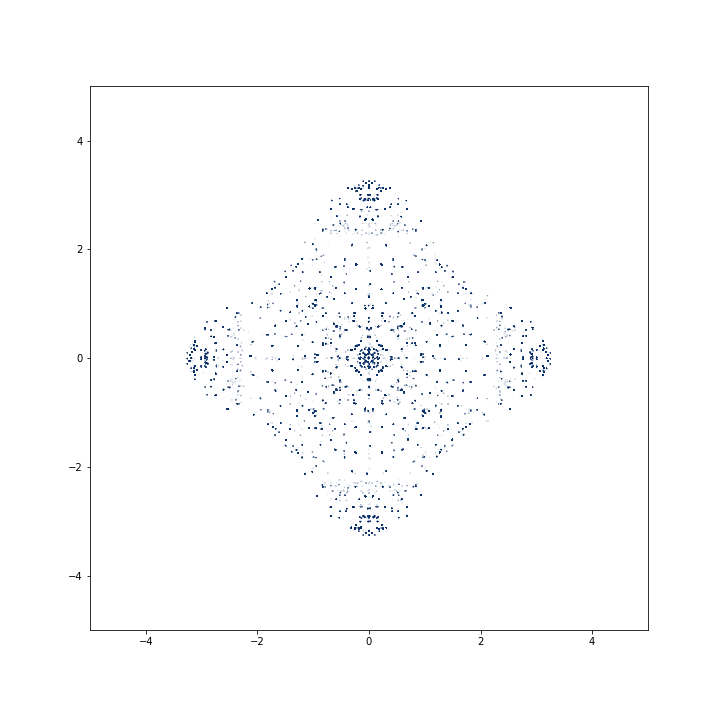}
      \caption{plane XY, K=10000 \label{fig:minimasXYK10000}}
    \end{subfigure}

    \captionsetup[sub]{labelformat=bnumber}
    \setcounter{subfigure}{3}
    \begin{subfigure}{0.32\textwidth}
      \includegraphics[width=\textwidth]{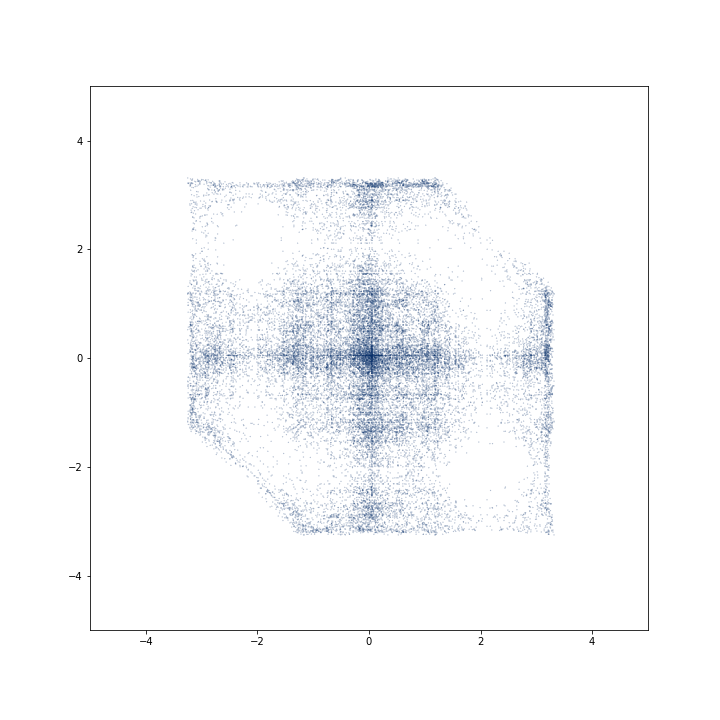}
      \caption{X axis, K=320 \label{fig:minimasXK320}}
    \end{subfigure}
    \begin{subfigure}{0.32\textwidth}
      \includegraphics[width=\textwidth]{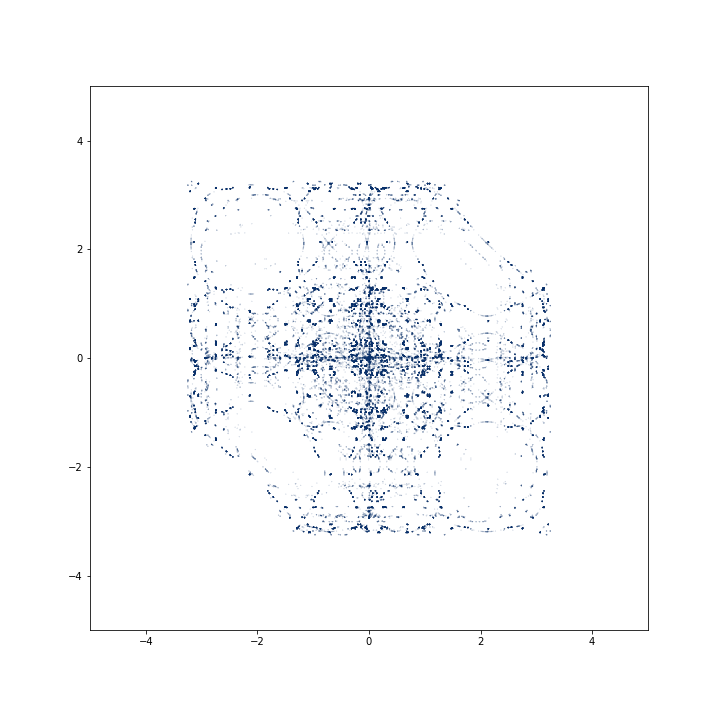}
      \caption{X axis, K=1000 \label{fig:minimasXK1000}}
    \end{subfigure}
    \begin{subfigure}{0.32\textwidth}
      \includegraphics[width=\textwidth]{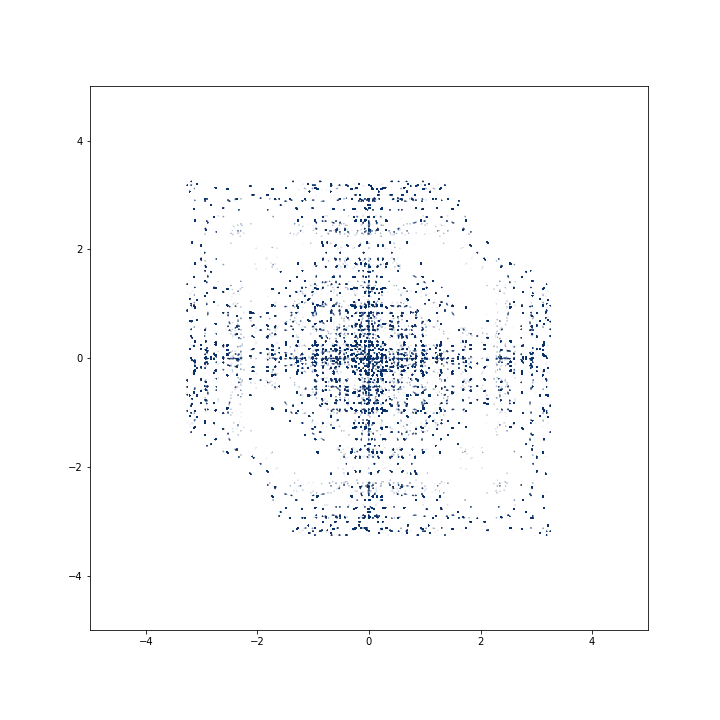}
      \caption{X axis, K=10000 \label{fig:minimasXK10000}}
    \end{subfigure}

    \captionsetup[sub]{labelformat=cnumber}
    \setcounter{subfigure}{3}
    \begin{subfigure}{0.32\textwidth}
      \includegraphics[width=\textwidth]{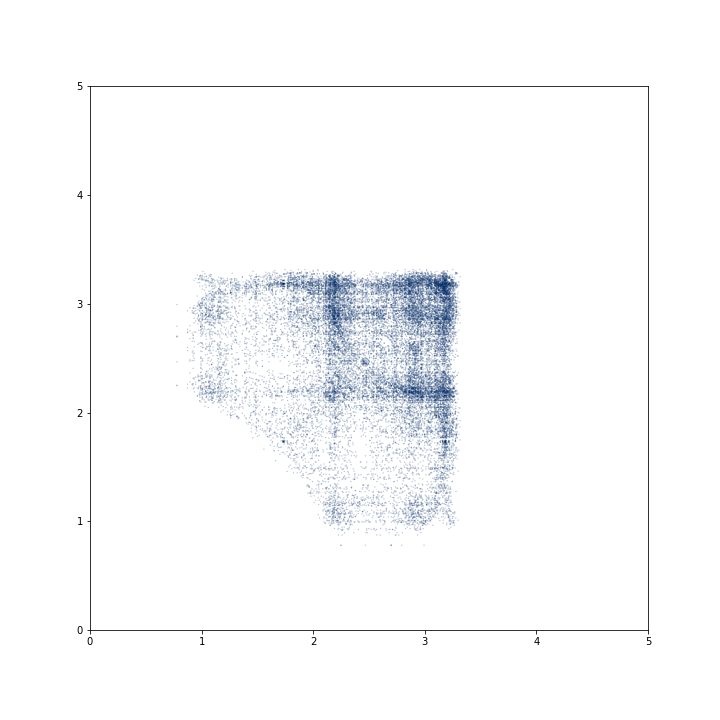}
      \caption{radial, K=320 \label{fig:minimasRK320}}
    \end{subfigure}
    \begin{subfigure}{0.32\textwidth}
      \includegraphics[width=\textwidth]{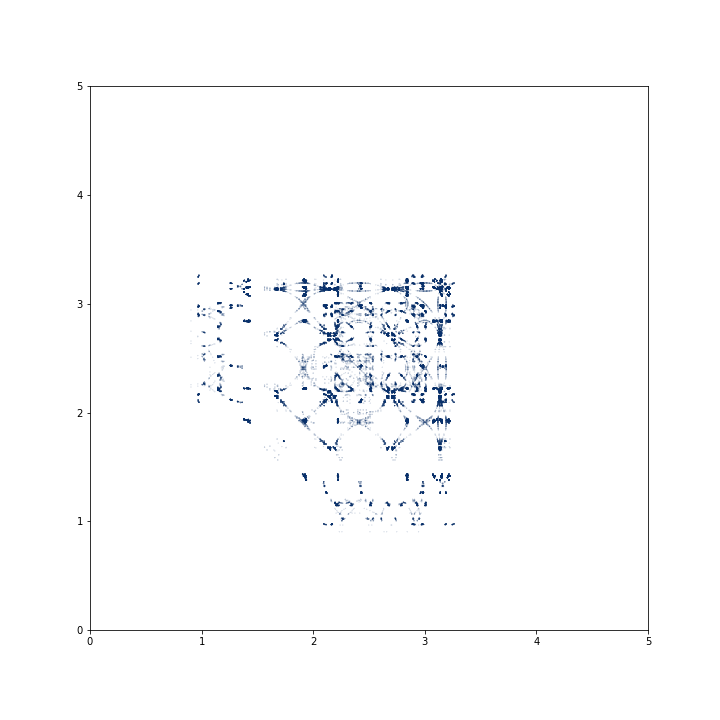}
      \caption{radial, K=1000 \label{fig:minimasRK1000}}
    \end{subfigure}
    \begin{subfigure}{0.32\textwidth}
      \includegraphics[width=\textwidth]{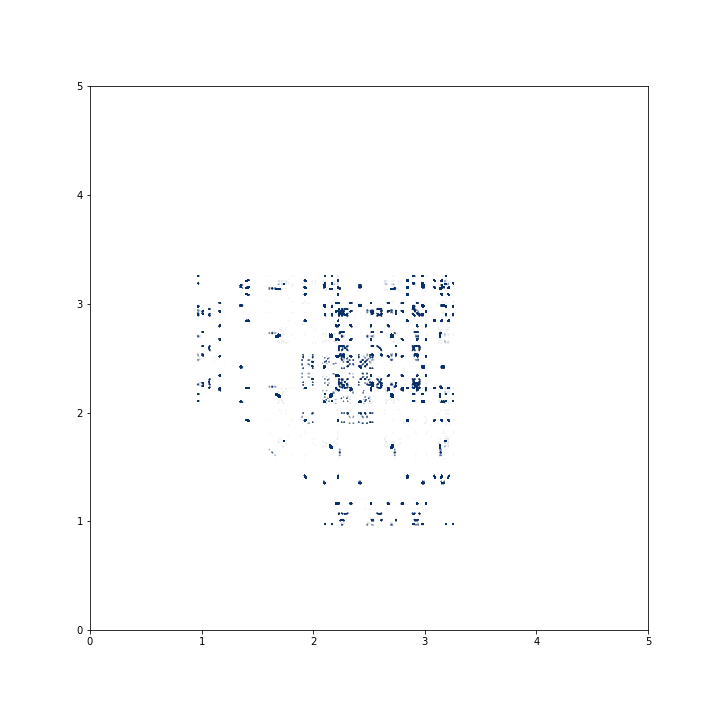}
      \caption{radial, K=10000 \label{fig:minimasRK10000}}
    \end{subfigure}
  \end{multicols}
  \caption{
  Optimal transport with $\mu_1$, $M=10$, $N=27$, $\beta_0  = 0$ and $\Delta t_0 = 10^{-4}$, for $K= 320, 1000, 10000$.
  In figures of column (a) is showed {\small $\frac 1 {MK} \sum_{k=1}^{K} \sum_{m = 1}^{M} \delta_{x^k_{m,1},x^k_{m,2}}$}. In figures of column (b) is showed
  {\small $\frac 1 {M(M-1)K} \sum_{k=1}^{K} \sum_{m \neq m' = 1}^{M} \delta_{x^k_{m,1},x^k_{m',1}}$}. In figures of column (c) is showed {\small $\frac 1 {M(M-1)K} \sum_{k=1}^{K} \sum_{m \neq m' = 1}^{M} \delta_{|x^k_{m}|,|x^k_{m'}|}$}, where {\small $|x^k_{m}| = \sqrt{\sum_{i=1}^3(x^k_{m,i})^2}$}. Corresponding costs can be found in Table \ref{tbl:minimascost}.
  \label{fig:minimas2}}
\end{figure}

\begin{figure}[htp]
  \vspace{-1.5cm}
  \centering
  \begin{multicols}{3}
    \captionsetup[sub]{labelformat=anumber}
    \begin{subfigure}{0.29\textwidth}
      \includegraphics[width=\textwidth]{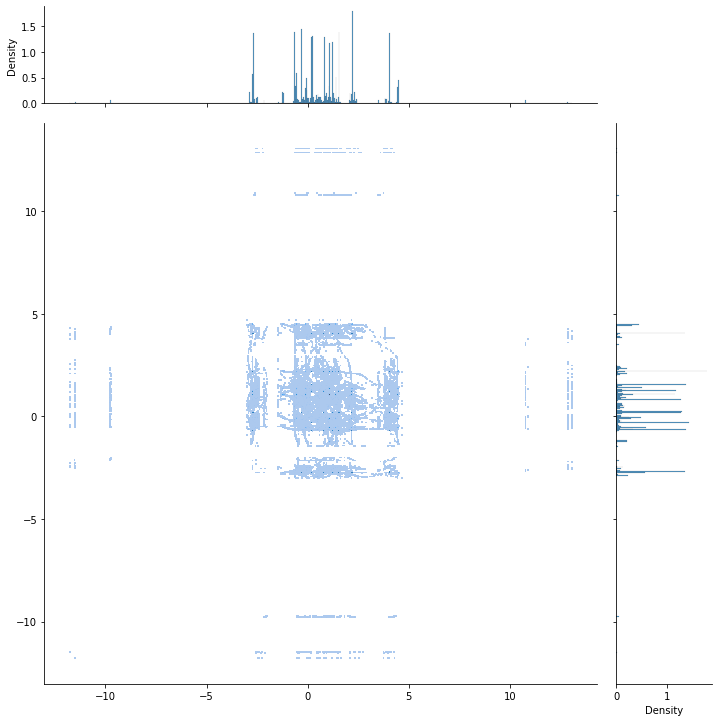}
      \caption{{\scriptsize $\mu_2$, X axis, $M=10$, $N=27$ \label{fig:minimasns2X1027}}}
    \end{subfigure}
    \begin{subfigure}{0.29\textwidth}
      \includegraphics[width=\textwidth]{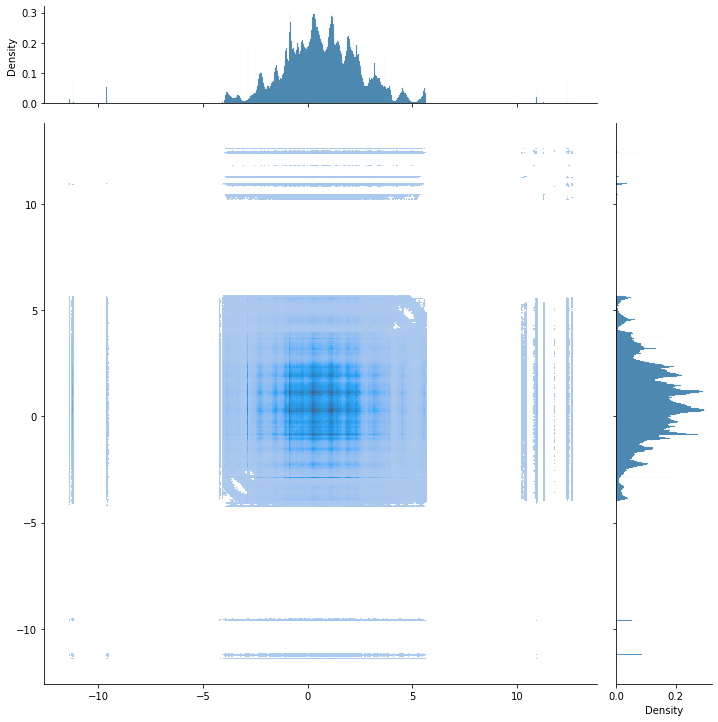}
      \caption{{\scriptsize $\mu_2$, X axis, $M=100$, $N=27$ \label{fig:minimasns2X10027}}}
    \end{subfigure}
    \begin{subfigure}{0.29\textwidth}
      \includegraphics[width=\textwidth]{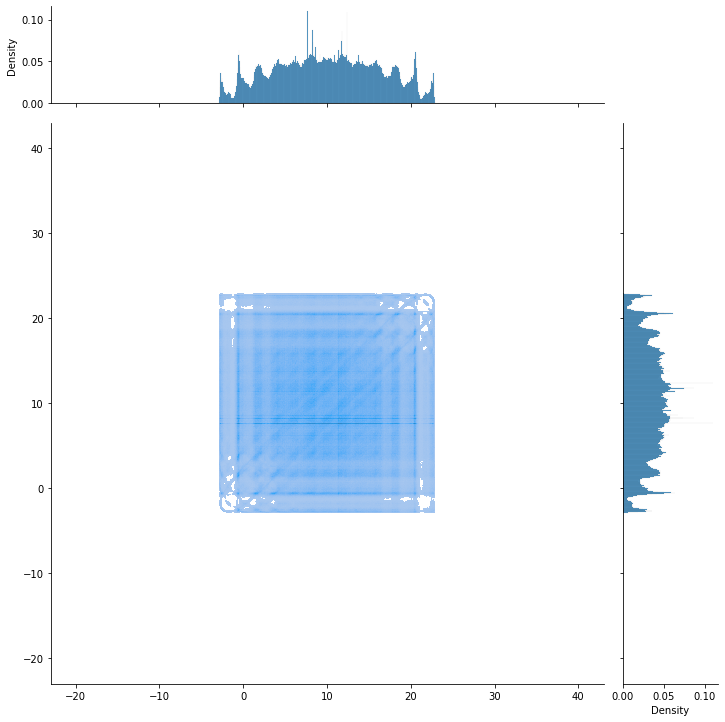}
      \caption{{\scriptsize $\mu_3$, X axis, $M=100$, $N=27$ \label{fig:minimasns3X10027}}}
    \end{subfigure}
    \begin{subfigure}{0.29\textwidth}
      \includegraphics[width=\textwidth]{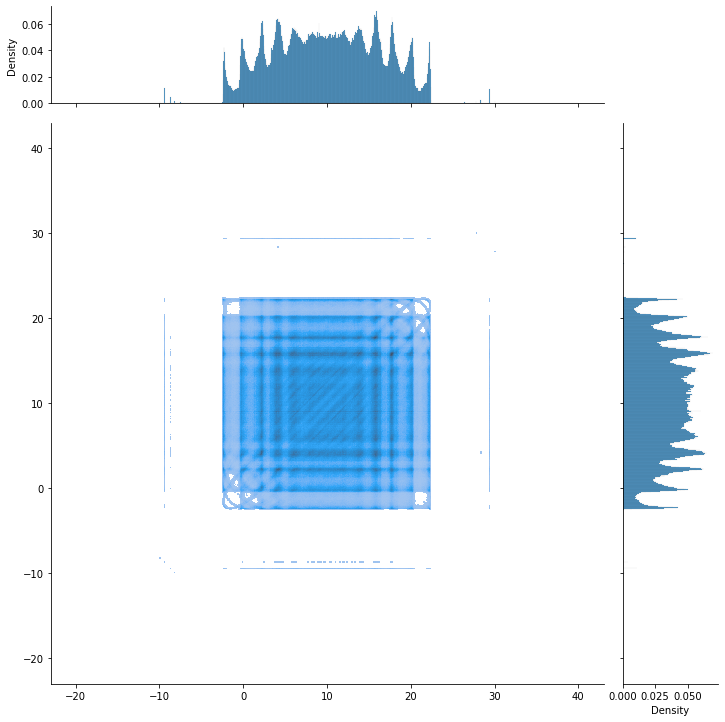}
      \caption{{\scriptsize $\mu_3$, X axis, $M=100$, $N=52$ \label{fig:minimasns3X10052}}}
    \end{subfigure}

    \captionsetup[sub]{labelformat=bnumber}
    \setcounter{subfigure}{0}
    \begin{subfigure}{0.29\textwidth}
      \includegraphics[width=\textwidth]{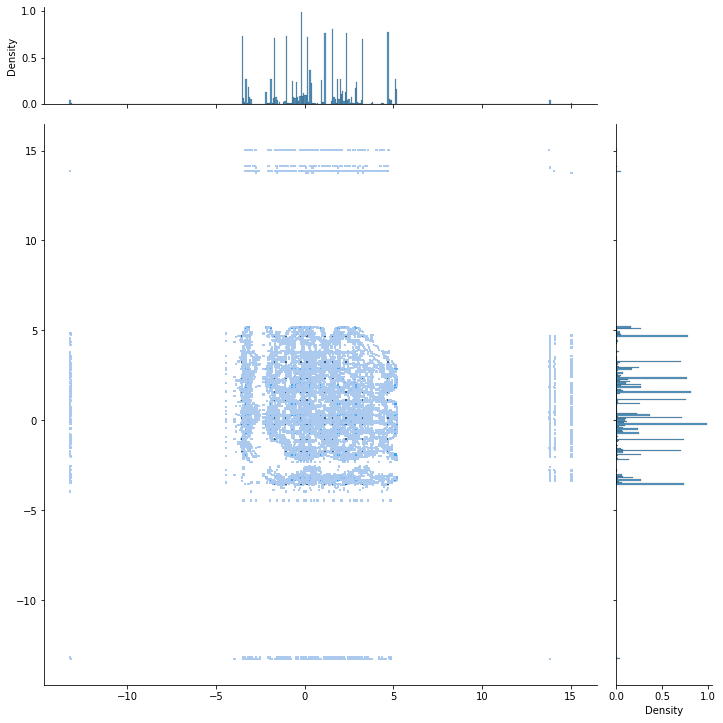}
      \caption{{\scriptsize $\mu_2$, Y axis, $M=10$, $N=27$ \label{fig:minimasns2Y1027}}}
    \end{subfigure}
    \begin{subfigure}{0.29\textwidth}
      \includegraphics[width=\textwidth]{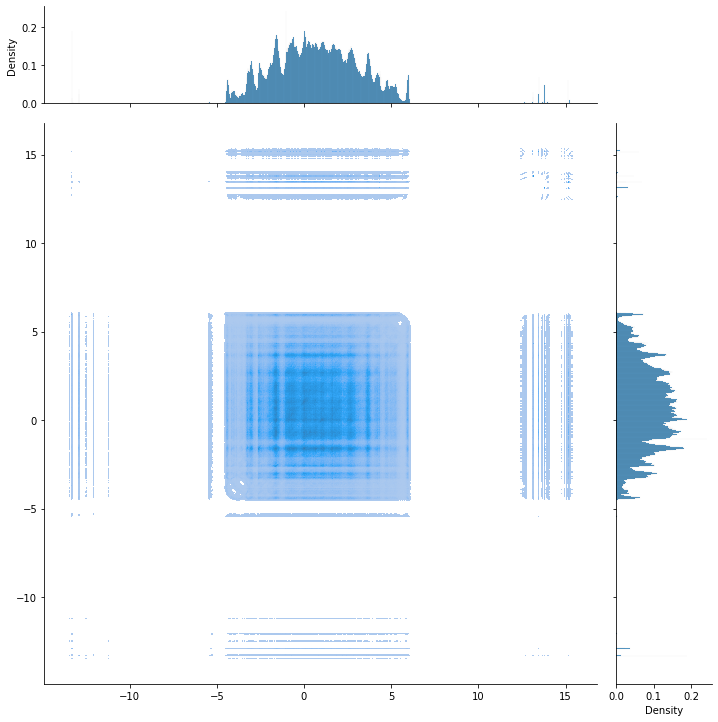}
      \caption{{\scriptsize $\mu_2$, Y axis, $M=100$, $N=27$ \label{fig:minimasns2Y10027}}}
    \end{subfigure}
    \begin{subfigure}{0.29\textwidth}
      \includegraphics[width=\textwidth]{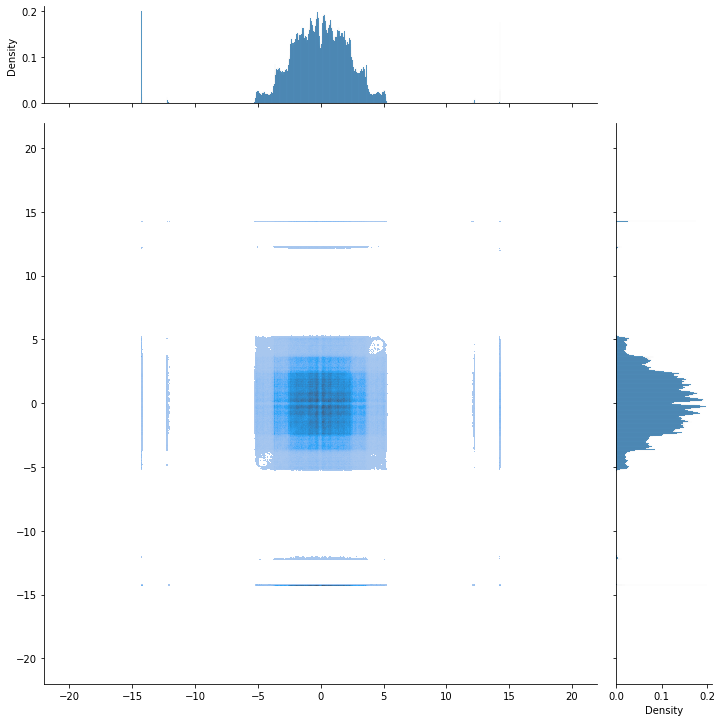}
      \caption{{\scriptsize $\mu_3$, Y axis, $M=100$, $N=27$ \label{fig:minimasns3Y10027}}}
    \end{subfigure}
    \begin{subfigure}{0.29\textwidth}
      \includegraphics[width=\textwidth]{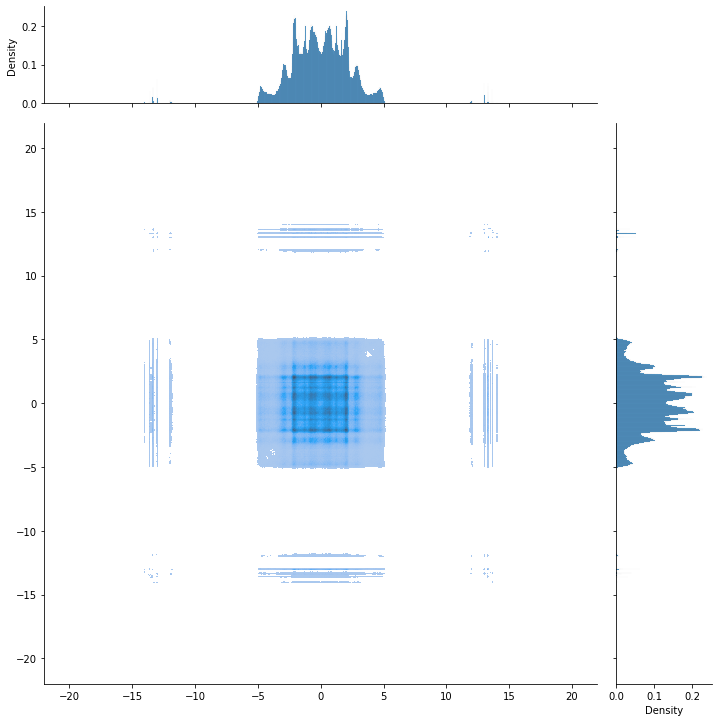}
      \caption{{\scriptsize $\mu_3$, Y axis, $M=100$, $N=52$ \label{fig:minimasns3Y10052}}}
    \end{subfigure}

    \captionsetup[sub]{labelformat=cnumber}
    \setcounter{subfigure}{0}
    \begin{subfigure}{0.29\textwidth}
      \includegraphics[width=\textwidth]{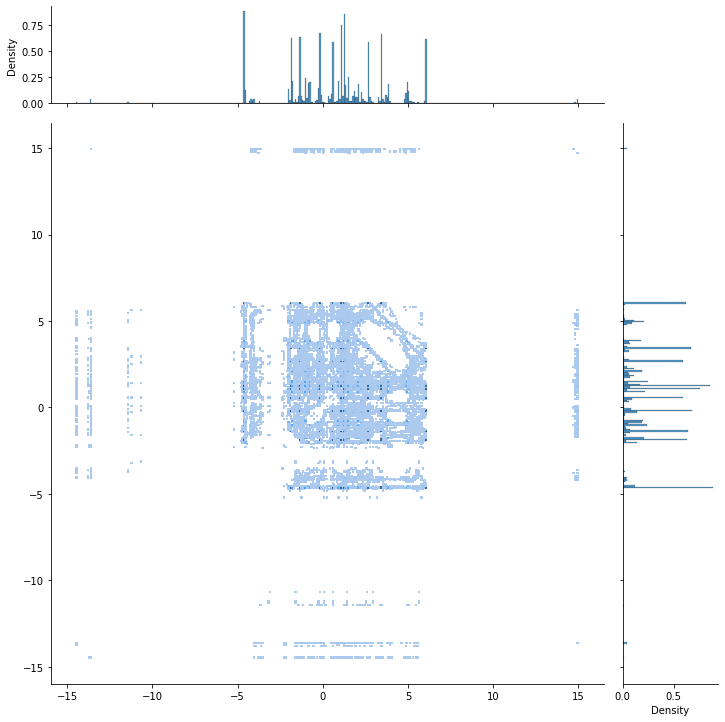}
      \caption{{\scriptsize $\mu_2$, Z axis, $M=10$, $N=27$ \label{fig:minimasns2Z1027}}}
    \end{subfigure}
    \begin{subfigure}{0.29\textwidth}
      \includegraphics[width=\textwidth]{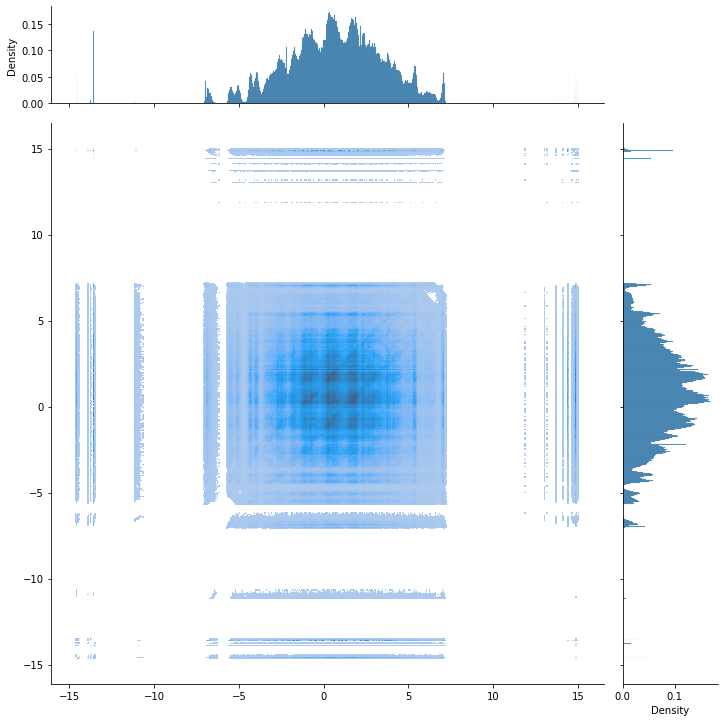}
      \caption{{\scriptsize $\mu_2$, Z axis, $M=100$, $N=27$ \label{fig:minimasns2Z10027}}}
    \end{subfigure}
    \begin{subfigure}{0.29\textwidth}
      \includegraphics[width=\textwidth]{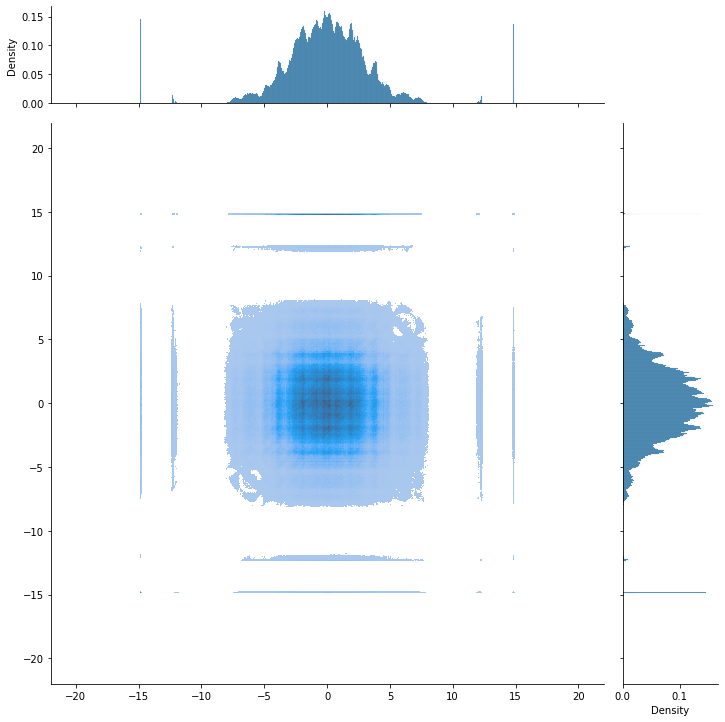}
      \caption{{\scriptsize $\mu_3$, Z axis, $M=100$, $N=27$ \label{fig:minimasns3Z10027}}}
    \end{subfigure}
    \begin{subfigure}{0.29\textwidth}
      \includegraphics[width=\textwidth]{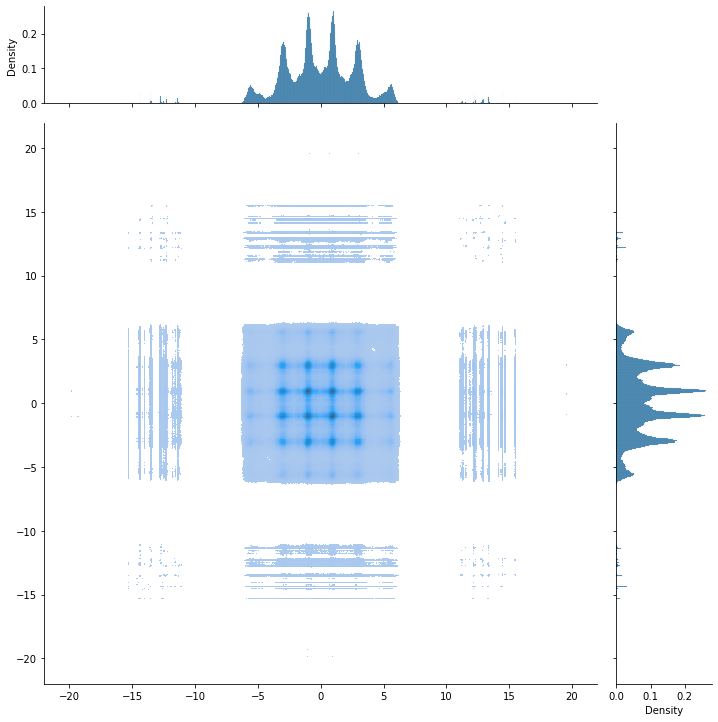}
      \caption{{\scriptsize $\mu_3$, Z axis, $M=100$, $N=52$ \label{fig:minimasns3Z10052}}}
    \end{subfigure}
  \end{multicols}
  \vspace{-1em}
  \caption{
  Optimal transport for $\mu_2$ and $\mu_3$, $M = 10, 100$, $N= 27, 52$, $\beta_0 = 0$, $K=10000$ and $\Delta t_0 = 10^{-4}$
  In figures of column (a) is showed {\small $\frac 1 {MK} \sum_{k=1}^{K} \sum_{m = 1}^{M} \delta_{x^k_{m,1},x^k_{m,2}}$}. In figures of column (b) is showed
  {\small $\frac 1 {M(M-1)K} \sum_{k=1}^{K} \sum_{m \neq m' = 1}^{M} \delta_{x^k_{m,1},x^k_{m',1}}$}. In figures of column (c) is showed {\small $\frac 1 {M(M-1)K} \sum_{k=1}^{K} \sum_{m \neq m' = 1}^{M} \delta_{|x^k_{m}|,|x^k_{m'}|}$}, where {\small $|x^k_{m}| = \sqrt{\sum_{i=1}^3(x^k_{m,i})^2}$}. In order to better distinguish between areas of low and high particles density, plots are represented as 2D histograms.
  \label{fig:minimasnonsym}}
\end{figure}

\begin{figure}[htp]
  \centering
  \begin{subfigure}{0.48\textwidth}
    \centering
    \includegraphics[width=\textwidth]{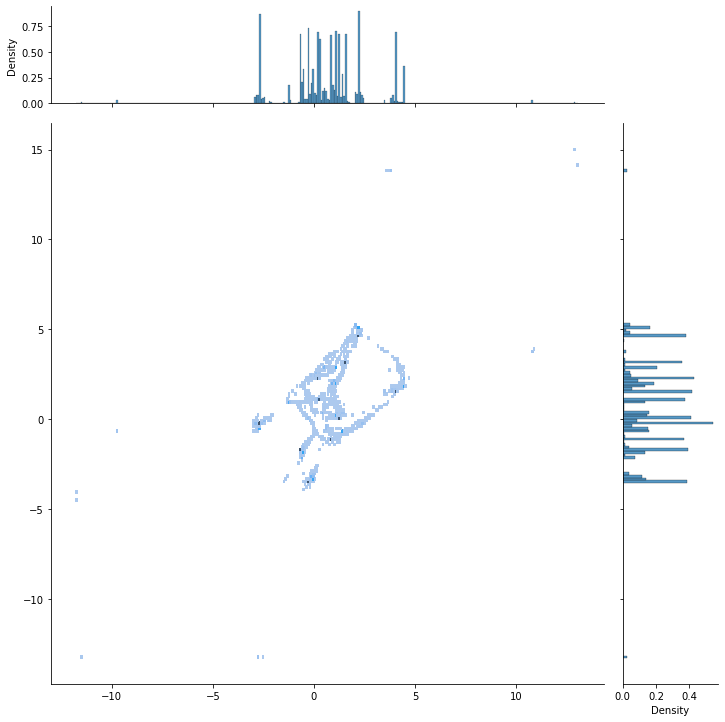}
    \caption{$\mu_2$, plane XY, $M=10$, $N=27$ \label{fig:minimasns2XY1027}}
  \end{subfigure}
  \begin{subfigure}{0.48\textwidth}
    \centering
    \includegraphics[width=\textwidth]{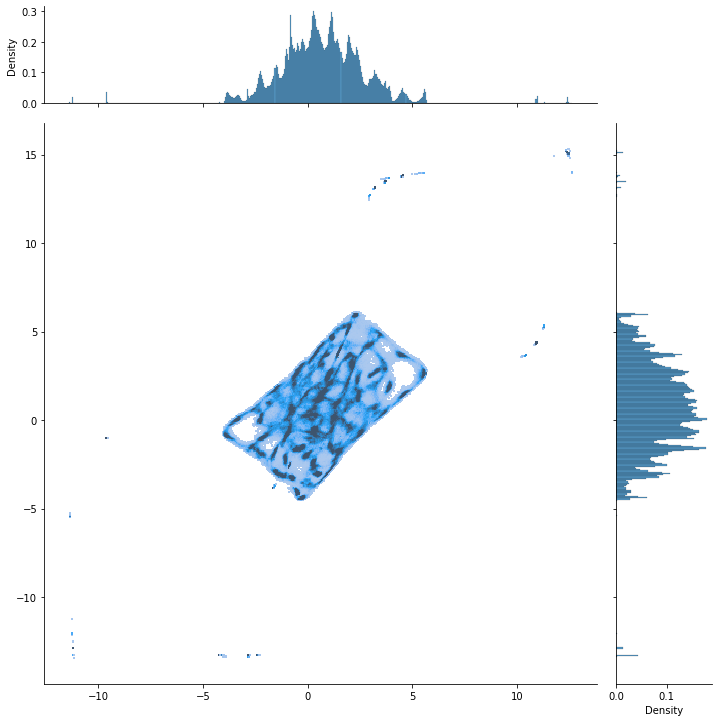}
    \caption{$\mu_2$, plane XY, $M=100$, $N=27$ \label{fig:minimasns2XY10027}}
  \end{subfigure}
  \begin{subfigure}{\textwidth}
    \centering
    \includegraphics[width=0.8\textwidth]{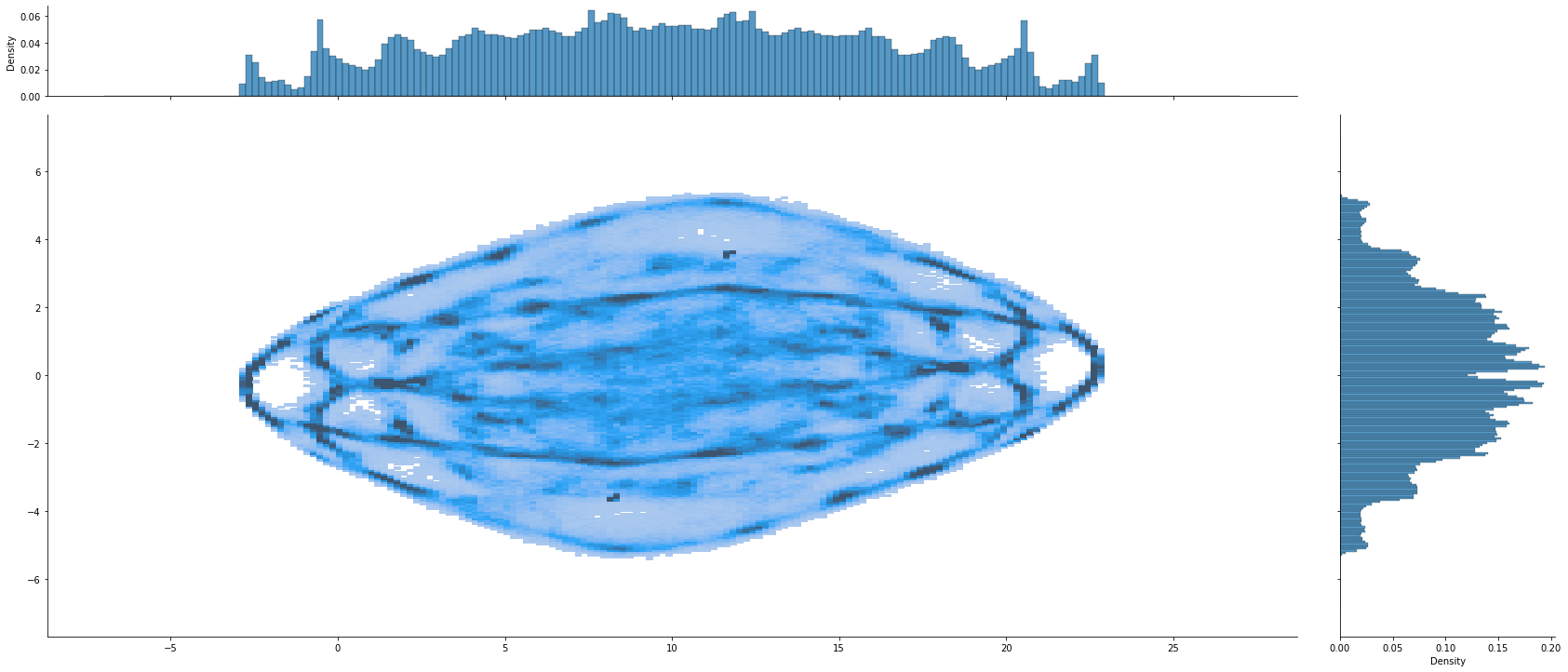}
    \caption{$\mu_3$, plane XY, $M=100$, $N=27$ \label{fig:minimasns3XY10027}}
  \end{subfigure}
  \begin{subfigure}{\textwidth}
    \centering
    \includegraphics[width=0.8\textwidth]{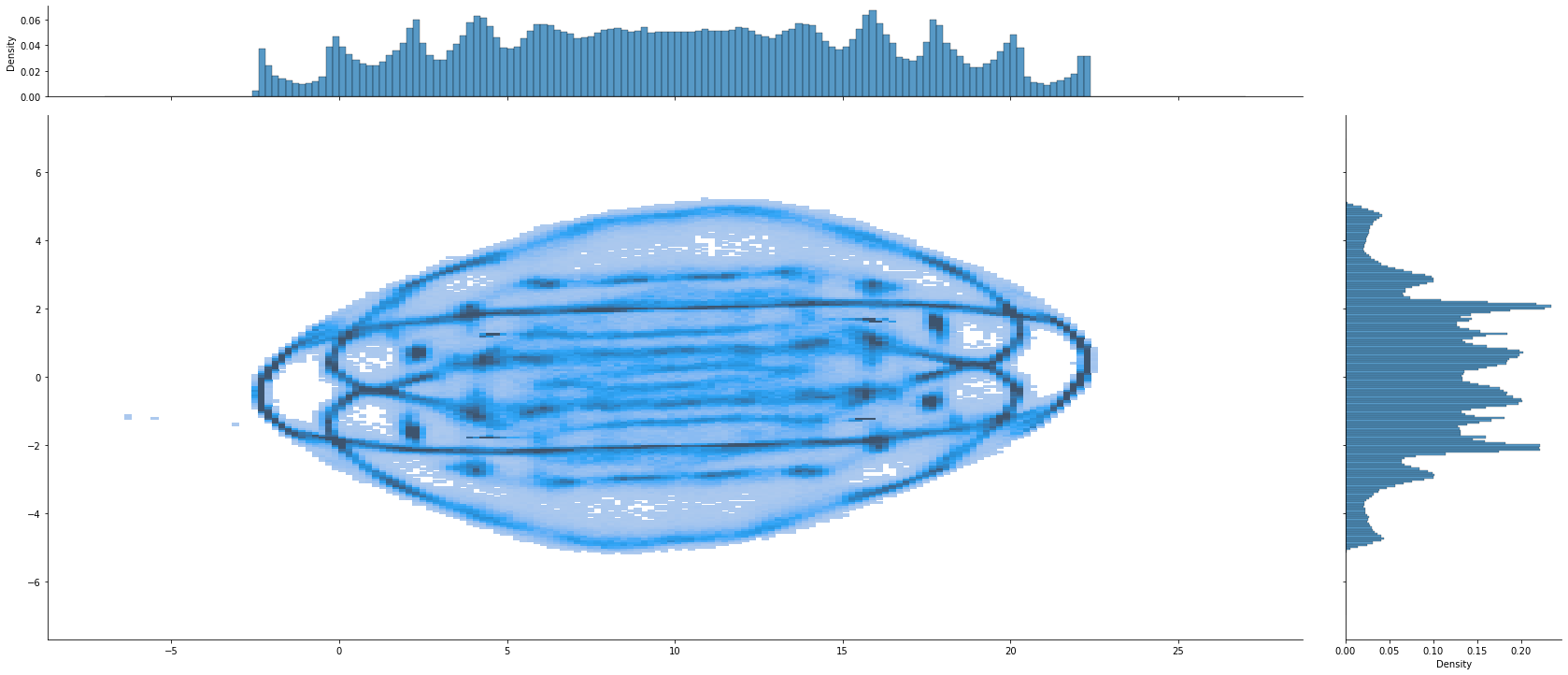}
    \caption{$\mu_3$, plane XY, $M=100$, $N=52$ \label{fig:minimasns3XY10052}}
  \end{subfigure}
  \caption{Optimal transport for $\mu_2$ and $\mu_3$, $M = 10, 100$, $N= 27, 52$, $\beta_0 = 0$, $K=10000$ and $\Delta t_0 = 10^{-4}$
  In each graph, minimizers are represented as {\small $\frac 1 {MK} \sum_{k=1}^{K} \sum_{m = 1}^{M} \delta_{x^k_{m,1},x^k_{m,2}}$}. In order to better distinguish between areas of low and high particles density, plots are represented as 2D histograms. \label{fig:minimasnonsymXY}}
\end{figure}

\subsubsection{Optimization for $\mu_4$ - Figure \ref{fig:jellium}}

Optimal transport for $\mu_4$ with a large number of electrons is of theoretical interest as it might provide approximations for a uniform electronic density in a large space \cite{MR3732693}. Numerical results for its MCOT relaxation with $M=100$ and $N=52$ are presented in Figure \ref{fig:jellium}. Although the cost has been optimized (Figure \ref{fig:jellium:cost}), it is only 3\% lower than the initial uniform sampling (after a Runge-Kutta~3 initialization). Although the 1D marginal laws seem well approximated (Figures \ref{fig:jellium:XY} and \ref{fig:jellium:X}), planar and radial graphs (Figures \ref{fig:jellium:XY} and \ref{fig:jellium:radial}) show that particles are concentrated on two spheres (of radius 0.6 and 1 respectively). Most of the transport takes place inside and between those two spheres.

\begin{figure}[htp]
  \centering
  \begin{subfigure}{0.4\textwidth}
    \centering
    \includegraphics[width=\textwidth]{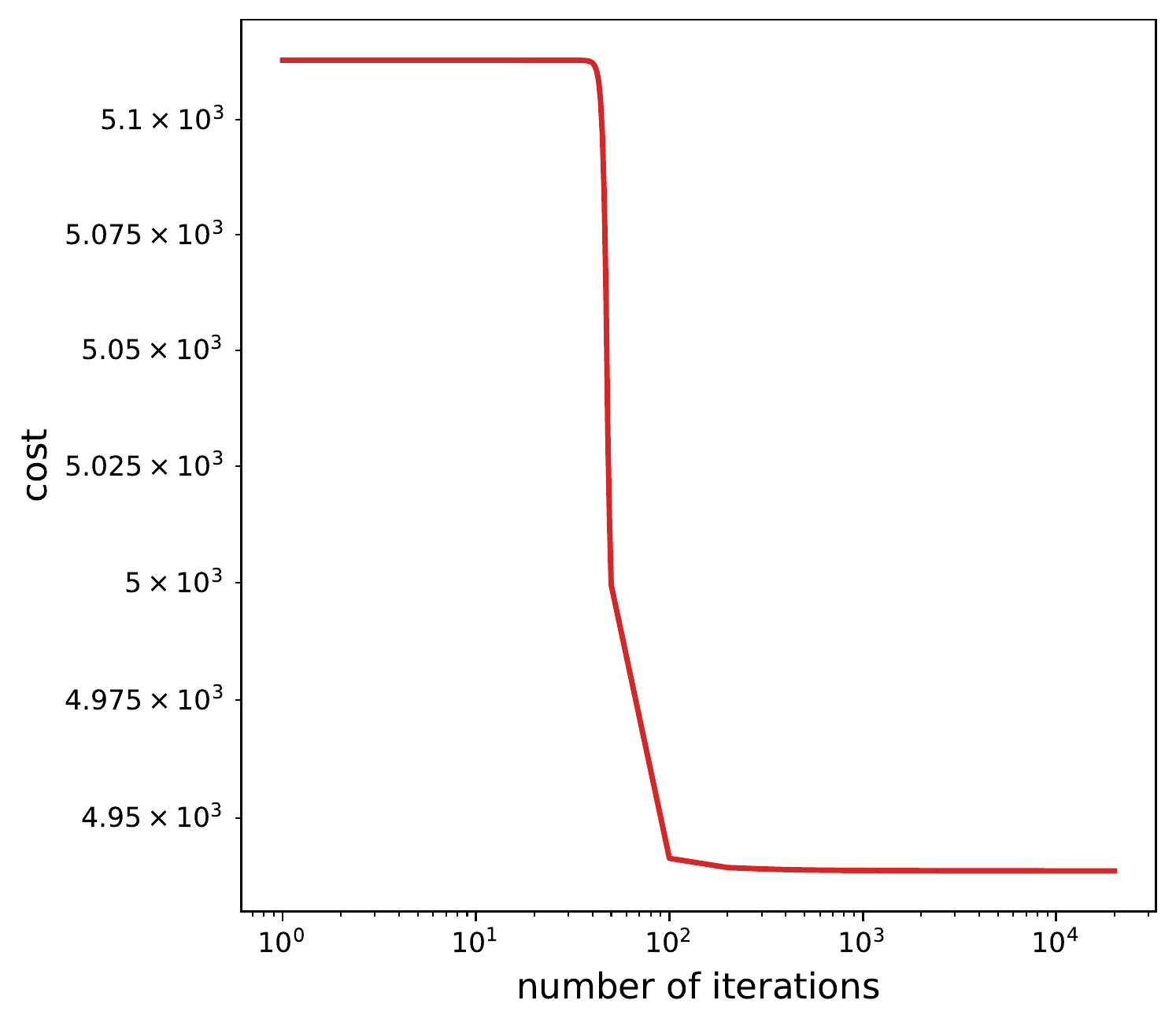}
    \caption{Cost as a function of $n$ \label{fig:jellium:cost}}
  \end{subfigure}
  \begin{subfigure}{0.4\textwidth}
    \centering
    \includegraphics[width=\textwidth]{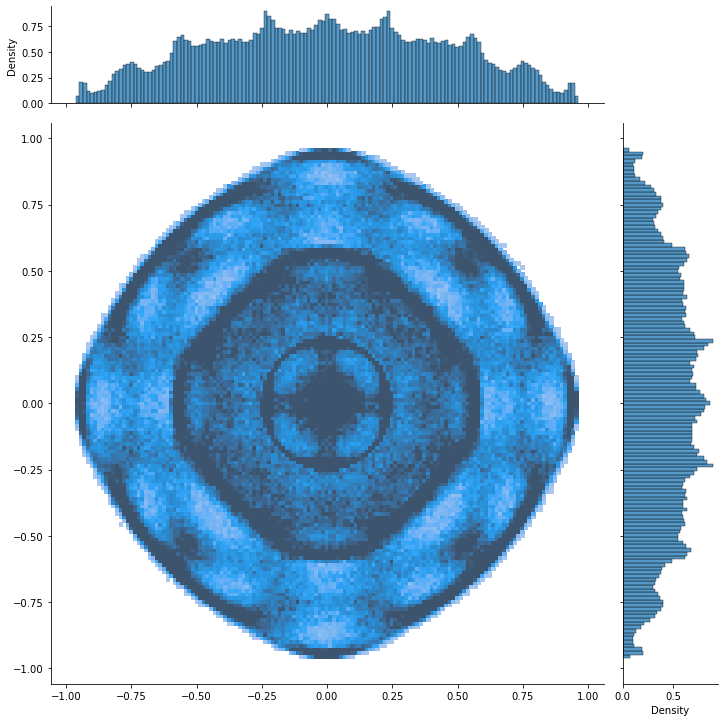}
    \caption{plane XY \label{fig:jellium:XY}}
  \end{subfigure}
  \begin{subfigure}{0.4\textwidth}
    \centering
    \includegraphics[width=\textwidth]{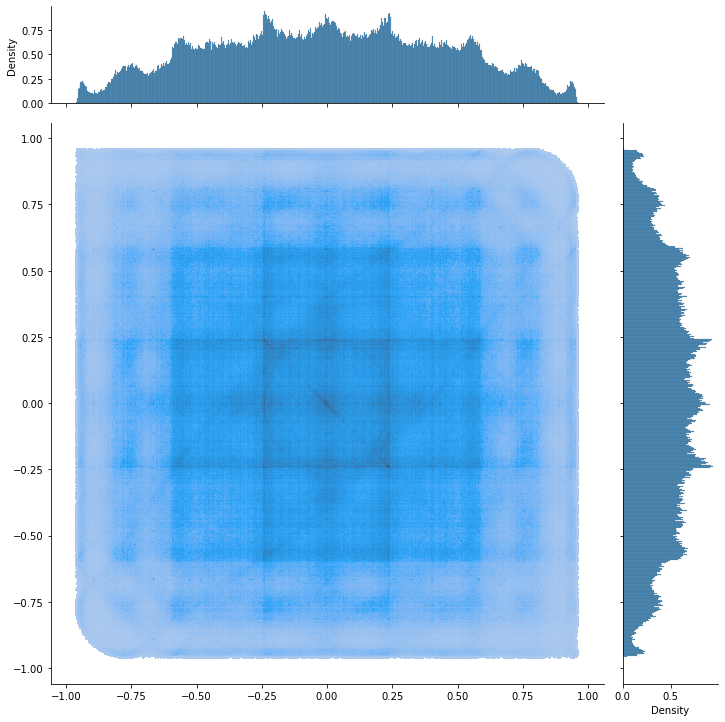}
    \caption{X axis \label{fig:jellium:X}}
  \end{subfigure}
  \begin{subfigure}{0.4\textwidth}
    \centering
    \includegraphics[width=\textwidth]{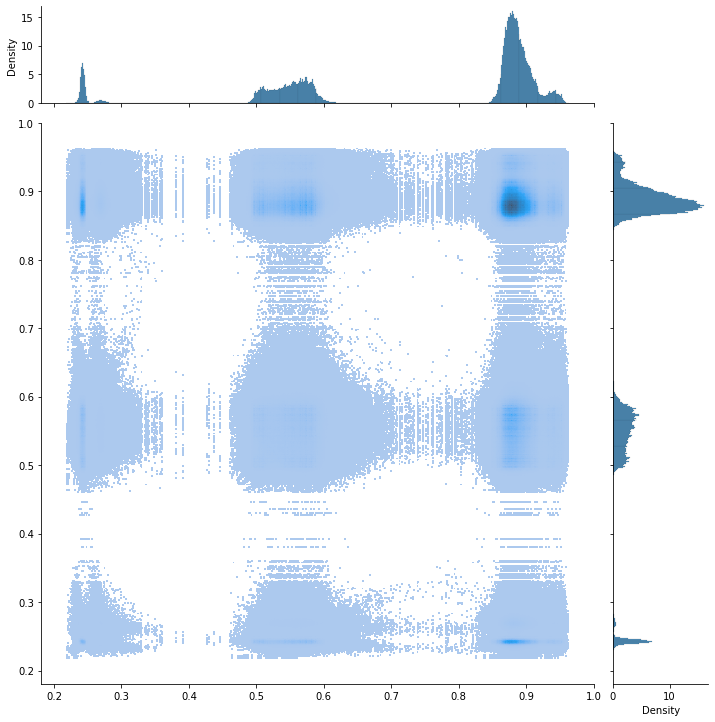}
    \caption{radial \label{fig:jellium:radial}}
  \end{subfigure}
  \caption{
  Evolution of the cost as a function of the number of iteration $n$ (Figure \ref{fig:jellium:cost}) and optimal transport with $\mu_4$, $M=100$, $N=52$, $K=10000$ $\beta_0  = 0$ and $\Delta t_0 = 10^{-4}$.
  In Figure \ref{fig:jellium:XY} is showed {\small $\frac 1 {MK} \sum_{k=1}^{K} \sum_{m = 1}^{M} \delta_{x^k_{m,1},x^k_{m,2}}$}. In Figure \ref{fig:jellium:X} is showed
  {\small $\frac 1 {M(M-1)K} \sum_{k=1}^{K} \sum_{m \neq m' = 1}^{M} \delta_{x^k_{m,1},x^k_{m',1}}$}. In Figure \ref{fig:jellium:radial} is showed {\small $\frac 1 {M(M-1)K} \sum_{k=1}^{K} \sum_{m \neq m' = 1}^{M} \delta_{|x^k_{m}|,|x^k_{m'}|}$}, where {\small $|x^k_{m}| = \sqrt{\sum_{i=1}^3(x^k_{m,i})^2}$}.
  In order to better distinguish between areas of low and high particles density, plots are represented as 2D histograms.\label{fig:jellium}
  }
\end{figure}

%

\clearpage
\section{Proof of Theorem~\ref{lem:monotonecost}}\label{sect:proofs}

The aim of this section is to gather the proofs of our main theoretical results.

\subsection{Tchakaloff's theorem}\label{sec:Tchakaloff}

We present here a corollary of the so-called Tchakaloff theorem which is the backbone of our results concerning the theoretical properties of the MCOT particle problem.
A general version of the Tchakaloff theorem has been proved by Bayer and Teichmann~\cite{bayer2006proof}. Theorem~\ref{cor:Tchakaloff} is an immediate consequence of Tchakaloff's theorem, see Corollary~2 in~\cite{bayer2006proof}.

\begin{theorem}\label{cor:Tchakaloff}
Let $\pi$ be a measure on $\R^d$ concentrated on a Borel set $A \in \cF$, i.e.\ $\pi(\R^d \setminus A) = 0$.
Let $N_0\in \bN^*$ and $\Lambda : \R^d \to \Reel^{N_0}$ a measurable Borel map.
Assume that the first moments of $\Lambda \# \pi$ exist, i.e.\
\[
\int_{\Reel^{N_0}} \Vert u \Vert \dd \Lambda \# \pi(u) =\int_{\R^d} \Vert \Lambda(z) \Vert \dd \pi(z) < \infty,
\]
where $\|\cdot\|$ denotes the Euclidean norm of $\Reel^{N_0}$. Then, there exist an integer $1 \leq K \leq N_0$, points $z_1, ..., z_K \in A$ and weights $p_1, ..., p_K > 0$ such that
\[
\forall 1\leq i \leq N_0, \quad \int_{\R^d} \Lambda_i(z) \dd \pi(z) = \sum_{k = 1}^K p_k \Lambda_i(z_k),
\]
where $\Lambda_i$ denotes the $i$-th component of $\Lambda$.
\end{theorem}
\noindent We recall here that $\Lambda \# \pi$ is the push-forward of $\pi$ through $\Lambda$, and is defined as  $\Lambda \# \pi (A)=\pi(\Lambda^{-1}(A))$ for any Borel set $A\subset \R^{N_0} $.

Last, let us mention that Theorem \ref{cor:Tchakaloff} is a consequence of Caratheodory's theorem \cite[Corollary 17.1.2]{rockafellar1970convex} applied to $\int_{\R^{N_0}} u \dd \Lambda \# \pi (u)$ which lies in the (convex) cone induced by $\mathrm{spt}(\Lambda \# \pi)$, the support of the measure $\Lambda \# \pi$.

\subsection{Proof of Theorem~\ref{lem:monotonecost}}\label{sec:prooftheorem}

We denote here by $\mathcal S_K$ the set of permutations of the set $\{1, \cdots, K\}$.

\begin{lemma}\label{lem_permutation}
Let $(W,Y)\in \mathcal{U}^N_K$ be such that there exists  $k'$ such that $w_{k'}=0$. Then for any permutation $\sigma \in \mathcal{S}_K$, there exists a polygonal map $\psi:[0,1]\to\mathcal{U}^N_K$ such that $\psi(0)=(W,Y)$, $\psi(1)=(W^\sigma,Y^\sigma)$ and $\mathcal{I}(\psi(t))$ is constant, where
$Y^\sigma:=(X^{\sigma(k)})_{1\leq k \leq K} \in ((\mathbb{R}^d)^M)^K$ and $W^\sigma:=(w_{\sigma(k)})_{1\leq k \leq K} \in (\mathbb{R}_+)^K$.
\end{lemma}

\begin{proof}
  For $(W,Y)$ and $(W',Y')$, we will denote $[(W,Y),(W',Y')]$ the segment map $t\in[0,1] \mapsto [(1-t)W+tW',(1-t)Y+tY']$ and we will construct $\psi$ as the concatenation of segments that are clearly in $\mathcal{U}^N_K$ and leaves $\mathcal{I}$ constant.

  It is sufficient to prove this result for transpositions i.e.\ for $\sigma$ such that there exist $i_1<i_2$ such that $\sigma(i_1)=i_2$, $\sigma(i_2)=i_1$ and $\sigma(i)=i$ for $i\not \in \{i_1,i_2\}$. We distinguish two cases.
  \begin{itemize}
  \item  $k'\in\{i_1,i_2\}$, say $k'=i_2$. We then define $Y_1=(X_1^k)_{1\le k\le K}$ by $X_1^{k'}=X^{i_1}$ and $X_1^{k}=X^{k}$ for $k\not=k'$ and consider the segment $[(W,Y),(W,Y_1)]$. We then set $w_1^{k'}=w^{i_1}$, $w_1^{i_1}=0$ and $w_1^k=w^k$ for $k\not \in \{k',i_1\}$ (note that $W_1=W^\sigma$) and consider the segment $[(W,Y_1),(W_1,Y_1)]$. Last, we define $Y_2=(X_2^k)_{1\le k\le K}$ as $X_2^{i_1}=X^{k'}$ and $X_2^{k}=X_1^{k}$ for $k\not = i_1$ (note that $Y_2=Y^\sigma$) and consider the segment  $[(W_1,Y_1),(W_1,Y_2)]$.
  \item $k'\not \in\{i_1,i_2\}$. First, we define $Y_1=(X_1^k)_{1\le k\le K}$ by $X_1^{k'}=X^{i_1}$ and $X_1^{k}=X^{k}$ for $k\not=k'$ and consider the segment $[(W,Y),(W,Y_1)]$. We then set $w_1^{k'}=w^{i_1}$, $w_1^{i_1}=0$ and $w_1^k=w^k$ for $k\not \in \{k',i_1\}$ and consider the segment $[(W,Y_1),(W_1,Y_1)]$. Then, we define $Y_2=(X_2^k)_{1\le k\le K}$ as $X_2^{i_1}=X^{i_2}$ and $X_2^{k}=X_1^{k}$ for $k\not = i_1$, and consider the segment $[(W_1,Y_1),(W_1,Y_2)]$. We the set $w_2^{i_1}=w^{i_2}$, $w_2^{i_2}=0$ and $w_2^k=w_1^k$ for $k\not \in \{i_1,i_2\}$,  and consider the segment    $[(W_1,Y_2),(W_2,Y_2)]$. Now, we define $Y_3=(X_3^k)_{1\le k\le K}$ by $X_3^{i_2}=X^{i_1}$, $X_3^{k}=X_2^{k}$ for $k\not= i_2$ and consider the segment $[(W_2,Y_2),(W_2,Y_3)]$. Then, we define $w_3^{i_2}=w_2^{k'}=w^{i_1}$, $w_3^{k'}=0$ and $w_3^k=w_2^k$ for $k\not \in \{i_2,k'\}$ (note that $W_3=W^\sigma$) and consider the segment $[(W_2,Y_3),(W_3,Y_3)]$. Last, we set $Y_4=(X_4^k)_{1\le k\le K}$ with $X_4^{k'}=X^{k'}$ and $X_4^k=X_3^k$ for $k\not=k'$ (note that $Y_4=Y^\sigma$) and finally consider the segment $[(W_3,Y_3),(W_3,Y_4)]$, which gives the claim.
  \end{itemize}

\end{proof}

\begin{proof}
For $i=0,1$, let $W_i:= (w_{k,i})_{1\leq k \leq K} \in \mathbb{R}_+^K$, $Y_i= (X^k_i)_{1\leq k \leq K} \subset (\mathbb{R}^d)^M$ and $\pi_i:= \sum_{k=1}^K w_{k,i} \delta_{X^k_i} \in \mathcal{P}\left( (\mathbb{R}^d)^M \right)$.
Note that, for $i=0,1$, the support of $\pi_i$ is included in the discrete set $\{ X^k_i, \; 1\leq k \leq K\}$.

For $i=0,1$, using Theorem~\ref{cor:Tchakaloff} with $\pi = \pi_i$ and $\Lambda: (\mathbb{R}^d)^M \to \mathbb{R}^{N+3}$ the map defined such that, for all $X\in (\mathbb{R}^d)^M$,
$$
\Lambda_n(X) = \varphi_n(X), \; \forall 1\leq n \leq N, \quad \Lambda_{N+1}(X) = 1, \quad \Lambda_{N+2}(X) = c(X)\quad \mbox{ and } \quad  \Lambda_{N+3}(X) = \vartheta(X),
$$
it holds that there exists a subset $J^i \subset \{ 1, \cdots , K\}$ such that $K_i:=\# J^i \le N+3$, and weights $(\widetilde{w}_j^i)_{j\in J^i} \subset \mathbb{R}_+$ such that
\begin{align}\label{eq:eq1}
\forall 1 \le n \le N,\, \sum_{j \in J^i} \widetilde{w}^i_j \varphi_n(X^j_i) & = \int_{(\mathbb{R}^d)^M}\varphi_n\,d\pi_i =  \sum_{k=1}^K w_{k,i} \varphi_n(X^k_i) = \mu_n,\\\label{eq:eq2}
\sum_{j \in J^i} \widetilde{w}^i_j & = \int_{(\mathbb{R}^d)^M}\,d\pi_i = \sum_{k=1}^K w_{k,i} = 1,\\\label{eq:eq3}
\sum_{j \in J^i} \widetilde{w}^i_j c(X^j_i) & = \int_{(\mathbb{R}^d)^M}c\,d\pi_i =  \sum_{k=1}^K w_{k,i} c(X^k_i) = \mathcal I(W_i, Y_i),\\\label{eq:eq4}
\sum_{j \in J^i} \widetilde{w}^i_j \vartheta(X^j_i) & = \int_{(\mathbb{R}^d)^M}\vartheta\,d\pi_i =  \sum_{k=1}^K w_{k,i} \vartheta(X^k_i) \leq A.\\\nonumber
\end{align}
Without loss of generality, by using Lemma~\ref{lem_permutation}, we can assume that $J^0 = \llbracket 1, K_0\rrbracket$ where $K_0\leq N+3$ and that $J^1 = \llbracket K -K_1+1, K \rrbracket$ where $K-K_1+1\geq N+4$.

\medskip

We then define $\widetilde{W}_0:=(\widetilde{w}_1^0, \cdots, \widetilde{w}_{K_0}^0, 0, \cdots, 0)\in \mathbb{R}_+^K$ and
$\widetilde{W}_1:=(0, \cdots, 0, \widetilde{w}_{K -K_1+1}^1, \cdots, \widetilde{w}_{K}^1)\in \mathbb{R}_+^K$.
Let us first define the applications
$$
\psi_1: \left[0, \frac{1}{5}\right]\ni t \mapsto \left( W_0 + 5t (\widetilde{W}_0 - W_0), Y_0\right)
$$
and
$$
\psi_5: \left[\frac{4}{5}, 1\right]\ni t \mapsto \left( W_1 + 5(1-t) (\widetilde{W}_1 - W_1), Y_1\right)
$$
so that $\psi_0(0) = (W_0, Y_0)$, $\psi_0(1/5) = (\widetilde{W}_0, Y_0)$,  $\psi_1(1) = (W_1, Y_1)$, $\psi_1(4/5) = (\widetilde{W}_1, Y_1)$. Then, $\psi_0$ and $\psi_1$ are continuous applications and
identities (\ref{eq:eq1})-(\ref{eq:eq2})-(\ref{eq:eq3})-(\ref{eq:eq4}) implies that for all $t\in [0,1/5]$ (respectively all $t \in [4/5, 1]$), $\psi_0(t) \in \mathcal U_N^K$ and $\mathcal I(\psi_0(t)) = \mathcal I(W_0,Y_0)$
(respectively $\psi_1(t)\in \mathcal U_N^K$ and $\mathcal I(\psi_1(t) = \mathcal I(W_1, Y_1)$).

\medskip

We then define $\widetilde{Y}:= \left( X_0^1, \cdots, X_0^{K_0}, 0, \cdots, 0, X_1^{K-K_1+1}, \cdots, X_1^{K}\right) \in ((\mathbb{R}^d)^M)^K$. We then introduce the continuous applications
$$
\psi_2:  \left[\frac{1}{5}, \frac{2}{5}\right]\ni t \mapsto \left( \widetilde{W}_0 , Y_0 + 5(t-1/5) \widetilde{Y} \right)
$$
and
$$
\psi_4:  \left[\frac{3}{5}, \frac{4}{5}\right]\ni t \mapsto \left( \widetilde{W}_1 , Y_1 + 5(4/5 - t) \widetilde{Y} \right).
$$
It thus holds that $\psi_2(1/5) = (\widetilde{W}_0, Y_0)$ and $\psi_2(2/5) = (\widetilde{W}_0, \widetilde{Y})$. Similarly, $\psi_4(4/5) = (\widetilde{W}_1, Y_1)$ and $\psi_4(3/5) = (\widetilde{W}_1, \widetilde{Y})$.
Let us point out here that, by the definition of $\widetilde{Y}$, for any $t\in \left[\frac{1}{5}, \frac{2}{5}\right]$, the $K_0$ first components of $\psi_2(t)$ are equal to $X_0^1, \cdots, X_0^{K_0}$. Thus, since
$\widetilde{W}_0:=(\widetilde{w}_1^0, \cdots, \widetilde{w}_{K_0}^0, 0, \cdots, 0)\in \mathbb{R}_+^K$, this implies that for all $t\in \left[\frac{1}{5}, \frac{2}{5}\right]$, $\psi_2(t)\in \mathcal U_N^K$ and in addition,
$$
\mathcal{I}(\psi_2(t)) = \mathcal{I}(\widetilde{W}_0, Y_0) = \mathcal{I}(\widetilde{W}_0, \widetilde{Y}) = \mathcal{I}(W_0, Y_0).
$$
Similarly, for any $t\in \left[\frac{3}{5}, \frac{4}{5}\right]$, $\psi_4(t)\in \mathcal U_N^K$ and in addition,
$$
\mathcal{I}(\psi_4(t)) = \mathcal{I}(\widetilde{W}_1, Y_1) = \mathcal{I}(\widetilde{W}_1, \widetilde{Y}) = \mathcal{I}(W_1, Y_1).
$$
Notice that in particular, $\mathcal I$ remains constant along the paths in $\mathcal U_N^K$ given by the applications $\psi_1$, $\psi_2$, $\psi_4$ and $\psi_5$.

Last, we introduce the application
$$
\psi_3:  \left[\frac{2}{5}, \frac{3}{5}\right]\ni t \mapsto \left( \widetilde{W}_0  + 5(t-2/5) \widetilde{W}_1, \widetilde{Y} \right)
$$
which is continuous and such that $\psi_3(2/5) = (\widetilde{W}_0, \widetilde{Y})$ and $\psi_3(3/5) = (\widetilde{W}_1, \widetilde{Y})$. Using similar arguments as above, it then holds that for all $t\in[2/5, 3/5]$,
$\psi_3(t)\in \mathcal{P}_N^K$ and
$$
\mathcal I (\psi_3(t)) = \mathcal I\left(\widetilde{W}_0, \widetilde{Y}\right) + 5 (t-2/5) \mathcal I\left(\widetilde{W}_1, \widetilde{Y}\right) = \mathcal I\left(W_0, Y_0\right) + 5 (t-2/5) \mathcal I\left(W_1, Y_1\right).
$$
This implies that $\mathcal I$ monotonically varies along the path given by the application $\psi_3$.

\medskip

We finally consider the application $\psi: [0,1] \to (\mathbb{R}_+)^K \times \left( (\mathbb{R}^d)^M \right)^K$ defined by
$$
\forall t \in [0,1], \quad \psi(t) = \left\{
\begin{array}{ll}
 \psi_1(t) & \mbox{ if } t \in [0, 1/5],\\
  \psi_2(t) & \mbox{ if } t \in [1/5, 2/5],\\
   \psi_3(t) & \mbox{ if } t \in [2/5, 3/5],\\
    \psi_4(t) & \mbox{ if } t \in [3/5, 4/5],\\
     \psi_5(t) & \mbox{ if } t \in [4/5, 1].\\
\end{array}
\right.
$$
Gathering all the results we have obtained so far, it then holds that $\psi$ is continuous, that for all $t\in [0,1]$, $\psi(t)\in \mathcal U_N^K$ and that the application $\mathcal I \circ \psi$ is monotone. Hence the desired result.
  \end{proof}

\section*{Acknowledgements}

The Labex B\'ezout is acknowledged for funding the PhD thesis of Rafa\"el Coyaud. Aur\'elien Alfonsi benefited from the support of the ``Chaire Risques Financiers'', Fondation du Risque.
We are very grateful to Gero Friesecke, Daniela V\"ogler, Tony Leli\`evre, Gabriel Stoltz and Pierre Monmarch\'e for stimulating discussions, as well as Mathieu Lewin for precious comments on the paper. This publication is part of a project that has received funding from the European Research Council (ERC) under the European Union’s
Horizon 2020 Research and Innovation Programme – Grant Agreement $n^{\circ}$ 810367.



%

\bibliographystyle{plain}
\bibliography{bibli}
\vspace{1cm}
{\footnotesize

\begin{tabular}{rl}
A. Alfonsi & \textsc{Universit\'e Paris-Est, CERMICS (ENPC), INRIA,} \\
 &  F-77455 Marne-la-Vall\'ee, France\\
 &  \textsl{E-mail address}:  {\texttt{aurelien.alfonsi@enpc.fr}} \\
 \\
R. Coyaud & \textsc{Universit\'e Paris-Est, CERMICS (ENPC), INRIA,} \\
 &  F-77455 Marne-la-Vall\'ee, France\\
 &  \textsl{E-mail address}:  {\texttt{rafael.coyaud@enpc.fr}} \\
 \\
V. Ehrlacher & \textsc{Universit\'e Paris-Est, CERMICS (ENPC), INRIA,} \\
 &  F-77455 Marne-la-Vall\'ee, France\\
 &  \textsl{E-mail address}:  {\texttt{virginie.ehrlacher@enpc.fr}} \\
  \\
\end{tabular}
}

\end{document}